\documentclass[10pt]{article}
\usepackage{geometry}
\usepackage{comment,hyperref,slashed,tensor}

\usepackage{amsmath}
\usepackage[T1]{fontenc}
\usepackage[utf8]{inputenc}
\usepackage{appendix}

\usepackage{ mathrsfs }
\usepackage{amsmath, amsfonts, mathtools, amsthm, amssymb}

\mathtoolsset{showonlyrefs}
\usepackage{hyperref,  color}

\newtheorem{theorem}{Theorem}[section]
\newtheorem{lemma}[theorem]{Lemma}

\newtheorem{cor}[theorem]{Corollary}

\newtheorem{definition}[theorem]{Definition}

\newtheorem{remark}[theorem]{Remark}

\newcommand{\R}{\mathbb{R}}

\newcommand{\N}{\mathbb{N}}

\newcommand{\PP}{\mathbb{P}}

\newcommand{\p}{\partial}
\newcommand{\boldb}{\mathbf{B}}

\newcommand{\calH}{\mathcal{H}}
\newcommand{\calB}{\mathcal{B}}
\newcommand{\calL}{\mathcal{L}}
\newcommand{\calLloc}{\underline{\mathcal{L}}}
\newcommand{\calN}{\mathcal{N}}
\newcommand{\calM}{\mathcal{M}}
\newcommand{\calS}{\mathcal{S}}
\newcommand{\calA}{\mathcal{A}}
\newcommand{\calE}{\mathcal{E}}
\newcommand{\boldom}{\Omega}
\newcommand{\boldP}{P}
\newcommand{\scrL}{\mathscr{L}}

\newcommand{\calW}{\mathcal{W}}
\newcommand{\ul}{\underline{L}}
\newcommand{\ellud}{L_{12}} 
\newcommand{\lna}{\ul_{12}}      
\newcommand{\ellg}{L}         
\newcommand{\lgl}{\ul}        
\newcommand{\scrtemp}{\mathbb{L}}
\newcommand{\lfar}{\overline{L}_{12}}
\newcommand{\scrnonloc}{\overline{\mathscr{L}}}
\newcommand{\scrl}{\mathscr{L}}
\newcommand{\scrloc}{\underline{\mathscr{L}}}
\newcommand{\dbold}{\mathbf{D}}
\newcommand{\calK}{\mathcal{K}}
\newcommand{\boldpsi}{\Psi}
\newcommand{\tscrl}{\tilde{\scrl}}
\newcommand{\tscrloc}{\underline{\tilde{\mathscr{L}}}}

\newcommand{\scrn}{\mathscr{N}}
\newcommand{\scrH}{\mathscr{H}}

\def\Xint#1{\mathchoice
{\XXint\displaystyle\textstyle{#1}}%
{\XXint\textstyle\scriptstyle{#1}}%
{\XXint\scriptstyle\scriptscriptstyle{#1}}%
{\XXint\scriptscriptstyle\scriptscriptstyle{#1}}%
\!\int}
\def\XXint#1#2#3{{\setbox0=\hbox{$#1{#2#3}{\int}$ }
\vcenter{\hbox{$#2#3$ }}\kern-.6\wd0}}

\def\dashint{\Xint-}

\allowdisplaybreaks

\begin{document}

\title{From Instability to Singularity Formation in Incompressible Fluids}
\author{Tarek M. Elgindi\thanks{Department of Mathematics, Duke University. E-mail: tarek.elgindi@duke.edu} \space and Federico Pasqualotto\thanks{Department of Mathematics, UC Berkeley. E-mail: fpasqualotto@berkeley.edu}}

\date{\today}
\maketitle

\begin{abstract}
We establish finite-time singularity formation for $C^{1,\alpha}$ solutions to the Boussinesq system that are compactly supported on $\mathbb{R}^2$ and infinitely smooth except in the radial direction at the origin. The solutions are smooth in the angular variable at the blow-up point, which was a fundamental obstruction in previous works. This is done by exploiting a second-order effect, related to the classical Rayleigh--B\'enard instability, that overcomes the regularizing effect of transport. A similar result is established for the 3d Euler system based on the Taylor--Couette instability.
\end{abstract}

\tableofcontents

\section{Introduction}

Understanding the formation and propagation of singularities is a major problem in the study of partial differential equations. An especially difficult problem is the formation of singularities in incompressible fluids. In that setting, since the fluid velocity is divergence free, we focus our study on the dynamics of the vorticity, the curl of the velocity field. The first major difficulty in understanding the dynamics of vorticity is that one must have control on the full solution in all of space in order to determine the motion of even a single parcel of fluid. This is the problem of \emph{nonlocality}. A second major difficulty is the presence of an intrinsic \emph{local} regularization mechanism in the equations themselves. Indeed, because vorticity is transported and stretched by the same velocity field, particles that experience stretching are rapidly {ejected} away from regions of intense stretching. This is the problem of \emph{regularization by transport.} Over the past few years, a number of works have been devoted to trying to circumvent both non-locality and regularization by transport to construct singularities in the Euler equation and related models. Examples include the works \cite{EJB,EJ3dE,E_Classical,ChenHou1,ChenHou2}, where either non-smoothness or solid boundaries in the expanding direction of the flow are used to circumvent regularization by transport.

The main purpose of this work is to advance our understanding of the regularization by transport. We introduce and study a blow-up scenario that overcomes, rather than circumvents, this regularization mechanism. Indeed, in our setting we show that the first order growth mechanism present in the equations is completely thwarted by particle transport (so that there is no singularity, to first order). We show, however, that a second order effect still leads to singularity formation even when the solution is smooth in the direction of the flow. It seems that this secondary effect is truly special to the structure of two or three dimensional fluid flows. It does not appear in scalar one-dimensional problems like the De Gregorio equation \cite{CLM,DG1}, for example. Following our analysis and the conceptual clarity of the blow-up mechanism, it is natural to conjecture that truly $C^\infty$ solutions in the same geometric setting become singular in finite time in the Boussinesq system and the 3d Euler system, though we do not yet know how to establish this rigorously.       

\subsection{The Boussinesq System}
The Boussinesq system models the motion of a fluid with density or temperature variation, taking buoyancy into account:
\begin{equation}\label{B1} \partial_t u+u\cdot\nabla u + \nabla p =\rho e_1,\end{equation}
\begin{equation}\label{B2}\text{div}(u)=0,\end{equation}
\begin{equation}\label{B3}\partial_t\rho +u\cdot\nabla \rho=0.\end{equation} Here, $u:\mathbb{R}^2\times [0,T)\rightarrow\mathbb{R}^2$ is the velocity field of the fluid and $\rho:\mathbb{R}^2\times [0,T)\rightarrow\mathbb{R}$ represents its temperature profile. The velocity field is assumed to be divergence free and the term $\nabla p$ ensures the preservation of this property. 

It is well-known that \eqref{B1}-\eqref{B3} is locally well-posed in the natural spaces. In particular, since the non-linearities are essentially of transport type, the control of $|\nabla u|_{L^\infty}$ is crucial for the propagation of regularity. Indeed, from the Cauchy-Lipschitz theory, we know that this quantity is what ensures the existence, uniqueness, and smoothness of particle trajectories. It is not difficult thus to establish local existence and uniqueness in H\"older spaces like $C^{1,\alpha}$ or Sobolev spaces like $H^{2+\delta},$ when $0<\alpha<1$ or $\delta>0$. The time of existence of these solutions naturally depends on the size of $|\nabla u|_{L^\infty}$ and it is thus not difficult to see that classical solutions existing on an interval $[0,T_*)$ can be shown to exist past $T_*<\infty$ if and only if
\[\lim_{t\rightarrow T_*}\int_0^t |\nabla u|_{L^\infty}<\infty.\] For classical solutions on $\mathbb{R}^2,$ the problem of singularity formation is wide open, though progress has been made on domains with boundary \cite{EJB, ChenHou1,ChenHou2}. The purpose of this work is to give a new scenario in which we can establish singularity formation without a boundary.

\subsection{Statements of the main theorems}
We proceed to state our main results. 
\begin{theorem}\label{thm:boussinesq}
    There exists $\alpha_*>0$ and finite-energy initial data $u_0,\rho_0\in C^{1,\alpha_*}(\mathbb{R}^2)\cap C^\infty(\mathbb{R}^2\setminus\{0\})$ for which the local solution to \eqref{B1}-\eqref{B3} becomes singular in finite time. Additionally, the initial vorticity and temperature gradient are smooth functions of the variables $(r^{\alpha_*},\theta)$.
\end{theorem}
The proof proceeds as in \cite{E_Classical} by searching for solutions depending on the variables $(r^\alpha,\theta)$ by performing an expansion in $\alpha>0$. A key difference with \cite{E_Classical} and \cite{EGM3dE} is that we will construct solutions that are \emph{smooth in these variables}, while smoothness in $\theta$ is impossible in the setting of \cite{E_Classical, EGM3dE}. That we are able to do this here is due to our inclusion in the $\alpha$ expansion of a destabilizing profile that is stationary to leading order. This is explained in more detail in Section \ref{NewIdeas}. A remark on Theorem~\ref{thm:boussinesq} is in order.
\begin{remark}[On the regularity of the solution and belonging to $C^\infty(\mathbb{R}^2\setminus\{0\})$]
 The key difference between the solution constructed here and the one of \cite{E_Classical} is the regularity in $\theta$ near the blow-up point. Whether this was possible without boundaries was the paramount question after \cite{E_Classical}. The mere statement that the solution is smooth on $\mathbb{R}^2\setminus\{0\}$ is inconsequential regarding this point, as a solution may be $C^\infty$ except at a point while exhibiting behaviors that do not conform with scaling (as exhibited\footnote{This point was also independently observed by the first named author and by J. Chen \cite{ChenRemark} following \cite{CMZZ}.} in \cite{CMZZ}). One way to quantify this is by looking at the regularity of the solution in other spaces with the same scaling as $C^{1,\alpha}$. A direct consequence of the angular regularity obtained here is that we can have that $u_0,\rho_0\in H^{2+\delta},$ for $\delta$ small. Similarly, the solution constructed here can belong to \emph{any} Sobolev space of the same scaling, no matter how high the differentiability index is. In the case of axi-symmetric no-swirl solutions, such smoothness on the data would imply that the initial velocity belongs to $W^{2,p}$ for some $p>3,$ which is well-known to rule-out a blow-up in that setting \cite{Danchin}, even though such spaces scale exactly like $C^{1,\alpha}$. 
\end{remark}

It has been established rigorously in numerous works starting with \cite{EJ3dE} and then \cite{ChenHou1,ChenHou2} that the problem of singularity formation in the Boussinesq system is directly related to the corresponding problem for 3d axisymmetric Euler equation. In fact, certain self-similar singularities for the Boussinesq system \emph{imply} singularity formation in the Euler equation. It is tempting to try to do a similar analysis at the axis in the axisymmetric Euler equation. Unfortunately, since the square of the swirl necessarily increases as we leave the axis, the underlying instability mechanism does not seem to be of help at the axis (since it requires a drop in the swirl, as in the Taylor-Couette setup).  We thus obtain the following statement for the 3d Euler equations in axisymmetry \emph{with swirl} and away from the symmetry axis.

\begin{theorem}\label{thm:euler}
    Let $\mathcal{C}_0$ be a circle centered at the origin, contained in the $x_1x_2$ plane. There exists $\alpha_* > 0$ and finite-energy, axisymmetric (with respect to the $x_3$-axis) initial data $u_0 \in C^{1,\alpha_*}(\R^3) \cap C^\infty(\R^3\setminus \mathcal{C}_0)$ for which the local solution to the 3d Euler system:
    \begin{align}
        &\p_t u + u \cdot \nabla u + \nabla p = 0,\\
        &\text{div}(u) = 0,\\
        &u|_{t=0} = u_0
    \end{align}
    becomes singular in finite time. Moreover, the vorticity remains smooth prior to the blow-up time as a function of $(r^{\alpha_*}, \theta, \varphi)$, where $(r, \theta, \varphi)$ are the toroidal coordinates adapted to a circle $\mathcal{C}(t)$ which varies continuously in time, such that $\mathcal{C}(0) = \mathcal{C}_0$.
\end{theorem}

\begin{remark}
Note that, in this construction, the swirl component is what drives the singularity. At zeroth order, the effect of the swirl is a manifestation of the classical Taylor--Couette instability.
\end{remark}

\subsection{Remarks on Previous Works and Some Ideas}
We will now describe the previous works in this direction. The key is to understand the behavior of the solutions near the blow-up point, which we take to be the origin. In a nutshell, what we will say is that previous works on singularity formation in the Boussinesq and Euler systems have depended on either irregularity or a solid boundary in the angular direction $\theta$ at the point of blow-up. Here, we are offering a new mechanism that does not rely on imposed or boundary-induced irregularity.  

\subsubsection{Previous works on singularity formation in the Boussinesq and Euler systems}

Singularity formation in the Boussinesq and 3d Euler systems have been studied extensively for many years. We refer interested readers to the survey works \cite{ConstantinReview, BardosTitiReview, Gibbon2008, DrivasElgindi} for detailed accounts of early and recent works.  
Local existence and uniqueness for the Boussinesq and 3d Euler systems was established in a variety of function spaces \cite{ChaeNam97}. The key quantity to control in the local well-posedness theory is $|\nabla u|_{L^\infty_x}$, from which higher regularity generally follows.  
The first singularity result for strong and finite-energy solutions that we are aware of was given in the work \cite{EJB} on some corner domains. The follow-up work \cite{EJ3dE} gave a general strategy to pass from a singularity in the Boussinesq system to the Euler system. Inspired by the very influential numerical work of Luo and Hou \cite{Luo2014}, and building off of \cite{E_Classical}, it was shown in \cite{ChenHou1} that finite-energy $C^{1,\alpha}$ solutions on $\mathbb{R}^2_+$ could develop a singularity in finite time. Further numerical confirmation was provided for the scenario of \cite{Luo2014} in the recent work \cite{JaviTristan}. Very recently, Chen and Hou \cite{ChenHou2} proposed an interesting and highly non-trivial computer-assisted proof of singularity formation for \emph{smooth} solutions on $\mathbb{R}^2_+$ confirming \cite{Luo2014}. Another very recent work \cite{KiselevYaoPark} analytically established unbounded growth of the vorticity and the temperature gradient on $\mathbb{R}^2_+$ in the same setting as \cite{Luo2014}. Singularity formation in simplified models and for infinite energy solutions of the Boussinesq system on has also been established in \cite{CCW, ChoiEtAl,KiselevYaoPark}, the latter two being based on the important result on 2d Euler \cite{KS}. Very recently, C\'ordoba, Mart\'inez-Zoroa, and Zheng \cite{CMZZ} and C\'ordoba and Mart\'inez-Zoroa \cite{CMZ} introduced a new and seemingly flexible approach to singularity formation based on taking advantage of the non-locality of the Biot-Savart law to build an infinite chain of ODE's that cause $\nabla u$ to accumulate at a point and become singular in finite time (with non-smooth data, as of now). Let us finally remark that there is a wide literature on numerical simulations of the Boussinesq and Euler equations with varying outcomes (see the previously cited surveys, for example). Two works that appear related to ours are the work of Pumir and Siggia \cite{Pumir1992} and the very recent and interesting simulation of Hou \cite{Hou2022}. The example of Pumir and Siggia appears to have the opposite geometry to ours, with the formation of two ``plumes'' going in opposite directions; in our geometry, there are two plumes that form and collide at the origin. The simulation of Hou \cite{Hou2022} seems to have a similar geometry as our Euler singularity, though the singularity seems to form at the symmetry axis and it is not clear whether the mechanism is the same.

\subsubsection{Irregularity in \texorpdfstring{$r$}{r}: Expansion of the nonlocal operators}
A key idea introduced in \cite{E_Classical} to simplify the non-local effects is to consider data that is highly concentrated around the point of blow-up, which we take to be the origin. One way to quantify this is to consider solutions that are smooth in the variable $R=r^\alpha,$ for some $\alpha>0$ small. In this way, an expansion for the inverse of the Laplacian $\Delta^{-1}$ was given. We rely on this idea in the present work as well. This only served to \emph{simplify} the non-local effects in the axisymmetric 3d Euler equation. In the present setting, it can be shown that 
\begin{equation}\label{HeuristicExpansion}
\Delta^{-1} f = -\frac{r^2}{4\alpha} \int_{R}^\infty\mathbb{P}_{2}f(\sigma,\theta)\frac{d\sigma}{\sigma}+O(1),
\end{equation}
as $\alpha\rightarrow 0$ whenever $f$ is a smooth function of $R$ and where $\mathbb{P}_2$ is the projection to the second Fourier shell in $\theta$. 
As we have already remarked, smooth solutions in the variable $R$ that are axisymmetric without swirl are necessarily globally regular \emph{for any} $0<\alpha\leq 1$. This is consistent with the well-known global regularity for smooth solutions to that system. This indicates that, while the transformation $r\rightarrow r^\alpha$ does introduce radial irregularity, its purpose is to give us a small parameter in the problem and it does not give singularity formation in the $C^{1,\alpha}$ setting where there is none in the smooth setting. This takes us to another type of irregularity, which was the main culprit in previous works. 

\subsubsection{Irregularity in \texorpdfstring{$\theta$}{theta}: Advection and Vortex Stretching}

As we briefly described earlier, one of the main conceptual roadblocks toward the formation of singularities in incompressible fluids is regularity by transport. Indeed, consider the following simple one-dimensional example:
\[\partial_t \omega+u\partial_\theta\omega=\omega\partial_\theta u.\] The key is that pointwise growth of $\omega$ occurs only in regions where $\partial_\theta u>0$ (in higher dimensions, this means that growth occurs along unstable directions of $\nabla u$). On the other hand, those are \emph{expanding} directions of the velocity field. This means that particles do not spend much time in such regions and not much growth can occur. Even if some particles do spend a good amount of time in such regions, the region will generally be stretched and the non-local relation between $\omega$ and $u$ will lead to damping of $u.$  

There have been two ways to get around this problem. The first way is to ``pack'' a large amount of vorticity in a region where $\partial_x u$ is positive and hope that the vorticity grows faster than it is pulled away (this ``packing'' is most clearly seen in \cite{EJDG} and the recent construction \cite{CMZZ}). A second way is to introduce a solid boundary in the problem and take advantage of the fact that vorticity \emph{sticks} to the boundary. In the first case, there is \emph{weak regularization by transport} and in the latter case there is \emph{no regularization by transport}. Both of these ideas can be understood by analyzing the equation for time-dependent $\omega:[0,\infty)\times \mathbb{S}^1\rightarrow\mathbb{R}$ solving
\[\partial_t \omega + L(\omega)\partial_\theta\omega=L(\omega) \omega,\]
\[L(\omega)(r,\theta)=\int_r^\infty \mathbb{P}_2(\omega)(\sigma,\theta)\frac{d\sigma}{\sigma},\] in line with \eqref{HeuristicExpansion}. It is not difficult to prove that solutions that are smooth in $\theta$ retain their smoothness for all time (see \cite{DrivasElgindi}, for example). At the same time, solutions that are only $C^\beta$ in $\theta$ or smooth solutions that are suitably restricted to the half-space $[0,\infty)\times [0,\pi]$ can develop a singularity in finite time.

\subsection{Broad description of the ideas}\label{NewIdeas}
As a summary of what is to be understood from the above remarks: there are two settings where it is quite conceivable that one can establish existence of a self-similar singularity. First, for solutions in free-space that are smooth except in a direction corresponding to the \emph{expanding} direction of the flow or, second, for solutions on a domain with a smooth boundary when there is nontrivial vorticity on the boundary and particles are being ejected \emph{away} from the boundary\footnote{Of course, the velocity field must be tangent to the boundary.}.

We will now move to describe the setting of this paper, which gives a mechanism as to how a singularity could possibly form in free-space without relying on any ``directional'' irregularity. 

\subsubsection{A first look at the leading order model in \texorpdfstring{$\alpha$}{alpha}}
If we act with $\nabla^\perp=(\partial_y,-\partial_x)$ onto \eqref{B1}, we get the Boussinesq system in vorticity form:
\[\partial_t\omega + u\cdot\nabla\omega=\partial_y \rho,\]
\[\partial_t \rho +u\cdot\nabla \rho =0,\] where we note that since $\omega=\nabla^\perp\cdot u$ and $\text{div}(u)=0,$ we get
\[u=\nabla^\perp \Delta^{-1}\omega.\] 
If we make the change of variables $r\rightarrow r^\alpha:=R$, impose that $\omega$ is odd in $x$ and $y,$ and use expansion \eqref{HeuristicExpansion}, we see that the following first order expansion
\[\omega=\alpha \Omega +O(\alpha^2),\qquad \frac{1}{r^2}\Delta^{-1}\omega=-\frac{1}{4}\sin(2\theta)L_{12}(\Omega)+O(\alpha)\]
\[\frac{1}{r}\rho=\alpha  P +O(\alpha^2), \] 
gives:
\begin{equation}\label{FirstOrder1} \partial_t \Omega+\frac{1}{2}L_{12}(\Omega)\sin(2\theta)\partial_\theta\Omega = \sin \theta P + \cos \theta \p_\theta P,\end{equation}
\begin{equation}\label{FirstOrder2}\partial_t P +\frac{1}{2}L_{12}(\Omega)\sin(2\theta)\partial_\theta P =  \frac 12 \cos (2\theta) L_{12}(\Omega) P,\end{equation}
where $\Omega$ is odd with respect to $0$ and $\pi/2$ and $P$ is even with respect to $0$ and $\pi/2.$ It is not difficult to show that smooth solutions of \eqref{FirstOrder1}-\eqref{FirstOrder2} are globally regular on $[0,\infty)\times\mathbb{S}^1$ due to the presence of the transport terms.

Note moreover that there is finite time singularity in the system~\eqref{FirstOrder1}--\eqref{FirstOrder2} when the problem is posed on the upper half plane or when $P$ and $\Omega$ have limited regularity in $\theta.$ Let us observe, however, that in \eqref{FirstOrder1}, the expansion is consistent with having a higher order term in $P$ for which $\partial_\theta \Big(\frac{P}{\cos(\theta)}\Big)\equiv 0$. In fact, upon some reflection, we realize that the dynamics of \eqref{FirstOrder1}-\eqref{FirstOrder2} pushes $P$ to become a multiple of $\cos \theta.$ A linear analysis shows that $\omega \equiv 0, P\equiv\cos \theta$ is itself steady and \emph{unstable} (it corresponds to the density being $\rho=r\cos(\theta)=x,$ which is an unstably stratified solution to the Boussinesq system in our convention). 

\subsubsection{Adding a background term and the Rayleigh-B\'enard instability}
Given the above discussion, it behooves us to modify our expansion. More precisely, we will expand
\begin{equation}\label{Expansion}\omega=\alpha \Omega +O(\alpha^2),\qquad \frac{1}{r^2}\Delta^{-1}\omega=-\frac{1}{4}\sin(2\theta)L_{12}(\Omega)+\alpha\Psi_1+O(\alpha^2)\end{equation}
\[\frac{1}{r}\rho=\cos(\theta)P_0+\alpha  P_1 +O(\alpha^2).\] Notice that we have included the effects of two extra terms: $P_0$, which will be necessarily \emph{purely radial} and $\Psi_1$, which is the next order expansion in the Biot-Savart law. Observe that $P_0$ provides a ``large scale'' force that is actually \emph{stationary} to leading order if all other terms are zero. Indeed, to leading order, it is nothing but the unstable steady profile $u_*\equiv 0$ and $\rho_*\equiv x,$ which is the culprit behind the classical Rayleigh-B\'enard instability.

Upon doing a few computations and simplifications using this asymptotic expansion, we get a system involving $P_0, P_1,$ and $\Omega.$ When $P_0\equiv 0,$ the system collapses to \eqref{FirstOrder1}-\eqref{FirstOrder2}. When $P_0\not\equiv 0,$ if it is large and has the right sign (the sign that leads to instability in the Rayleigh-B\'enard problem), it produces an instability that overcomes the regularization by transport in \eqref{FirstOrder1}-\eqref{FirstOrder2}. It turns out that the construction of a self-similar profile to the new ``fundamental model'' can be recast as a generalized eigenvalue problem for a relatively simple operator on $\mathbb{S}^1$--the eigenvalue being the relative size of $\nabla u$ and $P_0.$
Indeed, the problem reduces to finding a pair of time-independent functions $\Gamma_1$ and $\Gamma_2$ that are suitably smooth on $\mathbb{S}^1$ and even with respect to $0$ and $\pi/2,$ as well as a constant $A_*$ solving
\begin{equation}\label{Gamma1}
\partial_t\Gamma_1+2\sin(2\theta)\partial_\theta\Gamma_1+6\Gamma_1=\Gamma_2,
\end{equation}
\begin{equation}\label{Gamma2}
\partial_t\Gamma_2+2\sin(2\theta)\partial_\theta\Gamma_2+9\Gamma_2=A_*\cos^2(\theta)(\Gamma_1-\dashint\Gamma_1)
\end{equation}
and such that $\dashint \Gamma_1\not=0.$ The size of $A_*$ encodes how large the unstable ``background'' profile $P_0$ is. Note that this is essentially \eqref{FirstOrder1}-\eqref{FirstOrder2} with the added effect of $P_0.$ A key part of our analysis relies on a serendipitous fact that \eqref{Gamma1}-\eqref{Gamma2} satisfies a certain positivity principle along with an invariant precisely when $A_*$ takes particular discrete values. In particular, negativity of $\bar{\Gamma}_i:=\Gamma_i-\Gamma_i|_{\theta=0}$ is preserved and the quantity 
\[\mathcal{I}(t):=\int_0^{\pi/2} \Big(13\bar \Gamma_1+\bar\Gamma_2\Big)\cot^2(\theta) d\theta \] satisfies \[\mathcal{I}'(t)=0\iff A_*=130.\] This value of $A_*$ corresponds to the ``ground state''. In particular, we can show that when $A_*=130,$ any sufficiently smooth solution to \eqref{Gamma1}-\eqref{Gamma2} converges to zero if $\mathcal{I}=0$ and converges to a multiple of the unique ``ground state'' $\Gamma_*$ when $\mathcal{I}\not=0.$

Let us remark that \eqref{Gamma1}-\eqref{Gamma2} essentially encodes the local angular dynamics around the origin. As part of our analysis, we give an almost full description of solutions to \eqref{Gamma1}-\eqref{Gamma2}, where we prove asymptotic stability of the ``ground state.'' This is a key step in constructing the full profile.

\subsubsection{From an approximate solution to a real solution}

Once we have constructed a self-similar solution to the leading order model and have made sure that the residual terms in the actual Boussinesq system are small in the parameter $\alpha,$ we see that we have thus constructed an \emph{approximate} self-similar blow-up for the Boussinesq system. We then desire to deduce the existence of a true self-similar blow-up solution to the Boussinesq system close to the approximate one. This was achieved in \cite{E_Classical} by proving coercivity of the linearization of the leading order model equation around the approximate self-similar solution. In the setting of the present work, the coercivity appears to be quite subtle and we were not able to prove it. It is possible that it is not even true, which would mean that the singularity we constructed is actually dynamically unstable. Despite this, we introduce a different argument based on \emph{invertibility} of the linear operator. This is the natural criterion, according to the implicit function theorem. However, since we are dealing with unbounded operators, some care needs to be taken. The basic argument is as follows.

Suppose we wish to solve:
\[L(f)=N(f)+\epsilon,\] where $\epsilon$ is sufficiently small in a sense to be made precise with $L$ and $N$ linear and non-linear operators, respectively. If $\mathcal{L}^{-1}N$ were a bounded non-linear operator, we could essentially apply the Banach-Fixed point theorem to conclude the existence and uniqueness of a small solution. When $L^{-1}N$ is an unbounded operator, the argument clearly breaks down and stronger assumptions must be imposed. Let us suppose that $L$ is invertible and that we can split $L$ into two parts:
\[L=L_0+L_1,\] where $L_0$ is coercive and $L_1$ is smoothing. In this case, we will be able to prove existence of a solution if energy estimates can be performed on $N.$ Indeed, we first write $f=F+g$ so that we are looking for a solution to:
\[L(F+g)=N(F+g)+\epsilon\]
and write:
\[L_0(F) + L(g) +L_1(F)=N(F+g)+\epsilon.\] Then, we \emph{define} $g=-L^{-1}(L_1(F))$ and we get that we seek a solution to
\[L_0(F)=N(F-L^{-1}L_1(F))+\epsilon.\]
Solutions to this equation can be constructed as the long-time limit of 
\[\partial_\tau F+L_0(F)=N(F-L^{-1}L_1(F))+\epsilon,\] assuming that $N$ satisfies an energy estimate in the space where $L_0$ is coercive and that $\epsilon$ is sufficiently small in that space. An argument of similar flavor also gives finite co-dimensional stability in our setting. 

\begin{remark}
Note that the decomposition of the linear operator into a coercive one and a smoothing one appears in numerous works including \cite{schrodinger, ChenHou2}.
\end{remark}

\subsubsection{Brief explanation of some of the technical difficulties}

Let us briefly discuss a few technical difficulties in executing some of the above ideas. Carrying out the $\alpha$ expansion mentioned above to get an approximate profile is quite technical and involves analyzing the interaction between a good number of terms; the complexity is lowered by making a number of judicious changes of the unknown. The first step is to derive the boxed set of equations above \eqref{eq:biotsavartPsi}. Formally, one can set $\alpha=0$ in that system and proceed with the analysis. This still leaves a significant number of non-linear terms to be studied, and it is \emph{a priori} unclear how to solve the profile equation in this form. We thereafter define the ``good unknowns'' in Lemma \ref{lem:good}, which satisfy a more manageable system. We then observe that there exist factorized solutions to the profile equation, in which the relevant functions $f(z,\theta)$ can be written as $F(z)G(\theta),$ with an explicit dependence in $z$. This reduces everything to studying a linear integro-differential equation in $\theta$ with a parameter $A_*.$ Proving existence of a non-trivial solution requires some work and this is done by constructing an invariant, as we have already discussed above. One difficulty we encounter is that the approximate profile is actually not smooth in $\theta$ at $\theta=\frac{\pi}{2}.$ By virtue of the transport term $2\sin(2\theta)\partial_\theta$, particles are flowing from $\theta=0$ to $\theta=\frac{\pi}{2}$ and this causes a ``build-up'' of particles near $\theta=\frac{\pi}{2}$ (that this is the case can be seen from an analysis of \eqref{Gamma1}-\eqref{Gamma2} in the limit $\theta\rightarrow\frac{\pi}{2}$). We overcome this by doing all of our estimates in a space that is strictly weaker than the angular regularity of the profile so that we are able to cut the singularity at $\theta=\frac{\pi}{2}.$ That this is possible is a consequence of the fact that particles are flowing out of $\theta=0$, so the coercivity estimates require a lot of information at $\theta=0$ (where the approximate profile is locally analytic in $\theta$), but not so much at $\theta=\frac{\pi}{2}.$ As compared to \cite{E_Classical}, we have to fully understand the angular dynamics near $R=0$ through the analysis of \eqref{Gamma1}-\eqref{Gamma2}; this means that we will rely first on the angular dynamics for coercivity rather than just the radial dynamics. 
Another technical difficulty we encounter is that the expansion we use, \eqref{Expansion}, leads to the loss of a derivative in $R.$ At the same time, the approximate profile we construct is analytic in $R$ and the original Boussinesq system is locally well-posed in the spaces we consider; thus, in doing the non-linear analysis, at the top order, we consider the original equations prior to the $\alpha$ expansion. A further difficulty we encounter is in establishing the invertibility of the full linearized operator. Even upon setting $\alpha=0$, the linear operator is quite complicated as it is non-local in $\theta$ and $z$, while also having transport that is coupled in $\theta$ and $z.$ We are able to reduce the proof of invertibility to proving the non-negativity of a certain function of one-variable on $[0,\infty).$ This latter property is proven using computer assistance. This is discussed in section~\ref{sec:computerassistance}.

\subsubsection{Adapting the construction to the Euler equations}
We briefly describe how to adapt our construction to the Euler equations. This is a perturbation argument after one notices that the 3d Euler equations in axisymmetry, away from the symmetry axis, are well approximated by the Boussinesq system. This is carried out in more detail in Section~\ref{sec:eulerproof}. We are going to construct a solution which, to highest order in the inverse magnitude of the distance from the symmetry axis to the blow-up location, is well-approximated by the background Boussinesq solution of Theorem \ref{thm:boussinesq}. In this case, the self-similar transformation is going to be a central rescaling with respect to the moving point\footnote{In the full 3d picture, this corresponds to a central rescaling with respect to toroidal coordinates adapted to the circle in the $x_1x_2$ plane of radius $\xi(t)$ centered at the origin.} $(\xi(t), 0)$ in axisymmetric variables, rather than the origin. 

One difficulty in approximating 3d Euler by Boussinesq in this setting is that the vorticity of the Boussinesq solutions we construct are odd in $x$ and $y$, while oddness in $x$ (here, the horizontal distance to the blow-up point) is not preserved in 3d. This issue was not present in \cite{EJ3dE,ChenHou1} due to the presence of the boundary. The proof thus has to be modified to account for this issue, both in the evolution equations and in the analysis of the Biot-Savart law. One important ingredient is that we also modulate the location of the blow-up, allowing it to follow the flow of the velocity field. A further difficulty is that the Boussinesq profile is not smooth along the line $x=0$; this is handled by introducing a suitable cut-off and controlling the corresponding angular support. Another point is that, in
the course of our proof, we need to propagate that the support of our solution stays away from the axis of axisymmetry.

\subsubsection{Remark on the Computer Assistance in this paper}\label{sec:computerassistance}
The purpose of this remark is to explain the nature of the computer assistance appearing in this paper, which is restricted to the proof of one lemma. We have relegated the detailed discussion related to computer assistance to the companion paper \cite{symbolic}. Briefly, in the process of our proof, we need to prove the invertibility of a certain linear operator (see Section \ref{sec:invert}). The invertibility of the linear operator is shown analytically to follow from the non-negativity of a particular function of time solving a linear Volterra integral equation on $[0,\infty)$. Proving non-negativity ``by hand'' via a purely analytical proof appears to be challenging, at the current time, though we do not doubt that it is possible to do so. Indeed, it would be interesting to find a simple analytical proof of this fact. We were able to prove analytically that this function of time is non-negative on $[5,\infty)$ and on $[0,\frac{1}{4}].$ To establish non-negativity for moderately short times, we adopted a symbolic computer assisted approach in the companion paper \cite{symbolic}. 

\subsection{Organization of the paper}
In Section~\ref{sec:derivation}, we derive the equations in the Boussinesq case, after the change of coordinates $R = r^\alpha$. We define a convenient short-hand notation for the system (which comprises $3$ unknowns) in Section~\ref{sec:shorthand}. We then change coordinates to the self-similar framework, and in Section~\ref{sec:good} we define transformed unknowns which are at the core of our analysis. They allow us to obtain a $2\times2$ system which governs the local angular dynamics around the origin. Finally, in Section~\ref{sec:profilealphazero}, we state the main theorem about existence of an (approximate) self-similar profile when $\alpha = 0$.

In Section~\ref{sec:angular}, we construct the approximate self-similar profile. This is done by analyzing the angular system.

In Section~\ref{sec:linearanalysis}, we linearize the full equation around the approximate self-similar profile constructed at $\alpha = 0$, and proceed to show that a suitable~\emph{local in $r$} version of the full linear operator is coercive with respect to a carefully chosen inner product. In this section, we also define all the relevant inner products and spaces at the highest order.

In Section~\ref{sec:invert}, we analyze the \emph{full, nonlocal in $r$,} linear operator and we show that it is maximal accretive, that it has a one-dimensional kernel, and finally that it does not admit any generalized kernel element. The statements concerning the kernel are proved in our companion paper~\cite{symbolic} by means of a symbolic, computer assisted approach.

In Section~\ref{sec:nonlinear}, we provide the necessary estimates for the nonlinear terms at the top order, which avoid derivative loss.

In Section~\ref{sec:general}, we close the argument. We provide a general framework to construct and prove co-dimensional stability of an approximate self-similar profile, under minimal assumptions on the linearized operator (that it is invertible, and that it can be split into coercive and smoothing parts).

Finally, in Section~\ref{sec:eulerproof}, we extend the construction to the 3d Euler equations in axisymmetry.

\section{Derivation of the fundamental model}\label{sec:derivation}

Recall the Boussinesq equations:
\begin{equation}\label{eq:bou}\tag{B}
\begin{aligned}
&\p_t \omega + u \cdot \nabla \omega = \p_y \rho,\\
&\p_t \rho + u \cdot \nabla \rho = 0,\\
&u = \nabla^\perp \Delta^{-1} \omega,
\end{aligned}
\end{equation}
Our convention is that $\nabla^\perp = (\p_y,-\p_x)$. We assume that $\omega$ is odd-symmetric both in $x$ and in $y$, and that $\rho$ is odd-symmetric in $x$ and even-symmetric in $y$.
 These symmetries are propagated by the system.

Letting $\rho = -x q^2$ and $\psi$ such that $\nabla^\perp \psi = u$, we have
\begin{align}
&\p_t \omega + \nabla^\perp \psi \cdot \nabla \omega = -x \p_y q^2,\\
&\p_t q + \nabla^\perp \psi \cdot \nabla q= - \frac{\p_y \psi}{2x} q.
\end{align}
Rewriting in standard polar coordinates $(r, \theta)$, and recalling that $\p_y = \sin \theta \p_r + \frac{\cos \theta}{r}\p_\theta,$ the Boussinesq system becomes
\begin{align}
&\p_t \omega+  \frac 1 r \nabla^\perp_{r, \theta} \psi \cdot \nabla_{r, \theta}\omega = -\frac 12  \sin(2 \theta)r \p_r q^2 - \cos^2 \theta \p_\theta q^2,\label{eq:omega}\\
&\p_t q + \frac 1 r \nabla^\perp_{r, \theta} \psi \cdot \nabla_{r, \theta}q =- \frac q {2 r \cos \theta} \big(\sin \theta \p_r \psi + \frac{\cos \theta}{r} \p_\theta \psi \big) \label{eq:q}\\
&\p_r^2 \psi + \frac 1 r \p_r \psi + \frac {\p_\theta^2 \psi}{r^2} = \omega \label{eq:elliptic}
\end{align}
Here, we let $\omega = \p_2 u_1 - \p_1 u_2$. We renormalize $\psi$ defining $\psi = r^2 \Psi,$ and rewrite Equation~\eqref{eq:elliptic} in the form
\begin{equation}
D_r^2 \Psi + 4 D_r\Psi + 4 \Psi + \p_\theta^2 \Psi =  \omega,
\end{equation}
where we let $D_r = r\p_r$.

In these variables, the Boussinesq system becomes
\begin{align}
&\p_t \omega + \Big(\p_\theta \Psi D_r \omega - D_r \Psi \p_\theta \omega - 2 \Psi \p_\theta \omega \Big) = -\frac 12 \sin(2 \theta)D_r q^2 - \cos^2 \theta \p_\theta q^2,\label{eq:omegar}\\
&\p_t q + \Big(\p_\theta \Psi D_r q - D_r \Psi \p_\theta q - 2 \Psi \p_\theta q \Big) = - \frac 12 (2 \tan \theta \, \Psi + \tan \theta D_r \Psi + \p_\theta \Psi) q,\label{eq:qr}\\
&D_r^2 \Psi + 4 D_r\Psi + 4 \Psi + \p_\theta^2 \Psi =  \omega.\label{eq:biotr}
\end{align}
We change variables $R := r^\alpha$ to obtain:
\begin{align}
&\p_t \omega + \Big(\alpha \p_\theta \Psi D_R \omega - \alpha D_R \Psi \p_\theta \omega - 2 \Psi \p_\theta \omega \Big) +\frac 12 \sin(2 \theta)\alpha D_R q^2 - \cos^2 \theta \p_\theta q^2 =0,\label{eq:omegaa}\\
&\p_t q + \Big(\alpha \p_\theta \Psi D_R q - \alpha D_R \Psi \p_\theta q - 2 \Psi \p_\theta q \Big) + \frac 12 (2 \tan \theta \, \Psi + \alpha \tan \theta  D_R \Psi + \p_\theta \Psi) q=0,\label{eq:qa}\\
&\alpha^2 D_R^2 \Psi + 4 \alpha D_R\Psi + 4 \Psi + \p_\theta^2 \Psi =  \omega.\label{eq:abiot}
\end{align}
We let the angular projections
\begin{equation}
    \PP_0(F) := \frac 1{2\pi} \int_0^{2\pi} F(\theta) d \theta, \qquad \PP_1(F) := \sin (2\theta) \frac 1 {\pi} \int_0^{2\pi} F(\theta) \sin(2\theta) d\theta.
\end{equation}
We also let
\begin{equation}
    \Omega_0 := \omega/\alpha, \qquad P_0 := \PP_0 q, \qquad P_1 := \alpha^{-1}(q - P_0), \qquad \Psi_0 := \PP_1 \Psi, \qquad \Psi_1 := \alpha^{-1}(\Psi - \Psi_0).
\end{equation}
This in particular implies $q = P_0 + \alpha P_1$, $\Psi = \Psi_0 + \alpha \Psi_1$. The decomposition induces a splitting of the Biot-Savart law as follows:
\begin{align}
&\Psi_0(R, \theta) = - \frac14 \int_R^\infty \PP_1 \Omega_0(s) \frac{ds}{s} - \frac{1}{4}\int_0^R \Big(\frac s R \Big)^{4/\alpha} \PP_1 \Omega_0(s) \frac{d s}{s} = - \frac 1 4 \ellud(\Omega_0) \sin(2 \theta) + \alpha  R_1(\Omega_0).\label{eq:bs0}
\end{align}
Note that $R_1$ is a radial multiple of $\sin(2\theta)$. Here, we defined
\begin{equation}
    \ellud(F) := \frac 1 \pi\int_R^\infty \int_0^{2\pi} \frac{F(s, \theta)}{s} \sin(2\theta) d\theta ds.
\end{equation}
On the other hand, $\Psi_1$ satisfies 
\begin{equation}\label{eq:bs1}
    \alpha^2 D_R^2 \Psi_1 + 4 \alpha D_R\Psi_1 + 4 \Psi_1 + \p_\theta^2 \Psi_1 =  (\Omega_0 - \PP_1\Omega_0).
\end{equation}
It is then elementary to show that, at least formally, the stream function $\Psi_1$ satisfies $\Psi_1 = \tilde \Psi_1 + O(\alpha)$, where
\begin{equation}
\tilde \Psi_1 := (\p^2_\theta + 4)^{-1}(\Omega_0 - \mathbb{P}_1(\Omega_0)).
\end{equation}
The equation satisfied by $\Omega_0$ is then
\begin{equation}
\begin{aligned}
&\p_t \Omega_0 + \Big(\alpha \p_\theta \Psi D_R \Omega_0 - \alpha D_R \Psi \p_\theta \Omega_0 - 2 \Psi \p_\theta \Omega_0 \Big)+(P_0+\alpha P_1)(\sin(2\theta)D_R(P_0 + \alpha P_1) + 2 \cos^2 \theta \p_\theta P_1)=0.
\end{aligned}
\end{equation}
We let $D_\theta := \sin(2\theta) \p_\theta$. Neglecting terms of order $\alpha$, we then have, formally,
\begin{equation}
    \p_t \Omega_0 + \frac 12 \ellud(\Omega_0) D_\theta \Omega_0 + P_0(\sin(2\theta)D_R P_0 + 2 \cos^2 \theta \p_\theta P_1)  = O(\alpha),
\end{equation}
It remains to deduce the equations satisfied by $P_0$ and $P_1$. Projection onto the zeroth $\theta$ mode reveals:
\begin{equation}
    \begin{aligned}
    &\p_t P_0 + \PP_0\Big(\alpha \p_\theta (\Psi_0 + \alpha \Psi_1) D_R (P_0 + \alpha P_1) - \alpha D_R (\Psi_0 + \alpha \Psi_1) \p_\theta (P_0 + \alpha P_1) - 2\alpha  (\Psi_0 + \alpha \Psi_1) \p_\theta P_1 \\
    &+ \frac 12 (2 \tan \theta \, (\Psi_0 + \alpha \Psi_1) + \alpha \tan \theta  D_R (\Psi_0 + \alpha \Psi_1) +  \p_\theta (\Psi_0 + \alpha \Psi_1)) (P_0 + \alpha P_1)  \Big)=0.
    \end{aligned}
\end{equation}
Formally neglecting terms of order $\alpha$, we have
\begin{equation}
    \p_t P_0  - \frac 1 4 \ellud(\Omega_0) P_0  = O(\alpha),
\end{equation}
We let $\PP_0^\perp q:= q - \PP_0 q$. We have, concerning $P_1$, upon projecting,
\begin{align*}
    &\alpha \p_t P_1 + \PP^\perp_0 \Big(\alpha \p_\theta (\Psi_0 + \alpha \Psi_1) D_R (P_0 + \alpha P_1)  - \alpha^2 D_R (\Psi_0 + \alpha \Psi_1) \p_\theta P_1 - 2\alpha  (\Psi_0 + \alpha \Psi_1) \p_\theta P_1 \Big) \\
    &+ \frac 12 \PP^\perp_0\Big( (2 \tan \theta \, (\Psi_0 + \alpha \Psi_1) + \alpha \tan \theta  D_R (\Psi_0 + \alpha \Psi_1) +  \p_\theta (\Psi_0 + \alpha \Psi_1)) (P_0 + \alpha P_1)\Big) = 0.
\end{align*}
This implies
\begin{align*}
    & \p_t P_1 + \PP_0^\perp( \p_\theta \Psi_0 D_R P_0- 2  \Psi_0 \p_\theta P_1)+ \frac12 \PP_0^\perp\Big((2 \tan \theta \Psi_1 + \tan \theta D_R \Psi_0 + \p_\theta \Psi_1)P_0+(2 \tan \theta \Psi_0 + \p_\theta \Psi_0) P_1 \Big)\\
    &+\frac 12 \alpha \PP_0^\perp\Big(\tan \theta D_R \Psi_1 P_0 + (2 \tan \theta \Psi_1 + \tan \theta D_R(\Psi_0 + \alpha \Psi_1) + \p_\theta \Psi_1)P_1\Big)\\
    &+ \PP^\perp_0 \Big(\alpha \p_\theta \Psi_0 D_R P_1 + \alpha \p_\theta (\Psi_0 + \alpha \Psi_1) D_R P_1 - \alpha D_R (\Psi_0 + \alpha \Psi_1) \p_\theta P_1 - 2\alpha \Psi_1 \p_\theta P_1 \Big) =0.
\end{align*}
This finally implies, formally neglecting terms of order $\alpha$,
\begin{align*}
    &\p_t P_1 +\PP_0^\perp\Big(\frac 12 \ellud(\Omega_0) D_\theta P_1 -\frac 12 \cos(2 \theta) \ellud(\Omega_0) D_R P_0 - \frac 14 \ellud(\Omega_0)P_1 \\
    &\qquad \qquad \qquad- \frac{P_0}{4} \sin^2 \theta D_R(\ellud(\Omega_0)) + \frac 12 \Big(2 \tan \theta \tilde \Psi_1 + \p_\theta \tilde \Psi_1 \Big)P_0\Big)= O(\alpha).
\end{align*}

\begin{remark}
Although the analysis here is only formal, later we will proceed to bound the error terms, which will show that they are actually of size $\alpha$.
\end{remark}

For future reference, we collect here the full system in $\Omega_0, P_0, P_1, \Psi_0, \Psi_1$, which is going to be the core of our subsequent analysis. Recall that $\Psi = \Psi_0 + \alpha \Psi_1$ and $P = P_0 + \alpha P_1$:
\begin{equation}
\boxed{
\begin{aligned}
    &\p_t \Omega_0 + \Big(\alpha \p_\theta \Psi D_R \Omega_0 - \alpha D_R \Psi \p_\theta \Omega_0 - 2 \Psi \p_\theta \Omega_0 \Big)+P(\sin(2\theta)D_RP + 2 \cos^2 \theta \p_\theta P_1)=0.\\
    &\p_t P_0 + \PP_0\Big(\alpha \p_\theta \Psi D_R P - \alpha D_R \Psi \p_\theta P - 2\alpha  \Psi \p_\theta P_1 + \frac 12 (2 \tan \theta \, \Psi + \alpha \tan \theta  D_R \Psi +  \p_\theta \Psi) P  \Big)=0,\\
     & \p_t P_1 + \PP_0^\perp( \p_\theta \Psi_0 D_R P_0- 2  \Psi_0 \p_\theta P_1)+ \frac12 \PP_0^\perp\Big((2 \tan \theta \Psi_1 + \tan \theta D_R \Psi_0 + \p_\theta \Psi_1)P_0+(2 \tan \theta \Psi_0 + \p_\theta \Psi_0) P_1 \Big)\\
    &\quad +\frac 12 \alpha \PP_0^\perp\Big(\tan \theta D_R \Psi_1 P_0 + (2 \tan \theta \Psi_1 + \tan \theta D_R\Psi + \p_\theta \Psi_1)P_1\Big)\\
    &\quad + \PP^\perp_0 \Big(\alpha \p_\theta \Psi_0 D_R P_1 + \alpha \p_\theta \Psi D_R P_1 - \alpha D_R \Psi \p_\theta P_1 - 2\alpha \Psi_1 \p_\theta P_1 \Big) =0,\\
    &\alpha D_R^2 \Psi_0 + 4 D_R\Psi_0 = \PP_1(\Omega_0),\\
    &\alpha^2 D_R^2 \Psi_1 + 4 \alpha D_R\Psi_1 + 4 \Psi_1 + \p_\theta^2 \Psi_1 =  (\Omega_0 - \PP_1\Omega_0).
\end{aligned}
    }
\end{equation}
In addition, the full stream function $\Psi = \Psi_0 + \alpha \Psi_1$ satisfies the Biot-Savart law:
\begin{equation}\label{eq:biotsavartPsi}
    \alpha^2 D_R^2 \Psi + 4 \alpha D_R\Psi + 4 \Psi + \p_\theta^2 \Psi =  \alpha \Omega_0.
\end{equation}

\subsection{Short-hand notation for the Boussinesq system}\label{sec:shorthand}
We recast the Boussinesq system in a useful shorthand notation. We consider $f = (f_1, f_2, f_3)$ a vector-valued function on $\R^2$, and we identify the three components as
$$
f_1 := \Omega_0[f], \qquad f_2 := P_0[f], \qquad f_3 := P_1[f].
$$
We set the notation, upon considering $g = (g_1, g_2, g_3) = (\Omega_0[g], P_0[g], P_1[g])$,
\begin{align}\label{eq:K1}
    &\calK(f,g)_1 := \Big(\alpha \p_\theta \boldpsi[g] D_R \boldom_0[f] - \alpha D_R \boldpsi[g] \p_\theta \boldom_0[f] - 2 \boldpsi[g] \p_\theta \boldom_0[f] \Big)\\
    &\qquad +(\boldP[g])(\sin(2\theta)D_R(\boldP[f]) + 2 \alpha^{-1} \cos^2 \theta \p_\theta \boldP[f]),\\
    &\calK(f,g)_2 := \Big(\alpha \p_\theta \boldpsi[g] D_R \boldP[f]- \alpha D_R \boldpsi[g] \p_\theta \boldP[f] - 2\boldpsi[g] \p_\theta \boldP[f] \Big) \label{eq:K2}\\
    &\qquad +\frac 12 (2 \tan \theta \,  \boldpsi[g] + \alpha \tan \theta  D_R  \boldpsi[g] + \p_\theta  \boldpsi[g]) \boldP[f],
\end{align}
Here, we used $\boldpsi[f] := \boldpsi_0[f] + \alpha \boldpsi_1[f]$, and $\boldP[f] := \boldP_0[f] + \alpha \boldP_1[f]$, and $\Psi_0[f]$, $\Psi_1[f]$ are obtained from $\Omega_0[f]$ according to resp.~\eqref{eq:bs0} and~\eqref{eq:bs1}.
\begin{remark}
Note that $\calK(f,f)_1$ corresponds to the nonlinear terms in \eqref{eq:omegaa}, whereas $\calK(f,f)_2$ corresponds to the nonlinear terms in \eqref{eq:qa}.
\end{remark}
Due to the analysis of the previous section, we are also lead to define the main nonlinear contribution $\calN(f,g) := (\calN(f,g)_1, \calN(f,g)_2, \calN(f,g)_3)$ as follows:
\begin{align}
    &\calN(f,g)_1 := \frac 12 \ellud(\boldom_0[g]) D_\theta \boldom_0[f] + \boldP_0[g](\sin(2\theta)D_z \boldP_0[f] + 2 \cos^2 \theta \p_\theta \boldP_1[f]) , \\
    &\calN(f,g)_2 :=  - \frac 1 4 \ellud(\boldom_0[g]) \boldP_0[f], \\
    &\calN(f,g)_3 :=  \PP_0^\perp\Big\{\frac 12 \ellud(\boldom_0[g]) D_\theta \boldP_1[f] \\
    &\qquad  - \frac 12 \cos(2 \theta) \ellud(\boldom_0[g]) D_z \boldP_0[f] -\frac 14 \ellud(\boldom_0[g])\boldP_1[f] \\
    & \qquad - \frac{\boldP_0[f]}{4} \sin^2 \theta D_z(\ellud(\boldom_0[g])) + \frac 12 \Big(2 \tan \theta \tilde{\Psi}_1[g] + \p_\theta \tilde{\Psi}_1[g] \Big)\boldP_0[f]\Big\},
\end{align}
This induces the following definition for the error terms:
\begin{align}
&\alpha E(f,g)_1 := \calK(f,g)_1- \calN(f,g)_1,\label{eq:e1}\\
&\alpha E(f,g)_2 := \PP_0 (\calK(f,g)_2) - \calN(f,g)_2,\label{eq:e2}\\
&\alpha E(f,g)_3 := \PP^\perp_0 (\calK(f,g)_2) - \calN(f,g)_3.\label{eq:e3}
\end{align}
In conclusion, the Boussinesq system can be written as:
\begin{equation}\label{eq:masterbou}
\p_t f + \calN(f,f) + \alpha E(f,f) = 0.
\end{equation}
\begin{remark}
Note that, in hindsight, this decomposition is very useful for our linear analysis, however it is not favorable in terms of derivative count at the top order nonlinear estimates. This is the reason why, in the later sections, in order to close the estimates at the top order without ``derivative loss'', we will need to consider a slightly altered version $\tilde \calN(f,g)$, which differs from $\calN(f,g)$ by terms of size $O(\alpha)$.
\end{remark}

\subsection{Self-similar change of variables}

In this Section, we introduce the self-similar change of coordinates. In this work, we are going to modulate the concentration speed and the blow-up time. We introduce the self-similar coordinates by the following relations:
\begin{equation}\label{eq:ss-mod}
\begin{aligned}
    &\frac{ds}{dt} = \tilde{\mu}(s), \qquad z = \frac{R}{\tilde{\lambda}(s)},\\
    &\Omega_0(t,R, \theta) =\tilde{\mu}(s) \Omega_0(s, z, \theta), \qquad P_0(t,R) = \tilde{\mu}(s) P_0(s, z), \qquad P_1(t,R, \theta) = \tilde{\mu}(s)  P_1(s, z, \theta),\\
    &\Psi_0(t,R, \theta) =\tilde{\mu}(s) \Psi_0(s, z, \theta), \qquad \Psi_1(t,R, \theta) = \tilde{\mu}(s)  \Psi_1(s, z, \theta).
\end{aligned}
\end{equation}

Here, $\tilde\lambda = \tilde \lambda(t)$ and $\tilde \mu = \tilde \mu(t)$ are functions of time (abusing notation, we will denote the corresponding functions of $s$ with the same name). By a slight abuse of notation,  we denote the quantities in self-similar coordinates with the same notation as their un-rescaled counterparts. In the remainder of the work, we will only refer to the quantities in self-similar coordinates so as to avoid confusion.

The system~\eqref{eq:masterbou} is then rewritten as:
\begin{equation}\label{eq:profiletimedep}
    \p_s f + \frac{\tilde \mu_s}{\tilde \mu} f  - \frac{{\tilde \lambda}_s}{\tilde \lambda} z\p_z f  + \calN(f,f) + \alpha E(f,f) = 0.
\end{equation}
In case $\tilde \mu(s) = (1 - t \mu)^{-1}$ and $\tilde \lambda(s) = (1-t\mu)^{\lambda}$, with $\mu$ and $\lambda$ positive constants, \eqref{eq:profiletimedep} reduces to
\begin{equation}\label{eq:profileshort}
    \p_s f + \mu f + \mu \lambda z \p_z f + \mathcal{N}(f, f) + \alpha E(f,f)=0,
\end{equation}
and $s = -\mu^{-1} \log(1- \mu t)$. 

\subsection{The ``good'' unknowns}\label{sec:good}
The goal of this Section is to show that application of a certain differential operator to $\calN$ transforms them in a ``nice'' form (which is amenable to linear analysis). More precisely, we have the following Lemma.
\begin{lemma}\label{lem:good}
Define the linear operator $\calB(f)$ as follows:
\begin{equation}
    \calB(f) = \Big(\cos^2 \theta \p_\theta f_1, \  f_2,  \ \cos^2 \theta \p_\theta \Big(\frac 12 \sin(2\theta) D_z f_2 + \cos^2 \theta \p_\theta f_3\Big)\Big).
\end{equation}
Then,
$$
\calB \calN(f,f) = \calM(\calB f, \calB f).
$$
Here, $\calM$ is the quadratic form defined by
\begin{align}
    &(\calM(f,f))_1 := \frac 12 \ellg(f_1) D_\theta f_1 + \ellg(f_1) f_1+ 2 f_2 f_3, \\
    &(\calM(f,f))_2 = \frac 1 4 \ellg(f_1) f_2,\\
    &(\calM(f,f))_3 = \frac 12 \ellg(f_1) D_\theta f_3 + \frac 7 4 \ellg(f_1) f_3 + \frac 12 \cos^2 \theta f_2 \Big(f_1- \dashint f_1\Big).
\end{align}
and $\ellg(h)$ is defined as
$$\ellg(F):= \int_z^\infty \frac{1}{z'\pi} \int_0^{2\pi}  F(z', \mu) d \mu d z' = 2 \int_z^\infty \frac{1}{z'} \dashint F dz'.$$
\end{lemma}

\begin{proof}[Proof of Lemma~\ref{lem:good}]
We recall $(f_1, f_2, f_3) := (\Omega_0[f], P_0[f], P_1[f])$. Our goal is to find equations for the ``good'' unknowns $G_1, G_2$, where
\begin{align}
&G_1 := \cos^2 \theta \p_\theta \Omega_0,\\
&G_2 := \cos^2 \theta \p_\theta \varpi,
\end{align}
and here $\varpi$ is defined as $\varpi := \frac 12 \sin(2\theta) D_z P_0 + \cos^2 \theta \p_\theta P_1$.

We let $W:= \cos^2 \theta \p_\theta P_1$. We have
\begin{equation}\label{eq:forvarpi1}
\begin{aligned}
&\cos^2 \theta \p_\theta(\calN(f,f)_3) = \frac 12 \ellud(\Omega_0) D_\theta W + \frac 3 4 \ellud(\Omega_0) W \\
& \qquad + \ellud(\Omega_0) \cos^2 \theta \sin(2\theta) D_z P_0 - \frac 14 \cos^2 \theta \sin(2\theta) P_0 D_z (\ellud(\Omega_0)) \\
& \qquad +\frac{P_0}{2} \int_0^\theta \cos^2 \mu (\p_\mu (\Omega_0 - \mathbb{P}_1 (\Omega_0))) d \mu.
\end{aligned}
\end{equation}
It is natural to define the operator $\scrtemp$ as $\scrtemp q := \frac 12 \ellud(\Omega_0) D_\theta q + \frac 34 \ellud(\Omega_0) q$.
Applying $\sin(2\theta) D_z$ to $\calN(f,f)_2$, we have
\begin{equation}\label{eq:forvarpi2}
\begin{aligned}
\sin(2\theta) D_z \calN(f,f)_2 = \scrtemp (\sin(2\theta) D_z P_0) - \frac{\sin(2\theta)}{4} D_z ( \ellud(\Omega_0)) P_0 - 2 \sin (2\theta) \cos^2 \theta \ellud(\Omega_0) D_z P_0.
\end{aligned}
\end{equation}
Combining~\eqref{eq:forvarpi1} and~\eqref{eq:forvarpi2}, we then have
\begin{equation}\label{eq:varpi}
\begin{aligned}
&\frac 12 \sin(2\theta) D_z \calN(f,f)_2 + \cos^2 \theta \p_\theta(\calN(f,f)_3) = \scrtemp \varpi - \frac 1 4 \sin(2\theta)\Big(\cos^2 \theta + \frac 12\Big) P_0 D_z(\ellud(\Omega_0))\\
&\qquad \qquad + \frac{P_0} 2 \int_0^\theta \cos^2 \mu \p_\mu (\Omega_0 - \mathbb{P}_1(\Omega_0))  d \mu.
\end{aligned}
\end{equation}
Recall that $G_1 = \cos^2 \theta \p_\theta \Omega_0$, and $G_2 = \cos^2 \theta \p_\theta \varpi$. Note that, by integration by parts, $\ellud(\Omega_0) = \ellg(G_1)$. We apply the operator $\cos^2 \theta \p_\theta$ to the equation~\eqref{eq:varpi} to obtain
\begin{equation}
\begin{aligned}
&\cos^2 \theta \p_\theta \Big(\frac 12 \sin(2\theta) D_z \calN(f,f)_2 + \cos^2 \theta \p_\theta(\calN(f,f)_3)\Big)\\
&= \frac 12 \ellg(G_1) D_\theta G_2 + \frac 7 4 \ellg(G_1) G_2+\frac 12 \cos^2 \theta P_0 \Big(G_1- \dashint G_1\Big).
\end{aligned}
\end{equation}
We then apply the operator $\cos^2 \theta \p_\theta$ to the equation for $\Omega_0$, which yields
\begin{equation}
\cos^2 \theta \p_\theta(\calN(f,f)_1 ) =  \frac 12 \ellg(G_1) D_\theta G_1 + \ellg(G_1) G_1 + 2 P_0 G_2.
\end{equation}
\end{proof}

\subsection{The angular profile equation in the ``good'' unknowns}
Setting formally $\p_s f = 0$, $\alpha = 0$ and $\mu = \lambda =1$ in Equation~\eqref{eq:profileshort}, and applying the transformation of Lemma~\ref{lem:good}, we obtain the equation (with $(G_1, P_0, G_2) := \calB f$):
\begin{align}
& G_1 + z\p_z G_1 + \frac 12 \ellg(G_1) D_\theta G_1 + \ellg(G_1) G_1 = -2 P_0 G_2, \label{eq:main1ss}\\
& G_2 + z\p_z G_2+ \frac 12 \ellg(G_1) D_\theta G_2 + \frac 7 4 \ellg(G_1) G_2 =-  \frac 12 \cos^2 \theta P_0 \Big(G_1- \dashint G_1\Big), \label{eq:main2ss}\\
& P_0 + z\p_z P_0 = \frac 1 4 \ellg(G_1) P_0.\label{eq:main3ss}
\end{align}
For $B_*\in \R$, and for angular functions $\Gamma_1$ and $\Gamma_2$, we seek a steady profile of the form
$$
G_1 = \frac{z}{(1+z)^2} \Gamma_1(\theta), \quad G_2 = -(2B_*)^{-1} \frac{z}{(1+z)^2}  \Gamma_2(\theta), \quad P_0 = \frac{B_*}{1+z}.
$$
This yields the equations
\begin{align}
& 2 D_\theta \Gamma_1 + 6 \Gamma_1 =   \Gamma_2, \label{eq:gamma1}\\
& 2 D_\theta \Gamma_2 + 9 \Gamma_2 = B_*^2 \cos^2 \theta \Big(\Gamma_1 - \dashint \Gamma_1\Big),\label{eq:gamma2}\\
&\dashint \Gamma_1 = 2.\label{eq:gamma3}
\end{align}
In what follows, we are going to denote $A_* := B_*^2$. Our intermediate goal will be to show that the system~\eqref{eq:gamma1}--\eqref{eq:gamma2}--\eqref{eq:gamma3} admits a suitably regular solution.
\begin{lemma}\label{lem:profile}
There exist $$
(\Gamma^*_1, \Gamma^*_2) \in (C^\infty([0, \pi/2))^2 \cap (C^{3/2}([0,\pi/2]) \times C^{9/4}([0, \pi/2)))
$$
and a value of $B_*$ (in our case, $B_*^2 = 130$), which solve~\eqref{eq:gamma1},~\eqref{eq:gamma2},~\eqref{eq:gamma3}.
\end{lemma}
We defer the proof to Section~\ref{sec:angular}. An important consequence of Lemma~\ref{lem:profile} is that there exists an approximate profile in the original variables $\Omega_0, P_0, P_1$.

\subsection{Construction of an approximate profile}\label{sec:profilealphazero}
We have the following Lemma.
\begin{lemma}\label{lem:profileomega}
There exist $(\tilde\Gamma_1, \tilde \Gamma_2) \in (C^\infty([0, \pi/2)))^2 \cap (C^{\frac{1}{2}}([0,\pi/2]) \times C^{\frac{1}{2}-}([0, \pi/2])$ such that, letting
\begin{align}
    &\Omega_0^* := \frac{z}{(1+z)^2} \tilde \Gamma_1,\\
    &P^*_1 := - \frac1{2B_*} \frac{z}{(1+z)^2} \tilde \Gamma_2,
\end{align}
we have
\begin{align}
&G^*_1= \cos^2 \theta \p_\theta \Omega_0^*,\\
&G^*_2 = \cos^2 \theta \p_\theta \Big(\frac 12 \sin(2\theta) D_z P_0^* + \cos^2 \theta \p_\theta P_1^* \Big),
\end{align}
and $\mathbb{P}_0 P^*_1 = 0$. In particular, defining $
f^* := (\Omega_0^*, P_0^*, P_1^*)$,
$f^*$ satisfies the $\alpha  = 0$ profile equation
\begin{equation}
f^* + z\p_z f^* + \calN(f^*, f^*) = 0.
\end{equation}

\end{lemma}
\begin{proof}

We first notice that, if $(\Gamma_1^*, \Gamma_2^*)$ are as in Lemma~\ref{lem:angularprofile}, then, by first integrating the profile equation~\eqref{eq:gamma2} (using the partial regularity provided by Lemma~\ref{lem:angularprofile}) we obtain  $\Gamma_2^* \in C^{2}([0, \pi/2])$. Moreover, integrating~\eqref{eq:gamma2}, we get $\Gamma_1^* \in C^{\frac 32}([0, \pi/2])$. From this, the existence of $\tilde \Gamma_1$ in the conditions above follows directly by integration. For $\tilde \Gamma_2$, we consider the expression $\Xi^* :=-2B_* \frac{(1+z)^2}{z} \big(G_2^* - \cos^2 \theta \p_\theta\big(\frac12\sin(2\theta) D_z P_0^* \big)\big)$ and show that $\Xi^* \in C^{\frac{5}{2}}[0, \pi/2]$. The claim for $\tilde \Gamma_2$ then follows by direct integration.

We have: $\Xi^* =  \Gamma_2 - B_*^2 \cos^2 \theta \p_\theta\big(\sin(2\theta)\big) = \Gamma_2 +2 B_*^2 \cos^2 \theta + E_1(\theta)$, where $E_1(\theta)$ vanishes to fourth order at $\theta = \pi/2$. Now, $2 D_\theta (2\cos^2\theta) + 9 \cdot 2 \cos^2(\theta) = 2 \cos^2 \theta + E_2(\theta)$, where $E_2(\theta)$ vanishes to fourth order at $\theta = \pi/2$. We conclude that
$$
2 D_\theta \Xi^* + 9 \Xi^* = E_3,
$$
where $E_3$ vanishes to order $5/2$ at $\pi/2$. The claim  $\Xi^* \in C^{\frac{5}{2}}[0, \pi/2]$ then follows integrating the above equation.
\end{proof}

\section{Analysis of the angular profile equation}\label{sec:angular}
The goal of this section is to analyze the angular profile equation in order to prove Lemma~\ref{lem:profile} and to collect important properties of the linearized operator in the angular variables. These properties will be useful to show dissipativity of the full linearized operator later on.

We consider the angular system:
\begin{align}\label{eq:ang1}
    &\p_t \Gamma_1 + 2 D_\theta \Gamma_1 + 6 \Gamma_1 = \Gamma_2,\\
    &\p_t \Gamma_2 + 2 D_\theta \Gamma_2 + 9 \Gamma_2 = A_* \cos^2 \theta (\Gamma_1 - \dashint \Gamma_1).\label{eq:ang2}
\end{align}
We let $A_* = 130$. We define, for functions $F: \mathbb{S}^1 \to \R$:
\begin{equation}
\bar F:= F- F|_{\theta = 0}.
\end{equation}

\begin{definition}
We let $L^2_*(\mathbb{S}^1)$ to be the space of $L^2$ functions on $\mathbb{S}^1$ that are even with respect to $0$ and $\pi/2$. 
\end{definition}

\begin{definition}
We let the Banach space $\mathcal{H}^k_\theta(\mathbb{S}^1) \subset L^2_*(\mathbb{S}^1)$  with the norm:
\begin{equation}\label{eq:hknorm}
    |f|_{\mathcal{H}^k_\theta}^2:=\sum_{j=0}^k |(\cos \theta)^{j-7/4}  \partial_\theta^j f|_{L_\theta^2}^2.
\end{equation}

\end{definition}

\begin{remark}
Weighted Sobolev embedding shows that $\mathcal{H}^k_\theta\hookrightarrow C^{5/4}_\theta \cap  C^{k-1}_\theta([0, \pi/2)) $ when $k\geq 2$.
\end{remark}

\begin{definition}\label{eq:lthetadef}
We define the angular operator $L_\theta: D(L_\theta) \to \calH^k_\theta$ as follows:
\begin{equation}
    L_\theta \Gamma := \Big(2 D_\theta \Gamma_1 + 6 \Gamma_1 - \Gamma_2, \quad 2 D_\theta \Gamma_2 + 9 \Gamma_2 - 130 \cos^2 \theta \Big(\Gamma_1 - \dashint \Gamma_1\Big) \Big).
\end{equation}
Here, $\Gamma = (\Gamma_1, \Gamma_2)$ and $D(L_\theta)\subset \calH^k_\theta$ is the natural domain of the operator $L_\theta$ which is dense in $\calH^k_\theta$.
\end{definition}

\begin{definition}
 We define the operator $\mathcal{A}$ by
\begin{equation}
\mathcal{A}(F) := (\tan \theta)^{-2}(F-F(\theta = 0)).
\end{equation}
\end{definition}

We introduce an important subspace on which the angular operator is dissipative.
\begin{definition} Let $\gamma = \tan \theta$. We define the following linear subspace:
\begin{equation}\label{eq:decaysubsp}
\calS := \Big\{ (\Gamma_1, \Gamma_2) \in D(L_\theta): \int_0^\infty (13 \bar \Gamma_1 + \bar \Gamma_2) \gamma^{-2} d \gamma = 0\Big\}.
\end{equation}
Note that we require in particular that the function is twice differentiable on compact subsets of $[0, \pi/2)$.
\end{definition}

\begin{remark}
The choice of $A_*$ above is dictated by requiring that $\calS$ is invariant under the semigroup generated by $L_\theta$ (see Lemma~\ref{lem:angulardecay}). Alternatively, one can show the existence of a value of $A_*$ with favorable properties by a continuity argument.
\end{remark}

We define the inner product which makes $L_\theta$ a coercive operator.

\begin{definition}
Let $\mathbf{C} = (C_1, C_2, C_3, C_4)$ be a string of constants. For $f = (f_1, f_2), g = (g_1, g_2) \in (\calH^2_\theta)^2$, we define the inner product
\begin{equation}\label{eq:definn}
\begin{aligned}
    &(f,g)_{\mathbf{C}} :=  \int_0^\infty \Big( (1+\gamma^2) \mathcal{A}^2(13 f_1 + f_2) \mathcal{A}^2(13 g_1 + g_2) + \gamma^2  \mathcal{A}^2 f_1 \mathcal{A}^2 g_1 \Big) d\gamma \\
    &+ C_1  \int_0^\infty \mathcal{A}^2 (13 f_1 + f_2) d\gamma  \int_0^\infty \mathcal{A}^2 (13 g_1 + g_2) d\gamma\\
    &+ C_2 \int_0^\infty \mathcal{A} f_1 d\gamma  \int_0^\infty \mathcal{A} g_1 d\gamma\\
    &+ C_3 \int_0^\infty (1+\gamma^2)^{\frac 38}\Big( (\calA f_1) (\calA g_1) + \calA(13 f_1 + f_2) \calA(13 g_1 + g_2) \Big) d\gamma \\
    &+C_4 \int_0^\infty (1+\gamma^2)^{\frac 3 4} ( f_1 g_1 +  f_2 g_2 )  d\gamma.
\end{aligned}
\end{equation}
\end{definition}

We then have the main Lemma on the coercivity of the angular operator. 

\begin{lemma}\label{lem:angulardecay}
Consider the semi-group generated by $-L_\theta$ (with $k \geq 2$). The following assertions hold true.
\begin{enumerate}
\item The subspace $\calS$ is invariant under the evolution given by the semi-group generated by $-L_\theta$.

\item There exists a choice of $\mathbf{C}$ in~\eqref{eq:definn} such that there is $c>0$ for which
\[(L_\theta f, f)_\mathbf{C} \geq c (f,f)_\mathbf{C},\] for every $f\in\mathcal{H}^2_\theta\cap \mathcal{S}$.

\item Let $k\geq 3$, and consider the inner product $(\cdot, \cdot)_\mathbf{C}$ from the previous point. For $f, g \in \calH^k_\theta$, define
\[
(f,g)_{\calH^k_\theta\cap \calS} := (f,g)_{\mathbf{C}} + C_5 \int_0^\infty (1+\gamma^2)^{\frac34} \Big( (\Lambda^k f_1) (\Lambda^k g_1) +  (\Lambda^k f_2) (\Lambda^k g_2)\Big) d\gamma.
\]
Here, $\Lambda := \sqrt{1+\gamma^2}\p_\gamma$. 
Then, there exists a constant $C_5 > 0$ such that:
\begin{enumerate}
    \item There is $c>0$ such that the inequality \[(L_\theta f, f)_{\mathcal{H}^k_\theta\cap\mathcal{S}} \geq c (f,f)_{\mathcal{H}^k_\theta\cap\mathcal{S}}\]
    holds for every $f\in\mathcal{H}^k_\theta\cap \mathcal{S}$;
    \item Consider the norm induced by $(\cdot, \cdot)_{\calH^k_\theta\cap \calS}$ on $\calH^k_\theta \cap \calS$:
    \[|f|_{\calH^k_\theta \cap \calS}:= \sqrt{(f,f)_{\calH^k_\theta \cap \calS}}.
    \]
    Then, $|\cdot |_{\calH^k_\theta \cap \calS}$ is equivalent to the norm $|\cdot |_{\calH^k_\theta}$ defined in~\eqref{eq:hknorm}.
\end{enumerate}

\item Let $k \geq 2$, and let $\Gamma_0 := (\Gamma_{1,0}, \Gamma_{2,0}) \in \calS$. The semi-group generated by $-L_\theta$ is strongly continuous, and there exist $C, c > 0$ such that the semi-group estimate
\begin{equation}
|\exp(- t L_\theta)\Gamma_0|_{\calH_\theta^k} \leq C \exp(-c t)
\end{equation}
holds true for all $t \geq 0$.
\end{enumerate}
\end{lemma}

\begin{proof}[Proof of Lemma~\ref{lem:angulardecay}] {\bf Proof of Claim 1)}. \quad  We define $Z_1 := \calA (\Gamma_1) = \bar \Gamma_1/\gamma^2$ and $Z_2 := \calA (\Gamma_2) = \bar \Gamma_2 / \gamma^2$. Then,
\begin{align}
    &\calA ((L_\theta \Gamma)_1) =  4 \gamma \p_\gamma Z_1 + ( 6 + 4 \times 2) Z_1 - Z_2\\
    &\calA ((L_\theta \Gamma)_2) = 4 \gamma \p_\gamma Z_2 + (9 + 4 \times 2) Z_2 -   \frac{A_*}{1+\gamma^2} Z_1 - \frac{A_*}{1+\gamma^2} \dashint \bar \Gamma_1.
\end{align}
We recall that $A_* = 130$. We also define a new variable $
H := 13 Z_1 + Z_2$. We then have
\begin{equation}
    \begin{aligned}
    \calA(13 (L_\theta \Gamma)_1 + (L_\theta \Gamma)_2) = 4 \gamma\p_\gamma H + 13 \times 14 Z_1 - 13 Z_2 + 17 Z_2 - \frac{130}{1+\gamma^2}Z_1 - \frac{130}{1+\gamma^2} \dashint \bar \Gamma_1.
    \end{aligned}
\end{equation}
Rewriting, we have
\begin{equation}\label{eq:H}
    \begin{aligned}
    \calA(13 (L_\theta \Gamma)_1 + (L_\theta \Gamma)_2)  =  4 \gamma\p_\gamma H + 4H+  \frac{130\gamma^2}{1+\gamma^2}Z_1 - \frac{130}{1+\gamma^2} \dashint \bar \Gamma_1.
    \end{aligned}
\end{equation}
Integration of the above equation in $\gamma$ from $0$ to $\infty$ reveals that $
\int_0^\infty H d\gamma
$
is an invariant under the evolution given by the semi-group induced by $L_\theta$, hence $\calS$ is an invariant subspace for the dynamics under $L_\theta$, which shows {\bf Claim 1).}

\vspace{5pt}

{\bf Proof of Claim 2)}.  We apply $\calA$ to equation~\eqref{eq:H}, and we obtain, letting $W:=\calA(H) = \frac{\bar H}{\gamma^2}$:
\begin{equation}\label{eq:barHdiv}
    \begin{aligned}
    \calA^2(13 (L_\theta \Gamma)_1 + (L_\theta \Gamma)_2) = 4 \gamma\p_\gamma W + 12 W+  \frac{130}{1+\gamma^2}\frac{\bar \Gamma_1}{\gamma^2} + \frac{130}{1+\gamma^2} \dashint \bar \Gamma_1.
    \end{aligned}
\end{equation}
We notice that
$$
\dashint \bar \Gamma_1 = \frac2\pi \Big( \int_0^\infty \frac{\bar \Gamma_1}{1+\gamma^2}\, d\gamma \pm \int_0^\infty \frac{\bar \Gamma_1}{\gamma^2}\, d \gamma \Big) = -\dashint \frac{\bar \Gamma_1}{\gamma^2}\,  + \frac2\pi\int_0^\infty \frac{\bar \Gamma_1}{\gamma^2}\, d\gamma.
$$
Plugging this in equation~\eqref{eq:barHdiv}, we obtain, recalling that $Z_1 = \bar \Gamma_1/\gamma^2$,
\begin{equation}\label{eq:barHalmost}
    \begin{aligned}
    \calA^2(13 (L_\theta \Gamma)_1 + (L_\theta \Gamma)_2) = 4 \gamma\p_\gamma W + 12 W+  \frac{130}{1+\gamma^2}Z_1 - \frac{130}{1+\gamma^2} \dashint Z_1 + \frac{260}{\pi(1+\gamma^2)} \int_0^\infty Z_1 \, d \gamma .
    \end{aligned}
\end{equation}
This can then be rewritten:
\begin{equation}\label{eq:barHfin}
    \begin{aligned}
    \calA^2(13 (L_\theta \Gamma)_1 + (L_\theta \Gamma)_2)  = 4 \gamma\p_\gamma W + 12 W+  \frac{130}{1+\gamma^2}\bar Z_1 - \frac{130}{1+\gamma^2} \dashint \bar Z_1 + \frac{260}{\pi(1+\gamma^2)} \int_0^\infty Z_1 \, d \gamma.
    \end{aligned}
\end{equation}
We also have the expression involving $Z_1$:
\begin{align}\label{eq:decayG1}
    \calA (L_\theta \Gamma)_1 = 4 \gamma \p_\gamma Z_1 + 27 Z_1 - H.
\end{align}
Integrating the above expression, we obtain:
\begin{equation}\label{eq:witH}
   \int_0^\infty \calA (L_\theta \Gamma)_1 d\gamma =  23 \int_0^\infty Z_1 d\gamma - \int_0^\infty H d \gamma.
\end{equation}
Since we are restricting to $\calS$, we have that $\int_0^\infty H d\gamma= 0$, and therefore:
\begin{equation}\label{eq:Gdec}
   \int_0^\infty \calA (L_\theta \Gamma)_1 d\gamma =  23 \int_0^\infty Z_1 d\gamma.
\end{equation}
Integrating equation~\eqref{eq:barHfin}, we also have
\begin{equation}
    \begin{aligned}
     \int_0^\infty \calA^2(13 (L_\theta \Gamma)_1 + (L_\theta \Gamma)_2) \, d\gamma =  8\int_0^\infty W \, d\gamma + 130 \int_0^\infty Z_1 \, d \gamma.
    \end{aligned}
\end{equation}
We are now ready for the $L^2$ estimates. Recall the system:
\begin{align}
    &\calA^2(13 (L_\theta \Gamma)_1 + (L_\theta \Gamma)_2) = 4 \gamma\p_\gamma W + 12 W+  \frac{130}{1+\gamma^2}\bar Z_1 - \frac{130}{1+\gamma^2} \dashint \bar Z_1 + \frac{260}{\pi(1+\gamma^2)} \int_0^\infty Z_1 \, d \gamma, \label{eq:one}\\
    &\gamma^2 \calA^2((L_\theta \Gamma)_1) = 4 \gamma \p_\gamma \bar Z_1 + 27 \bar Z_1 - \gamma^2 W. \label{eq:two}
\end{align}
(the last equation has been obtained from the equation for $Z_1$ by subtracting the value at $\gamma = 0$ and using the definition of the operator $\calA$).

We multiply equation~\eqref{eq:one} by $(1+\gamma^2)W$ and we sum it to equation~\eqref{eq:two} multiplied by $130 \calA^2 (\Gamma_1)$, and integrate. We have, since the cross-terms cancel:
\begin{equation}
\begin{aligned}
    &\int_0^\infty \Big\{ (1+\gamma^2)  \calA^2(13 (L_\theta \Gamma)_1 + (L_\theta \Gamma)_2)W + 130 \gamma^2 \calA^2((L_\theta \Gamma_1)) \calA^2( \Gamma_1) \Big\} \, d\gamma \\
    &= \int_0^\infty (10+6\gamma^2) W^2 + 23 \times 130 \gamma^{-2}\bar Z_1^2  \, d\gamma\\
    &- 130 \int_0^\infty W \, d\gamma \dashint \bar Z_1 + \frac{260}{\pi}  \int_0^\infty W \, d\gamma  \int_0^\infty Z_1 \, d \gamma.
\end{aligned}
\end{equation}
We are left with the penultimate term, which we are going to bound by:
\begin{align*}
    &\frac{260}{\pi} \Big|\int_0^\infty W \, d\gamma \Big| \ \Big|\int_0^\infty \frac{\bar Z_1}{1+\gamma^2} \, d\gamma \Big| \leq  \frac{260}{\pi} \Big|\int_0^\infty W \, d\gamma \Big| \ \Big( \int_0^\infty \frac{\bar Z^2_1}{\gamma^2} \, d\gamma \Big)^{\frac 12}\Big( \int_0^\infty \frac{\gamma^2}{(1+\gamma^2)^2} \, d\gamma \Big)^{\frac 12}\\
    &= \frac{130}{\sqrt{\pi}}  \Big|\int_0^\infty W \, d\gamma \Big| \ \Big( \int_0^\infty \frac{\bar Z^2_1}{\gamma^2} \, d\gamma \Big)^{\frac 12} \leq \frac {75}{\pi} \Big|\int_0^\infty W \, d\gamma \Big|^2 + 75 \int_0^\infty \frac{\bar Z^2_1}{\gamma^2} \, d\gamma.
\end{align*}
Combining the above estimates, we have
\begin{equation}\label{eq:thetaip1}
\begin{aligned}
    &\int_0^\infty \Big\{ (1+\gamma^2)  \calA^2(13 (L_\theta \Gamma)_1 + (L_\theta \Gamma)_2)W + 130 \gamma^2 \calA^2((L_\theta \Gamma_1)) \calA^2( \Gamma_1) \Big\} \, d\gamma \\
     & \geq \int_0^\infty (10+6\gamma^2) W^2 + (23 \times 130 -75)\gamma^{-2}\bar Z_1^2  \, d\gamma\\
     & - C\Big| \int_0^\infty W \, d\gamma \Big|^2 - C \Big|\int_0^\infty Z_1 d\gamma  \Big|^2.
\end{aligned}
\end{equation}
for some positive constant $C$.

We would now like to add the lower order terms in $\Gamma_1$ and $\Gamma_2$. We have:
\begin{align*}
     &\calA((L_\theta \Gamma)_1) =  4 \gamma \p_\gamma Z_1 + 27 Z_1 - H,\\
     &\calA(13 (L_\theta \Gamma)_1 + (L_\theta \Gamma)_2) = 4 \gamma\p_\gamma H + 4H + 130 Z_1 - \frac{130}{1+\gamma^2} \Big( \bar Z_1 - \dashint \bar Z_1\Big) + \frac{2 \times 130}{\pi \times(1+\gamma^2)} \int_0^\infty Z_1 d\gamma.
\end{align*}
Multiplying the first equation by $130 Z_1$ and the second by $H$ and summing, we deduce the existence of positive constants $c$ and $C$ such that
\begin{align*}
     &\int_0^\infty (Z_1\calA((L_\theta \Gamma)_1) + H \calA(13 (L_\theta \Gamma)_1 + (L_\theta \Gamma)_2)) d\gamma\\
     &\geq c \int_0^\infty  (Z^2_1 + H^2) d \gamma - C \int_0^\infty \frac{\bar Z^2_1}{\gamma^2} d\gamma - C\Big|\int_0^\infty Z_1 d\gamma  \Big|^2.
\end{align*}
Similarly, multiplying by $(1+\gamma^2)^{\frac 38}$ (the sharp power is $\gamma^1$), we have
\begin{equation}\label{eq:thetaip2}
\begin{aligned}
     &\int_0^\infty (1+\gamma^2)^{\frac 38}(Z_1\calA((L_\theta \Gamma)_1) + H \calA(13 (L_\theta \Gamma)_1 + (L_\theta \Gamma)_2)) d\gamma \\
     &\geq c \int_0^\infty  (1+\gamma^2)^{\frac 38}(Z^2_1 + H^2) d \gamma - C \int_0^\infty \frac{\bar Z^2_1}{\gamma^2} d\gamma - C\Big|\int_0^\infty Z_1 d\gamma  \Big|^2.
\end{aligned}
\end{equation}
We finally look at the equations for $\Gamma_1$ and $\Gamma_2$:
\begin{equation}
\begin{aligned}
    &(L_\theta \Gamma)_1 = 4 \gamma \p_\gamma \Gamma_1 + 6 \Gamma_1 - \Gamma_2,\\
    &(L_\theta \Gamma)_2 = 4 \gamma \p_\gamma \Gamma_2 + 9 \Gamma_2 - \frac{130}{1+\gamma^2}  (\bar \Gamma_1 - \dashint \bar \Gamma_1).
\end{aligned}
\end{equation}
We first focus on the second equation, and multiply it by $(1+\gamma^2)^\beta$, with $\beta \in (1/2,1)$. We obtain
$$
\int_0^\infty \Gamma_1 (L_\theta \Gamma_2)(1+\gamma^2)^\beta d\gamma \geq c \int_0^\infty \Gamma_2^2 (1+\gamma^2)^\beta d\gamma - C \Big|\int_0^\infty (1+\gamma^2)^\beta | \Gamma_2 ||Z_1| d \gamma \Big|
$$

Multiplying the first equation by $(1+\gamma^2)^{\frac 14} \Gamma_1$ and summing it to a constant multiple of the second equation times $(1+\gamma^2)^{\frac 14}\Gamma_2$, we obtain positive constants $c$ and $C$ such that:
\begin{equation}\label{eq:thetaip3}
\begin{aligned}
     &\int_0^\infty (1+\gamma^2)^{\frac 14}(\Gamma_1(L_\theta \Gamma)_1 + \Gamma_2(L_\theta \Gamma)_2) d\gamma \geq c \int_0^\infty (1+\gamma^2)^{\frac 14} (\Gamma^2_1 + \Gamma_2^2) d \gamma - C \int_0^\infty (1+\gamma^2)^{\frac 38} Z^2_1  d\gamma.
\end{aligned}
\end{equation}
We finally need to improve the weight in $\gamma$. To that end, we have, for a positive constant $C$,
\begin{equation}\label{eq:loww1}
\begin{aligned}
&\int_0^\infty \Gamma_2 (L_\theta \Gamma)_2 \gamma^{\frac 32} d\gamma = \int_0^\infty \Gamma_2 \Big(4 \gamma \p_\gamma \Gamma_2 + 9 \Gamma_2 - \frac{130}{1+\gamma^2} \Big(\Gamma_1 - \dashint \Gamma_1\Big) \Big) \gamma^{\frac 32} d\gamma\\
&\geq 4 \int_0^\infty \Gamma_2^2 \gamma^{\frac 32} d\gamma - C \int_0^\infty \Gamma_1^2 (1+\gamma^2)^{\frac 14} d\gamma.
\end{aligned}
\end{equation}
Moreover,
\begin{equation}\label{eq:loww2}
\begin{aligned}
&\int_0^\infty \Gamma_1 (L_\theta \Gamma)_1 \gamma^{\frac 32} d\gamma = \int_0^\infty \Gamma_1 \Big(4 \gamma \p_\gamma \Gamma_1 + 6 \Gamma_1 - \Gamma_2\Big) \gamma^{\frac 32} d\gamma \geq  \int_0^\infty \Gamma_1^2 \gamma^{\frac 32} d\gamma - C \int_0^\infty \Gamma_2^2 (1+\gamma^2)^{\frac 34} d\gamma.
\end{aligned}
\end{equation}
Combining~\eqref{eq:decayG1},~\eqref{eq:Gdec},~\eqref{eq:thetaip1},~\eqref{eq:thetaip2},~\eqref{eq:thetaip3},~\eqref{eq:loww1},~\eqref{eq:loww2}, we then readily see that there exists a choice of positive constants $\mathbf{C} = (C_1, C_2, C_3, C_4)$ such that, for all $f \in \calH^{2}_\theta \cap \calS$,
\begin{equation}
    (f, L_\theta f)_{\mathbf{C}} \geq c (f, f)_{\mathbf{C}}.
\end{equation}
This concludes the proof of Claim 2).

\vspace{5pt}

{\bf Proof of Claim 3). } Denote the $L^2$-weighted spaces as follows:
\begin{equation}
    |f|^2_{L^2_{\gamma, w(\gamma)}} := \int_0^\infty (f(\gamma))^2 w(\gamma) d \gamma.
\end{equation}
Let $w_1(\gamma) := (1+\gamma^2)^{3/4}$. We immediately have that
\begin{equation}\label{eq:pos1}
    (f,L_\theta f)_\mathbf{C} \geq c |f|^2_{L^2_{\gamma,w_1(\gamma)}}.
\end{equation}
for some positive constant $c$. Moreover, by interpolation,
\begin{equation}\label{eq:pos2}
    |[\Lambda^k, L_\theta] f|^2_{L^2_{\gamma, w_1(\gamma)}} \leq C( |\Lambda^k f|^2_{L^2_{\gamma, w_1(\gamma)}} +| f|^2_{L^2_{\gamma, w_1(\gamma)}}).
\end{equation}
Now,~\eqref{eq:pos1} and~\eqref{eq:pos2} give that there exists $C_5 > 0$ such that
\begin{equation}
    (f, L_\theta f)_{\calH^k_\theta \cap \calS} \geq c|f|^2_{\calH^k_\theta \cap \calS}
\end{equation}
for a constant $c > 0$.

We will be done once we show the equivalence of norms $|f|_{\calH^k_\theta \cap \calS}$ and $|f|_{\calH^k_\theta}$. First,~\eqref{eq:pos1} and interpolation give that
\begin{equation}
    |f|_{\calH^k_\theta \cap \calS} \gtrsim |f|_{\calH^k_\theta}.
\end{equation}
In addition, Sobolev embedding shows that
$
    (f,f)_\mathbf{C} \leq C |f|^2_{\calH^3_\theta},
$
which then immediately yields
$
     |f|_{\calH^k_\theta \cap \calS} \lesssim |f|_{\calH^k_\theta}.
$
This concludes the proof of {\bf Claim 3)}.

\vspace{5pt}

{\bf Proof of Claim 4). } The semi-group estimate follows readily from {\bf Claims 1--3}.
\end{proof}

We are now going to use the long-term decay properties of the semi-group induced by $L_\theta$ to construct a profile.

\begin{lemma}[Existence of a profile for the angular equation]\label{lem:angularprofile}

There exists $\Gamma_* = ((\Gamma_*)_1, (\Gamma_{*})_2) \in (\calH_\theta^\infty)^2$ which satisfies the steady profile equation:
\begin{align}\label{eq:prof1}
    &2 D_\theta (\Gamma_*)_1 + 6 (\Gamma_*)_1 = (\Gamma_*)_2,\\
    &2 D_\theta (\Gamma_*)_2 + 9 (\Gamma_*)_2 = A_* \cos^2 \theta \Big((\Gamma_*)_1 - \dashint (\Gamma_*)_1\Big).\label{eq:prof2}
\end{align}

\end{lemma}

\begin{proof}
First, we note that the semi-group generated by $- L_\theta$ is strongly continuous\footnote{This follows from showing that $-L_\theta$ is maximal dissipative when restricted to $\calS$, which follows from a reasoning similar to the one appearing in the proof of Lemma~\ref{lem:inv}}. We construct the profile by a long-time limit of the evolution $\Gamma(t) := \exp(- t L_\theta) \Gamma_0.$ Here, $\Gamma_0 := (\Gamma_{0,1}, \Gamma_{0,2})$ belongs to $\calH_\theta^\infty$, and moreover it is chosen such that $\bar \Gamma_{0,1}$ and $\bar \Gamma_{0,2}$ are non-positive everywhere. It is straightforward to show that this non-positivity property is propagated by the time evolution.

We look at the equation satisfied by $\p_t \Gamma$, where $\Gamma$ is as above. First, we notice that $\p_t \Gamma|_{t = 0} = - L_\theta \Gamma_0 \in \mathcal{S}$, so that $\p_t \Gamma$ decays exponentially due to the semi-group estimate, and therefore the limit $\Gamma_* := \lim_{t \to \infty} \Gamma(t)$ is well-defined and belongs to $\calH^\infty_\theta$. Moreover, the convergence of both $\p_t \Gamma$ to $0$ and $\Gamma$ to $\Gamma_*$ is strong in all $\calH^k_\theta$. Hence, taking limits,
\begin{equation}\label{eq:profp}
\begin{aligned}
    &2 D_\theta (\Gamma_*)_1 + 6 (\Gamma_*)_1 = (\Gamma_*)_2,\\
    &2 D_\theta (\Gamma_*)_2 + 9 (\Gamma_*)_2 = 130 \cos^2 \theta \Big((\Gamma_*)_1 - \dashint (\Gamma_*)_1\Big).
\end{aligned}
\end{equation}
 We finally need to show that $\dashint (\Gamma_*)_1 \neq 0$, which in particular shows that $(\Gamma_*)_1$ is nontrivial. Suppose the contrary, then necessarily, from system~\eqref{eq:profp}, we have that $(\Gamma_*)_1|_{\gamma=0} = 0$, which implies that $\dashint (\bar \Gamma_*)_1 = 0$, but we know that $(\bar \Gamma_*)_1\leq 0$, so this would imply that $(\Gamma_*)_1$ is a constant. This, in turn, would imply that $(\Gamma_*)_2$ is a constant. However, $\int_0^\infty \calA(13 (\Gamma_*)_1 + (\Gamma_*)_2) d \gamma \neq 0$. Contradiction.
\end{proof}

We finally note the following corollary, which gives that $L_\theta$ is non-negative without the restriction to $\calS$.

\begin{cor}\label{cor:thetanonneg}
For $\Gamma \in D(L_\theta)$, let $ \PP_{\calS} \Gamma := \Gamma - c(\Gamma) \Gamma_*$, where $c(\Gamma)$ is chosen such that $ \Gamma - c(\Gamma) \Gamma_* \in \calS$. Then, the inner product defined by 
\begin{equation}
    (\Gamma, \Xi)_{\tilde{\mathbf{C}}}:= (\PP_\calS \Gamma, \PP_\calS \Xi)_{\mathbf{C}} + c(\Gamma) c(\Xi)
\end{equation}
satisfies, for all $\Gamma \in D(L_\theta)$,
\begin{equation}\label{eq:nonneg}
    (\Gamma, L_\theta \Gamma)_{\calS} \geq 0.
\end{equation}
\end{cor}

\begin{proof}
The proof follows from the fact that $c(L_\theta \Gamma) = 0$, and that $\PP_\calS$ commutes with $L_\theta$.
\end{proof}

\section{Analysis of the linearized operator}\label{sec:linearanalysis}

In this Section, we analyze the full ``local'' operator $\scrloc$ (and its modified counterpart\footnote{This modification will be introduced later in order to avoid derivative loss at the top order.}, $\tscrloc$) in order to show dissipativity under an appropriate inner product. First, we proceed to define the main linear operators and their local counterparts.

\subsection{Local--nonlocal splitting of the system}
We define the main ``full'' operator
\begin{equation}
    \scrl(f) := \calN(f, f_*) + \calN(f_*, f).
\end{equation}
We then proceed to split the operator into a ``local'' part and a ``nonlocal'' part. 
\begin{definition}
We define the ``local'' nonlinear term: $              \underline{\calN}(f,f)$ where $\underline{\calN}(f,f)$ has the same form as $\calN(f,f)$, however all instances of $\ellud(\Omega_0)$ are replaced with $\lna(\Omega_0)$. Here, $\lna$ is defined as:
\begin{equation}
   \lna(q) :=  -\int_0^z  \frac 1 {z' \pi} \int_0^{2\pi}\sin(2\mu) q(z', \theta) d \theta d z'.
\end{equation}
\end{definition}
We define
\begin{align}\label{eq:scrlocdef}
    & \scrloc (f) := \scrl(f) - \scrnonloc (f),
\end{align}
where
\begin{equation}\label{eq:defnonloc}
\begin{aligned}
    &(\scrnonloc (f))_1 := \frac 12 \lfar(\boldom_0[f]) D_\theta \boldom_0[f_*],\\
    &(\scrnonloc (f))_2 := - \frac 1 4 \lfar(\boldom_0[f]) \boldP_0[f_*],\\
    &(\scrnonloc (f))_3 := \PP_0^\perp\Big\{\frac 12 \lfar(\boldom_0[f]) D_\theta \boldP_1[f_*]  - \frac 12 \cos(2 \theta) \lfar(\boldom_0[f]) D_z \boldP_0[f_*] -\frac 14 \lfar(\boldom_0[f])\boldP_1[f_*] \Big\},
\end{aligned}
\end{equation}
and
\begin{equation}\label{def:ltotal}
    \lfar(h) :=  \int_0^\infty  \frac 1 { z' \pi} \int_0^{2\pi}\sin(2\mu) h(z', \mu) d \mu dz'.
\end{equation}
\begin{remark}
The remained of this section will be focused on showing that the operator $\scrloc$ is accretive.
\end{remark}
The operator $\scrloc$ can be written as
$$
\scrloc (f) = \underline{\calN}(f, f_*) + \underline{\calN} (f_*, f) + \calW(f),
$$
where $\mathcal{W}$ is a linear operator which takes the form (since $\lfar(\boldom_0[f_*])=4$)
$$
\calW(f) = (2 D_\theta f_1, - f_2, 2 D_\theta f_3 - f_3 - 2 \cos(2\theta) D_z f_2 ).
$$
We notice that, by the same exact reasoning as in Section~\ref{sec:good}, the transformation $\calB$ of Lemma~\ref{lem:good} satisfies favorable commutation properties with $\underline{\calN}$. More precisely, we have the following
\begin{lemma}[``Good'' transformation for $\underline{\calN}$]\label{lem:goodloc}
We have that $\calB \underline{\calN}(f,f) = \underline{\calM} (\calB f, \calB f)$, where $\underline{\calM}$ is defined as $\calM$ in Lemma~\eqref{lem:good}, where all instances of $\ellg(f_1)$ are replaced with $\lgl(f_1)$, and $\lgl(h)$ is defined as
\begin{equation}
   \lgl(F) = -2 \int_0^z \frac{1}{z'} \dashint F dz'.
\end{equation}
\end{lemma}
By Section~\ref{sec:good} and the above lemma, we have
$
\mathcal{B} \calN(f,f) = \calM(\calB f, \calB f),
$
and similarly
$
\mathcal{B} \underline{\calN}(f,f) = {\underline{\calM}}(\calB f, \calB f),
$
so that, by polarization, 
$$
\mathcal{B}(\underline{\calN}(f, f_*) + \underline{\calN} (f_*, f)) = \mathcal{B}(\underline{\calN}(f+f_*, f+f_*) - \underline{\calN} (f, f) - \underline{\calN} (f_*, f_*) ) =  {\underline{\calM}}(\calB f_*, \calB f) + {\underline{\calM}}(\calB f, \calB f_*)
$$
We proceed to calculate $\calB \calW$.

\begin{lemma}\label{lem:calwcalc}
We have the following expression for $\calB \calW$:
$$
\calB \calW f = (2 D_\theta  (\calB f)_1 + 4 (\calB f)_1, - (\calB f)_2, 2 D_\theta (\calB f)_3 + 7 (\calB f)_3) .
$$
\end{lemma}

\begin{proof}[Proof of Lemma~\ref{lem:calwcalc}]
We let $g = \calW f$, and calculate $\calB g$. We note that
$$
[\cos^2 \theta \p_\theta, D_\theta] = 2 \cos^2 \theta \p_\theta,
$$
so that
\begin{align*}
    \calB g_1 = \cos^2 \theta \p_\theta (2 D_\theta f_1) = 2 D_\theta \calB f_1 + 4 \calB f_1. 
\end{align*}
Moreover, $\calB g_2 = \calW f_2 = - \calB f_2$.

Finally,
\begin{align*}
    (\calB g)_3 = \cos^2 \theta \p_\theta\Big(- \frac 12 \sin(2\theta) D_z f_2 + \cos^2(\theta) \p_\theta \Big(2 D_\theta f_3 - f_3 - 2\cos(2\theta) D_z f_2 \Big) \Big).
\end{align*}
We calculate 
\begin{align*}
    \cos^2 \theta \p_\theta\Big(- \frac 12 \sin(2\theta) D_z f_2 + \cos^2(\theta) \p_\theta \Big(2 D_\theta f_3 - f_3 - 2\cos(2\theta) D_z f_2 \Big) \Big) = 2 D_\theta (\calB f)_3 + 7 (\calB f)_3
\end{align*}
\end{proof}
We define
\begin{equation}
    \calLloc(f) := {\underline{\calM}}(\calB f_*, \calB f) + {\underline{\calM}}(\calB f, \calB f_*) + \calB \calW(f).
\end{equation}
With these preliminary calculations concluded, we proceed to show the existence of an inner product which makes $\scrloc$ a dissipative operator. We will do this in several steps. First, in Section~\ref{sec:gooddissip}, we show that there is an inner product which makes $\calLloc$ dissipative. We will then upgrade the dissipativity to $\scrloc$ in Sections~\ref{sec:scrloc1}--\ref{sec:scrloc3}.

\subsection{Coercivity of \texorpdfstring{$\underline{\calL}$}{L}}\label{sec:gooddissip}

We first recall the form of $\calLloc$. Let $g^* = \calB f^*$, then
\begin{align*}
    &(\calLloc g)_1 := \frac 12 \ul(g^*_1) D_\theta g_1 + \ul(g^*_1) g_1+ 2 g^*_2 g_3+\frac 12 \ul(g_1) D_\theta g^*_1 + \ul(g_1) g^*_1+ 2 g_2 g^*_3 + 2 D_\theta g_1 + 4 g_1, \\
    &(\calLloc g)_2 = \frac 1 4 \ul(g^*_1) g_2+\frac 1 4 \ul(g_1) g^*_2 - g_2,\\
    &(\calLloc g)_3 = \frac 12 \ul(g^*_1) D_\theta g_3 + \frac 7 4 \ul(g^*_1) g_3 + \frac 12 \cos^2 \theta g^*_2 \Big(g_1- \dashint g_1\Big)\\
    &\qquad +\frac 12 \ul(g_1) D_\theta g^*_3 + \frac 7 4 \ul(g_1) g^*_3 + \frac 12 \cos^2 \theta g_2 \Big(g^*_1- \dashint g^*_1\Big) + 2 D_\theta g_3 + 7 g_3 .
\end{align*}
We have:
\begin{align*}
    &(\calLloc g)_1 := \frac{2}{1+z} D_\theta g_1 + \frac{4}{1+z} g_1  + \frac{2B_*}{1+z}g_3\\
    & \qquad +\frac{z}{(1+z)^2} \ul(g_1) \Big(\frac 12 D_\theta \Gamma_1^* + \Gamma_1^* \Big) - \frac{1}{B_*} g_2 \frac{z}{(1+z)^2} \Gamma_2^*.  \\
    &(\calLloc g)_2 = - \frac 1 {1+z} g_2 + \frac{B_*}4 \frac 1{1+z} \ul(g_1) ,\\
    &(\calLloc g)_3 = \frac 2{1+z} D_\theta g_3 + \frac{7}{1+z} g_3 + \frac 12 \frac{B_*}{1+z} \cos^2 \theta \Big(g_1- \dashint g_1\Big)\\
    &\qquad - \frac{1}{2B_*}\frac{z}{(1+z)^2} \ul(g_1)\Big( \frac 12  D_\theta \Gamma_2^* + \frac 7 4  \Gamma_2^*\Big)  + \frac 12 \cos^2 \theta g_2 \frac{z}{(1+z)^2}\Big(\Gamma^*_1-2\Big).
\end{align*}

 Whenever $f, g$ are vectors with 3 components, we set\footnote{Compare with the definition in~\eqref{eq:definn}}:
\begin{equation}\label{eq:redefinn}
     (f,g)_{\mathbf{C}} := ((f_1, f_3), (g_1, g_3))_{\mathbf{C}}.
 \end{equation}
 Note that the angular operator is acting on $\Gamma_1$ and $\Gamma_2$, which in our notation are resp. in the first and third component. We have the following Lemma.
\begin{lemma}\label{lem:dissip}
Let $k \geq 4$. Let $\dbold = (D_1, D_2, D_3, D_4)$ be a string of constants.
Let the inner product\footnote{Here, $(\cdot, \cdot)_{\tilde{\mathbf{C}}}$ is understood in the sense of display~\eqref{eq:redefinn}. \vspace{2pt}}
\begin{align*}
(f,g)_{\dbold} &:= D_1 (\p_z f|_{z=0}, \p_z g|_{z=0})_{\tilde{\mathbf{C}}}+D_2 \int_{0}^a  \Big( f_2 g_2 + \Big(\frac{f}{z} - \p_z f|_{z=0},\frac{g}{z}- \p_z g|_{z=0} \Big)_{\tilde{\mathbf{C}}} \Big)  z^{-2}dz\\
&+ D_3 \int_{0}^\infty (f_2 g_2 + (f,g)_{\tilde{\mathbf{C}}}) z^{-m}\chi(z) dz + D_4 \int_{0}^\infty (f_2 g_2 +(f,g)_{\tilde{\mathbf{C}}}) dz,
\end{align*}
where $a > 0$, $m \geq 4$, and $\chi(z)$ is a positive, smooth cut-off function which is identically $1$ on $[1, \infty)$, and vanishes on $[0,a/2]$. Suppose moreover that $f$ vanishes linearly at $z = 0$, and that\footnote{That is, $ \int_0^\infty (13 \bar \Xi_1(\theta) + \bar \Xi_2(\theta)) \gamma^{-2} d \gamma = 0$, where $\Xi = (\p_z f_1, \p_z f_3)|_{z = 0}$.}
\begin{equation}\label{eq:integralcond}
    (\p_z f_1, \p_z f_3)|_{z = 0} \in \mathcal{S} 
\end{equation}
Then, there exists a string of strictly positive constants $\mathbf{D}$, as well as $a > 0$, $m > 0$ such that the following holds true for a positive constant $c$:
\begin{equation}
    (f, f + z\p_z f + \calLloc f)_\dbold \geq c (f,f)_{\dbold}.
\end{equation}

\end{lemma}

\begin{proof}[Proof of Lemma~\ref{lem:dissip}]
Note that, when calculated at $z=0$, $\p_z(f + z\p_z f + \calLloc f)$ reduces to $L_\theta$ on the first and third component. Due to our assumptions and Lemma~\ref{lem:angulardecay}, the claim is immediate with $D_1 > 0$ and $D_2 = D_3 = D_4 = 0$. We let, whenever $f$ has three components, $L_\theta(f) := L_\theta(f_1, f_3)$, and we compute\footnote{Recall that $L_\theta$ acts on 2-vectors.}:
\begin{align*}
    &\frac 1 z(f + z\p_z f + \calLloc f)\\
    &\qquad = z\p_z (f/z) + \Big( \frac 1 {z(z+1)} (L_\theta(f))_1 + E_1, (1+ z/(1+z)) \frac{f_2}z + E_2, \frac 1 {z(z+1)} (L_\theta(f))_2+E_3\Big).
\end{align*}
Here,
\begin{align*}
    &E_1 =\frac{1}{(1+z)^2} \ul(f_1) \Big(\frac 12 D_\theta \Gamma_1^* + \Gamma_1^* \Big) - \frac{1}{B_*} \frac{f_2}{z} \frac{z}{(1+z)^2} \Gamma_2^*,\\
    &E_2= \frac{B_*}4 \frac 1{z(1+z)} \ul(f_1),\\
    &E_3=- \frac{1}{2B_*}\frac{1}{(1+z)^2} \ul(f_1)\Big( \frac 12  D_\theta \Gamma_2^* + \frac 7 4  \Gamma_2^*\Big)  + \frac 12 \cos^2 \theta \frac{f_2}{z} \frac{z}{(1+z)^2}\Big(\Gamma^*_1-2\Big).
\end{align*}
We choose $a>0$ sufficiently small. Using the positivity of $(\cdot, \cdot)_{\tilde{\mathbf{C}}}$ from Corollary~\ref{cor:thetanonneg}, since the error terms $E_1$, $E_2$, $E_3$ are perturbative, we obtain the claim with $D_1 >0$, $D_2 >0$, $D_3 = D_4 = 0$.

We then consider the weight function $w_1(z) := \chi(z) z^{-m}$, for $m$ large, and $\chi$ a positive cut-off function which transitions from $0$ to $1$ in the interval $[0, a/2]$, and remains equal to $1$ afterwards. We use the weight $w_1(z)$ in a weighted $L^2$ estimate. The claim with $D_1 > 0$, $D_2 >0$, $D_3=D_4 = 0$, in conjunction with the positive commutator term arising from the transport term $z\p_z$, then gives us the existence of positive constants $\tilde D_1$, $\tilde D_2$, $\tilde D_3$ and $c$ such that, letting again $h:= f+z\p_zf+\calLloc f$:
\begin{equation}\label{eq:aa}
\begin{aligned}
(f,h)_{\tilde{\mathbf{D}}} \geq c (f,f)_{\tilde{\mathbf{D}}}
\end{aligned}
\end{equation}
Here, $\tilde{\mathbf{D}} = (\tilde D_1, \tilde D_2, \tilde D_3, 0)$. We conclude by a weighted $L^2$ estimate, using the weight $\chi_1(z)$. Here, $\chi_1$ is a positive, smooth cut-off function which transitions from $0$ to $1$ in the region $[M, M+1]$, where $M$ is chosen sufficiently large.
\end{proof}

\subsection{Coercivity of \texorpdfstring{$\scrloc$}{L} at low order}\label{sec:scrloc1}

With the above lemma in hand, we are ready to state and prove the main lemma of this section. We first define the space on which the local operator $\scrloc$ is accretive.
\begin{definition}\label{def:scripts}
The space $\mathscr{S}$ is defined as\footnote{The space $\calH^k$ is introduced later in Definition~\ref{def:allinner}.}
\begin{align}
    &\mathscr{S}:= \{f \in \calH^k: (\p_z (\calB f)_1, \p_z (\calB f)_3)|_{z = 0} \in \mathcal{S} \quad \text{and} \quad f|_{z = 0} = 0 \}.
\end{align}
\end{definition}
We have the following Lemma.
\begin{lemma}[Dissipativity of $\scrloc (f)$ at low order]\label{lem:scrlocp}
Let $(f,g)_{\dbold}$ be the inner product from Lemma~\ref{lem:dissip}. We let $\boldb := (B_1, B_2, \ldots, B_7)$ be a string of positive constant. Define the inner product $(\cdot, \cdot)_{\boldb}$ as follows:
\begin{equation}\label{eq:binnerdef}
    \begin{aligned}
    (f,g)_{\boldb} := &B_1 (\calB f, \calB g)_{\boldsymbol{D}} + B_2 \int_0^\infty \int_{0}^{\frac \pi 2}f_1 g_1 d\theta  \frac{(1+z)^2}{z^2} dz  \\
    +& B_3 \int_0^\infty  f_2 g_2 \frac{(1+z)^2}{z^2} dz +  B_4\int_0^\infty (z \p_z f_2) (z\p_z g_2) \frac{(1+z)^2}{z^2} dz\\
    +& B_5 \int_0^\infty \PP_0 f_3 \PP_0 g_3  \frac{(1+z)^2}{z^2} dz  + B_6  \int_0^\infty \int_{0}^{\frac \pi 2}  f_3  g_3 d\theta \frac{(1+z)^2}{z^2} dz\\ 
    +& B_7 \int_0^\infty \int_{0}^{\frac \pi 2} \p_\theta f_1 \p_\theta g_1 (\cos \theta)^{\frac 12}d\theta \frac{(1+z)^2}{z^2} dz + B_8 \int_0^\infty \int_{0}^{\frac \pi 2} \p_\theta f_3 \p_\theta g_3 (\cos \theta)^{\frac 12}d\theta \frac{(1+z)^2}{z^2} dz.
    \end{aligned}
\end{equation}
Then, there exist a value of $\boldb$, and a positive constant $c$ such that, for any $f \in \mathscr{S}$,
\begin{equation}
    (f + z\p_z f + \scrloc f, f)_\boldb \geq c ( f, f)_\boldb.
\end{equation}

\end{lemma}

\begin{proof}[Proof of Lemma~\ref{lem:scrlocp}]

As a first step, let us rewrite the operator $\scrloc f$ using the ``transformed'' quantities $\calB f$. We have the identity, which follows from integration by parts: $\lna(f_1) = \lgl((\calB f)_1)$. Using this observation, the operator $\scrloc$ can be rewritten as follows\footnote{Recall that $\gamma = \tan \theta$.},
\begin{align}
    &(\scrloc (f))_1 = \frac 2{1+z} D_\theta f_1 - 2 f^*_2 \int_0^\gamma (\calB f)_3(z,\sigma) d\sigma +  \frac 12 \lgl((\calB f)_1) D_\theta f_1^* - 2f_2  \int_0^\gamma (\calB f^*)_3(z,\sigma) d\sigma, \label{eq:lbarloc1}\\
    &(\scrloc (f))_2 = -\frac1{1+z} f_2 - \frac 1 4 f_2^* \lgl((\calB f)_1), \\
     &(\scrloc (f))_3 = \PP_0^\perp\Big\{\frac1{1+z}\Big( 2D_\theta f_3  -2  \cos(2 \theta) D_z f_2 -f_3\Big) \\
    & \qquad - \frac{f_2^*}{4} \sin^2 \theta D_z(\lgl(\calB f)_1)) + \frac 12  \Big(2 \tan \theta \tilde \Psi_1^* + \p_\theta  \tilde \Psi_1^* \Big)f_2\\
    &\qquad +\frac 12 \lgl((\calB f)_1) D_\theta f_3^*  - \frac 12 \cos(2 \theta) \lgl((\calB f)_1) D_z f_2^* -\frac 14 \lgl((\calB f)_1)f_3^* \\
    & \qquad - \frac{f_2}{4} \sin^2 \theta D_z\lgl((\calB f^*)_1) + \frac 12 \Big(2 \tan \theta \tilde{\boldpsi}_1[f] + \p_\theta \tilde{\boldpsi}_1[f] \Big)f_2^*\Big\} \label{eq:lbarloc3}
\end{align}
 We let $g := f + z\p_z f + \scrloc f$.
 
 We first estimate the terms multiplying $B_2$ in the definition of $(f,g)_{\boldsymbol{B}}$:
\begin{align}
    &\int_0^\infty \int_{0}^{\frac \pi 2}f_1 g_1 d\theta  \frac{(1+z)^2}{z^2} dz
    \geq \int_0^\infty \int_{0}^{\frac \pi 2}f_1(f_1 + z\p_z f_1) d\theta  \frac{(1+z)^2}{z^2} dzd\theta -|(a_1)| - |(a_2)| - |(a_3)| - |(a_4)|.
\end{align}
Here, $(a_1), (a_2), (a_3), (a_4)$ correspond to the four terms on the RHS of~\eqref{eq:lbarloc1}. To estimate the transport term in $z$, note that $\frac{(1+z)^2}{z^2} - \frac12 \p_z\Big(\frac{(1+z)^2}{z} \Big) \geq \frac 12 \frac{(1+z)^2}{z^2}.$
Before proceeding, it is important to notice that, for some positive constant $c$,
\begin{equation*}
    (\calB f, \calB f)_{\boldsymbol{D}} \geq c \int_0^\infty \int_{0}^{\frac \pi 2} (((\calB f)_1)^2 + ((\calB f)_3)^2 ) (\cos \theta)^{- \frac 72} d \theta dz.
\end{equation*}
We also notice that $D_\theta f_1 = 2 \tan \theta (\calB f)_1$. This implies 
\begin{equation}\label{eq:a1estimate}
\begin{aligned}
    |(a_1)| &\leq  C \int_0^\infty \int_{0}^{\frac \pi 2} ((\calB f)_1)^2 \frac{d\theta}{(\cos(\theta))^{2}}  \frac{(1+z)^2}{z^2} dz +\frac 1 {10}\int_0^\infty \int_{0}^{\frac \pi 2}f^2_1  d\theta  \frac{(1+z)^2}{z^2} dz \\
    &\leq C (\calB f, \calB f)_{\boldsymbol{D}} + \frac 1 {10}\int_0^\infty \int_{0}^{\frac \pi 2}f^2_1  d\theta  \frac{(1+z)^2}{z^2} dz.
\end{aligned}
\end{equation}
We have, using the Cauchy--Schwarz inequality, $|\int_0^\theta (\calB f)_3 \frac{d\theta}{\cos^2 \theta}|^2 \leq C \int_0^{\frac \pi 2} (\calB f)_3^2 (\cos \theta)^{- \frac 72} d\theta$. Therefore, $(a_2)$ satisfies the same bound as $(a_1)$. 

Moreover, using the Cauchy--Schwarz inequality,
\begin{equation}
\begin{aligned}
    |(a_3)| &\leq \Big|\int_0^\infty \int_{0}^{\frac \pi 2}f_1 \Big( \frac 12 \lgl((\calB f)_1) D_\theta f_1^* \Big) d\theta  \frac{(1+z)^2}{z^2} dz \Big| \\
    &\leq \frac{1}{10}\int_0^\infty \int_{0}^{\frac \pi 2}f^2_1  d\theta  \frac{(1+z)^2}{z^2} dz + C\int_0^\infty \int_{0}^{\frac \pi 2} (\lgl((\calB f)_1))^2 (D_\theta \Omega^*_0)^2 d\theta  \frac{(1+z)^2}{z^2} dz \\
    &\leq \frac{1}{10}\int_0^\infty \int_{0}^{\frac \pi 2}f^2_1  d\theta  \frac{(1+z)^2}{z^2} dz + C \int_0^\infty \int_{0}^{\frac \pi 2} ((\calB f)_1)^2 d \theta  \frac{(1+z)^2}{z^2} dz.
\end{aligned}
\end{equation}
Similarly, since $(\calB f^*)_3$ is sufficiently regular,
\begin{equation}
\begin{aligned}
    (a_4)&\leq \Big|\int_0^\infty \int_{0}^{\frac \pi 2}f_1 \Big( 2f_2  \int_0^\theta (\calB f^*)_3(z,\theta') \frac{d\theta'}{\cos^2\theta'}\Big) d\theta  \frac{(1+z)^2}{z^2} dz \Big| \\
    &\leq \frac{1}{10}\int_0^\infty \int_{0}^{\frac \pi 2}f^2_1  d\theta  \frac{(1+z)^2}{z^2} dz + C \int_0^\infty \int_{0}^{\frac \pi 2} f^2_2  \frac{(1+z)^2}{z^2} dz.
\end{aligned}
\end{equation}
This concludes the estimates for the term in $B_2$. We turn our attention to the term multiplying $B_3$. We have
\begin{equation}
\begin{aligned}
    \int_0^\infty  f_2 \Big(\frac{z}{1+z} f_2 + z\p_z f_2\Big)  \frac{(1+z)^2}{z^2} dz = \frac12 \int_0^\infty  f^2_2  \frac{(1+z)^2}{z^2} dz.
\end{aligned}
\end{equation}
Moreover,
\begin{equation}
     \Big|\int_0^\infty  f_2  f_2^*\lgl((\calB f)_1) \frac{(1+z)^2}{z^2} dz\Big| \leq \frac{1}{10} \int_0^\infty  f^2_2  \frac{(1+z)^2}{z^2} dz + C \int_0^\infty \int_{0}^{\frac \pi 2} ((\calB f)_1)^2 d\theta  \frac{(1+z)^2}{z^2} dz.
\end{equation}
The term in $B_4$ is estimated similarly.

Concerning the term in $B_5$, we have $\dashint \big(f_3 + z\p_z f_3 + (\scrloc (f))_3 \big) d\theta= \dashint f_3 d\theta+ z\p_z \dashint f_3 d\theta,$ which implies 
\begin{equation}
    \int_0^\infty\PP_0 f_3 \PP_0 g_3 \frac{(1+z)^2}{z^2} dz \geq \frac12 \int_0^\infty(\PP_0 f_3)^2 \frac{(1+z)^2}{z^2} dz.
\end{equation}
We turn to the term in $B_6$.  Let us denote by $q$ the terms inside the curly brackets in the equation for $(\scrloc f)_3$. We have:
\begin{align}
    & \int_0^\infty \int_{0}^{\frac \pi 2} f_3  g_3 d\theta \frac{(1+z)^2}{z^2} dz\\
    & =\underbrace{ \int_0^\infty \int_{0}^{\frac \pi 2}  (f^2_3 + f_3z\p_z f_3 +f_3 q) d\theta \frac{(1+z)^2}{z^2} dz}_{(c)} -\frac 2 \pi \underbrace{\int_0^\infty \PP_0 f_3  \PP_0 q  \frac{(1+z)^2}{z^2} dz}_{(d)}.
\end{align}
We denote $(c_1)$, $(c_2)$, $(c_3)$ and $(c_4)$ the four terms arising in $(c)$, each corresponding to one of the four lines in the expression for $(\scrloc f)_3$ (see display~\eqref{eq:lbarloc3}). We proceed to estimate these terms.
\begin{enumerate}
    \item[$(c_1)$]  We have the following identity, which follows from the definition of $(\calB f)_3$:
\begin{equation}\label{eq:f3}
    D_\theta f_3  = 2 \tan \theta \int_0^{\theta} (\calB f)_3 \frac{d\theta'}{\cos^2\theta'} - 2 \sin^2 \theta D_z f_2.
\end{equation} 
For the $\tan \theta$ term, we then proceed exactly as in the estimates for term $(a_1)$ in Equation~\eqref{eq:a1estimate}, and we obtain
\begin{equation}
        |(c_1)| \leq C\Big( (\calB f, \calB f)_{\mathbf{D}} + \int_0^\infty (f_2^2 + (D_zf_2)^2) \frac{(1+z)^2}{z^2} dz \Big) + \frac 1{10}  \int_0^\infty \int_0^{\frac \pi 2}  f^2_3 d\theta \frac{(1+z)^2}{z^2} dz.
   \end{equation}
    \item[$(c_2)$] For this term, the regularity properties of the profile imply, after an application of the Cauchy--Schwarz inequality,
    \begin{equation}
        |(c_2)| \leq C\Big( (\calB f, \calB f)_{\mathbf{D}} + \int_0^\infty f_2^2 \frac{(1+z)^2}{z^2} dz \Big) + \frac 1 {10} \int_0^\infty \int_0^{\frac \pi 2}  f^2_3 d\theta \frac{(1+z)^2}{z^2} dz.
    \end{equation}
    \item[$(c_3)$] Similarly, in this case, the Cauchy--Schwarz inequality implies
    \begin{equation}
        |(c_3)| \leq C  (\calB f, \calB f)_{\mathbf{D}} + \frac{1}{10} \int_0^\infty\int_0^{\frac \pi 2}  f^2_3 d\theta \frac{(1+z)^2}{z^2} dz.
    \end{equation}
    \item[$(c_4)$] From a reasoning analogous to the proof of Lemma~\ref{lem:ellipticfirst}, and using the Hardy inequality in Lemma~\ref{Hardy1}, we have in particular that
    \begin{equation}
        \int_0^{\frac \pi 2}\Big(2 \tan \theta \tilde{\boldpsi}_1[f] + \p_\theta \tilde{\boldpsi}_1[f]  \Big)^2 \frac{d\theta}{(\cos(\theta))^{\frac 12}} \leq C  \int_0^{\frac\pi 2} f^2_1 \frac{d\theta}{ (\cos(\theta))^{\frac 12}}.
    \end{equation}
    This implies, for a constant $\eta > 0$,
    \begin{align}
    &|(c_4)| \leq C\Big( (\calB f, \calB f)_{\mathbf{D}} + \int_0^\infty f_2^2 \frac{(1+z)^2}{z^2} dz \Big) + \frac{1}{10} \int_0^\infty  f^2_3 \chi(\theta)d\theta \frac{(1+z)^2}{z^2} dz\\
     & \qquad  +C \int_0^\infty \int_0^{\frac\pi 2} f^2_1 d\theta \frac{(1+z)^2}{z^2} dz.
    \end{align}
\end{enumerate}
Concerning the terms $(d)$, we have
\begin{align}
(\PP_0 q)^2& =\Big|\PP_0 \Big( \frac1{1+z}\Big( 2D_\theta f_3  -2  \cos(2 \theta) D_z f_2 -f_3\Big) - \frac{f^*_2}{4} \sin^2 \theta D_z(\lgl((\calB f)_1))+\frac 12 \lgl((\calB f)_1) D_\theta f_3^* \\
    & \qquad   - \frac 12 \cos(2 \theta) \lgl((\calB f)_1) D_z f_2^* -\frac 14 \lgl((\calB f)_1)f_3^*- \frac{f_2}{4} \sin^2 \theta D_z(\lgl((\calB f^*)_1))\Big)\Big|^2\\
&\leq C(\PP_0 f_3)^2 + C\int_0^{\frac \pi 2} ((\calB f)^2_1+(\calB f)^2_3) \frac{d \theta}{(\cos(\theta))^{2}} + C ((f_2)^2 +(D_z f_2)^2).
\end{align}
Combining the estimates for the terms $(c)$ and $(d)$ concludes the proof for the term in $B_6$.

It remains to estimate the terms in $B_7$ and in $B_8$. The proof for the term in $B_7$ is analogous to the proof for the term in $B_2$, however we need to be careful about the $\theta$-weights. We have 
$$
(\cos \theta)^{\frac 12} \p_ \theta (D_\theta f_1) = 2(\cos \theta)^{\frac 12} \p_ \theta( \tan \theta (\calB f)_1) = \frac 1 {(\cos \theta)^{\frac 32}}(\calB f)_1 + (\cos \theta)^{\frac 12}\tan \theta \p_\theta (\calB f)_1.
$$
This implies, through the Cauchy--Schwarz inequality,
$$
\int_0^{\frac \pi 2} f_1 (\cos \theta)^{\frac 12} \p_ \theta (D_\theta f_1)d \theta \leq \Big(\int_0^{\frac \pi 2} (\cos \theta)^{\frac 12} f_1^2 d \theta \Big)^{\frac 12}\Big(\int_0^{\frac \pi 2} (\cos \theta)^{-\frac 72}\big( ((\calB f)_1)^2+(\cos(\theta) \p_\theta (\calB f)_1)^2\big) d \theta \Big)^{\frac 12},
$$
which is then readily controlled by $(\calB f,\calB f)_{\boldsymbol{D}}$. 

Similarly, we have
$
(\cos \theta)^{\frac 12}\p_\theta \Big(  \int_0^\gamma (\calB f)_3(z,\sigma) d\sigma\Big) = (\cos \theta)^{- \frac 32} (\calB f)_3$, which shows we can run the same argument as in the estimates for the term in $B_2$.

Finally, using the same observations on the angular weights as before, the proof for the term in $B_8$ is analogous to that for the terms in $B_6$, with the only difference in the terms analogous to $(c_4)$ above, for which we use the inequality\footnote{
This inequality follows from the following formula (which is obtained by variation of constants)
\begin{equation}
\p_\theta\big(2 \tan \theta \tilde \Psi+ \p_\theta \tilde\Psi\big) =  \PP^\perp_1 f_1+   \frac{1}{\sin^2(\theta)} \int_\theta^{\pi/2} \sin(2\theta') (\PP^\perp_1 f_1)(z,\theta') d \theta'.
\end{equation}
}
\begin{equation}
        \int_0^{\frac \pi 2}\Big(\p_\theta \Big(2 \tan \theta \tilde{\boldpsi}_1[f] + \p_\theta \tilde{\boldpsi}_1[f]  \Big)\Big)^2 ( \cos \theta)^{\frac 12} d \theta \leq C  \int_0^{\frac\pi 2} f^2_1(\cos \theta)^{\frac 12} d \theta.
\end{equation}
This concludes the proof of the Lemma.
\end{proof}

\subsection{The inner product spaces at high order}\label{sec:spaces}

Having shown coercivity of $\scrloc$ at low order, we proceed to define inner product spaces for high-derivative estimates. Technical facts about these spaces are proved in Appendix~\ref{app:spaces} and Appendix~\ref{sec:spaces2}. We start from the scalar spaces.

\begin{definition}\label{def:highinner}
Let $F,G$ be smooth scalar functions of $z$ and $\theta$. Let $\Upsilon = (1+ \varepsilon_0 z) \p_z$, where $\varepsilon_0 > 0$ is determined later in Lemma~\ref{lem:derhigh}. Recall that $\Lambda = \cos(\theta) \p_\theta$. Let $\iota: \N \to \N$ such that $\iota(0) = 0$ and $\iota(j) = 1$ otherwise. We let the following scalar homogeneous inner product:
\begin{equation}
\begin{aligned}
    &(F,G)_{\dot \calH^k_{(z, z_{\text{top}}),(\theta, \theta_{\text{top}}, w_1(\theta), w_2(\theta))}}  \\
    &:=\sum_{\substack{k_1 + k_2 = k\\ k_1\leq z_{\text{top}}\\ k_2 \leq \theta_{\text{top}}}} \int_0^\infty \int_0^{\frac \pi 2} \Upsilon^{k_1} \Lambda^{k_2} (F/z) \Upsilon^{k_1} \Lambda^{k_2} (G/z)
    (w_1(\theta))^{\iota(k_2)} (w_2(\theta))^{1-\iota(k_2)}d\theta (1+z)^2 d z.
\end{aligned}
\end{equation}
\end{definition}
\begin{remark}
The function $\iota$ serves the purpose of altering the weight at the lowest order in $\theta$.
\end{remark}
We define high-order spaces adapted to each component as follows.
\begin{definition}
Let $\eta_1 = \frac 32$. We let, whenever $F$ and $G$ are sufficiently smooth scalar functions:
\begin{align}
    & (F, G)_{{\dot \calH}^k_{,0}}:=(F,G)_{\dot \calH^{k+1}_{(z, k+1),(\theta, k+1, (\cos(\theta))^{-\eta_1+1}, \cos \theta )}} \\
    &(F, G)_{{\dot \calH}^k_{,1}}:=(F,G)_{\dot \calH^k_{(z, k),(\theta, k, (\cos(\theta))^{-\eta_1}, 1)}}\\
    &(F, G)_{{\dot \calH}^k_{,2}}:=(F ,G)_{\dot \calH^{k}_{(z, k),(\theta, k, 1, 1)}}\\
    &(F, G)_{{\dot \calH}^k_{,3}}:=(F,G)_{\dot \calH^{k+1}_{(z, k),(\theta, k+1, (\cos(\theta))^{-\eta_1}, 1)}}.
\end{align}
\end{definition}
We also define spaces at low order.
\begin{definition}\label{def:lowprod}
We let, whenever $F$ and $G$ are sufficiently smooth scalar functions:
\begin{align}
    &(F,G)_{\calH_{,\text{low};0}}=(F,G)_{\calH_{,\text{low};1}}=(F,G)_{\calH_{,\text{low};3}} :=  (F,G)_{\dot \calH^1_{(z, 0),(\theta,1, (\cos(\theta))^{-\eta_1},1)}},\\
    &(F,G)_{\calH_{,\text{low};2}} := 
    (F,G)_{\dot \calH^0_{(z, 0),(\theta, 0, 1,1)}}.
\end{align}
\end{definition}
We define useful inhomogeneous spaces.
\begin{definition}\label{def:inhomog}
We let, for $i \in \{0,1,2,3\}$, for $F,G$ sufficiently smooth functions,
\begin{equation}
    (F,G)_{\calH^k_{,i}}:= (F,G)_{\dot \calH^k_{,i}} + (F,G)_{\calH_{, \text{low};i}}.
\end{equation}
If $i \in \{0,1,3\}$, we say that $F \in \calH^k_{,i}$ if $(F,F)_{\calH^k_{,i}}$ is finite. If $i = 2$, we require in addition that $F$ is only a function of $z$.
\end{definition}
    
Having defined all the relevant scalar spaces, we turn to the vectorial spaces.

\begin{definition}
We define the vectorial high-order inner product, whenever $c_0, c_1, c_2, c_3$ are  positive constants, $k$ is a positive integer, and $f = (f_1, f_2, f_3)$, $g = (g_1, g_2, g_3)$:
\begin{align}
    (f, g)_{{\calH}^k_{,\text{high}}} :&= c_0 (f_2 + \alpha f_3,g_2 + \alpha g_3)_{\dot \calH^k_{,0}}+ c_1 (f_1,g_1)_{\dot \calH^k_{,1}} + c_2  (f_2 ,g_2)_{\dot \calH^k_{,2}}+c_3 (f_3,g_3)_{\dot \calH^k_{,3}}.
\end{align}
\end{definition}

We are ready to define the main inner product spaces.

\begin{definition}\label{def:allinner}
Recall the definition of the high derivative inner product spaces from Definition~\ref{def:highinner} (and the relevant parameters). For $c_5 > 0$, we define
\begin{align}
    &(f, g)_{\tilde{\calH}^k} := (f, g)_{{\calH}^k_{,\text{high}}} + c_5 (f, g)_{\tilde{\calH}_{,\text{low}}}.  
\end{align}
In addition, we define, for $d_1, d_2, d_3, d_4, d_5> 0$ (replacing each of the instances of $c_i$ with $d_i$ in the definition of $(f, g)_{\tilde{\calH}^k_{,\text{high}}} $):
\begin{align}
    (f, g)_{{\calH}^k} &:= (f, g)_{\calH^k_{,\text{high}}} + d_5 (f, g)_{\calH_{,\text{low}}}.  
\end{align}
Here, we used
\begin{align}
&(f, g)_{\tilde{\calH}_{,\text{low}}} := (f,g)_{\mathbf{B}}\\
&(f, g)_{{\calH}_{,\text{low}}} := (f_2 + \alpha f_3,g_2 + \alpha g_3)_{\calH_{, \text{low};0}}+(f_1,g_1)_{\calH_{, \text{low};1}} +(f_2,g_2)_{\calH_{, \text{low};2}} + (f_3,g_3)_{\calH_{, \text{low};3}}.
\end{align}
We finally denote by $|\cdot|_{\tilde{\calH}^k}$ (resp. $|\cdot|_{\calH^k}$) the norms induced by the above inner products).
\end{definition}
\begin{remark}
    Due to the structure of the equations at top order, we need to consider a space which controls one additional derivative on the combination $f_2 + \alpha f_3$, since that corresponds to the estimates for local existence in the original variables. This is the reason behind introducing the spaces with subscript $0$.
\end{remark}
We finally define the space $\scrH^k$, which does not require vanishing of the second component $f_2$ at $z=0$. To this end, let the projection $\mathbb{Q}$ be defined as
\begin{equation}\label{eq:qprojection}
\mathbb{Q}(f) := c(f) f_* + d(f) z\p_z f_*,
\end{equation}
where $c(f) = \frac{1}{B_*} f_2|_{z=0}$, and $d(f)$ is such that $f - \mathbb{Q}f \in \mathscr{S}$ (which has been introduced in Definition~\ref{def:scripts}).

\begin{definition}\label{def:scripth}
We let
\begin{equation}\label{eq:scriptH}
    (f,g)_{\scrH^k} := (f - \mathbb{Q}(f), g - \mathbb{Q}(g))_{\tilde{\calH}^k} + c(f) c(g) + d(f) d(g).
\end{equation}
We let the associated norm be $|\cdot|_{\mathscr{H}^k}$. 
\end{definition}
\begin{remark}
Importantly, $f-\mathbb{Q}f$ belongs to the space $\tilde{\calH}^k$ although $f_2$ is allowed to be non-vanishing at $z = 0$.
\end{remark}

\subsection{Decomposition of the nonlinear terms without derivative loss}

In this Section, we modify the decomposition of the nonlinear terms in terms of $\alpha$-errors, in order to avoid derivative loss at the highest order. Recall that $\check \Psi_0[g_1] := \frac{1}{\sin(2\theta)} \Psi_0[g_1]$. We let
\begin{align}
    &\tilde{\calN}(f,g)_1 := - 2 \check\Psi_0[g_1]D_\theta f_1 + g_2(\sin(2\theta)D_z (f_2+\alpha f_3) + 2 \cos^2 \theta \p_\theta f_3) , \\
    &\tilde{\calN}(f,g)_2 :=  \check\Psi_0[g_1]\PP_0 (f_2 + \alpha f_3)
    +\alpha \PP_0\Big(-2 \check\Psi_0[g_1] D_\theta f_3  +2 \cos(2 \theta) \check\Psi_0[g_1] D_z (f_2+\alpha f_3) \\
    & \qquad + \frac{f_2+\alpha f_3}{2} \tan \theta  D_z(\Psi[g_1]) + \frac 12 \Big(2 \tan \theta \Psi_1[g] + \p_\theta \Psi_1[g] \Big)(f_2+\alpha f_3)\Big)
    , \label{eq:ntildes}\\
    &\tilde{\calN}(f,g)_3 :=  \PP_0^\perp\Big(-2 \check\Psi_0[g_1] D_\theta f_3  +2 \cos(2 \theta) \check\Psi_0[g_1] D_z (f_2+\alpha f_3)+ \check\Psi_0[g_1] f_3 \\
    & \qquad + \frac{f_2+\alpha f_3}{2} \tan \theta  D_z(\Psi[g_1]) + \frac 12 \Big(2 \tan \theta \Psi_1[g] + \p_\theta \Psi_1[g] \Big)(f_2+\alpha f_3)\Big).
\end{align}

\begin{remark}
Note the presence of $\check \Psi_0[g_1]$ and of the full stream function $\Psi[g_1]$ in the last term (as opposed to $\tilde \Psi_1$).
\end{remark}
As a consequence, we will also consider modified error terms $\tilde E(f,g)$ as follows\footnote{The full expression for these error terms can be found in Appendix~\ref{sec:horrible}.}:
\begin{equation}\label{eq:tilderr}
\begin{aligned}
&\alpha \tilde E(f,g)_1 := \calK(f,g)_1 -\tilde{\calN}(f,g)_1,\\
&\alpha \tilde E(f,g)_2 := \PP_0(\calK(f,g)_2) -\tilde{\calN}(f,g)_2,\\
&\alpha \tilde E(f,g)_3 :=\alpha^{-1} \PP_0^\perp(\calK(f,g)_2) -\tilde{\calN}(f,g)_3.\\
\end{aligned}
\end{equation}
This decomposition, when linearized around the approximate profile $f^*$, induces modified operators $\tscrl$ and $\tscrloc$. We let
\begin{equation}
\begin{aligned}
&\tscrl f := \tilde{\calN}(f, f_*) + \tilde{\calN}(f_*, f),\\
&\tscrloc f := \tscrl f - \scrnonloc f.
\end{aligned}
\end{equation}
Note that the nonlocal terms are lower order, hence we subtract the same nonlocal operator as in Equation~\eqref{eq:defnonloc}\footnote{Note that the potential derivative loss would be in the ``local'' terms, since the ``nonlocal'' terms are finite-rank. Hence, we may as well subtract the same non-local part as before.}. We also define the following decomposition of $\tilde{\calN}$ (and consequently $\tscrloc$):
\begin{equation}
\begin{aligned}
&\tilde{\calN}_\text{der}(f,g)_1 := - 2 \check\Psi_0[g_1]D_\theta f_1,\\
&\tilde{\calN}_\text{der}(f,g)_2:=0,\\
&\tilde{\calN}_\text{der}(f,g)_3:=\PP_0^\perp\Big(- 2 \check\Psi_0[g_1]D_\theta f_3\Big).
\end{aligned}
\end{equation}
We also let:
\begin{equation}
     \tilde{\calN}_\text{rest}(f,g):=\tilde{\calN}(f,g)- \tilde{\calN}_\text{der}(f,g)
\end{equation}
We finally set:
\begin{equation}
\begin{aligned}
    \tscrloc_{\text{der}} f =\tilde{\calN}_\text{der}(f, f_*),  \qquad \qquad \tscrl_{\text{rest}} f = \tilde{\calN}_\text{der}(f_*, f)+\tilde{\calN}_\text{rest}(f, f_*) + \tilde{\calN}_\text{rest}(f_*, f).
\end{aligned}
\end{equation}
\begin{remark}\label{rmk:opdec}
    In what follows, we are going to use the decomposition
    \begin{equation}
    \tscrloc f := \tscrl_{\text{der}} f + \tscrl_{\text{rest}} f - \scrnonloc f
    \end{equation}
\end{remark}

\subsection{Coercivity of \texorpdfstring{$\tscrloc$}{L} at high order}\label{sec:scrloc3}

The final part of this section will be devoted to showing that the operator $f + z\p_z f +\tscrloc f$ is coercive with respect to the inner product $(\cdot, \cdot)_{{\tilde \calH}^k}$. We start with the following lemma on the derivative part of the operator $\tscrloc$, which gives coercivity at the high order.

\begin{lemma}[Coercivity at the high order for $\tscrloc$]\label{lem:derhigh} There exist values $\bar{z}_0 > 0$ and $\bar{\theta}_0 > 0$ such that the following holds true. Given $z_0 > \bar{z}_0$, $\theta_0 \in (\bar{\theta}_0, \pi/2)$ and $M > 0$, there exists an index $k \in \N$, a choice of $\varepsilon_0 > 0$ and a constant $ C > 0$ such that
\begin{align}
    &(f + z\p_z f + \tscrloc_{\text{der}} f, f)_{{\tilde \calH}^k_{\text{high}}} \geq M |f|_{{\tilde \calH}^k_{\text{high}}}^2  - C  |f|_{{\tilde \calH}^k_{\text{low}}}^2 \qquad \text{ if } \text{supp}(f) \subset I_{\text{near}} \times J_{\text{near}},\label{eq:nearI}\\
    &(f + z\p_z f + \tscrloc_{\text{der}} f, f)_{{\tilde \calH}^k_{\text{high}}} \geq  |f|_{{\tilde \calH}^k_{\text{high}}}^2  - C  |f|_{{\tilde \calH}^k_{\text{low}}}^2 \qquad \quad \text{ if } \text{supp}(f) \subset \Big(\tilde I_{\text{near}} \times \tilde J_{\text{near}}\Big)^c.\label{eq:farI}
\end{align}
 Here, we used the intervals $I_{\text{near}}:= [0, 2z_0], J_{\text{near}} := \Big[0, \frac 12 \Big(\theta_0 + \frac \pi 2\Big)\Big]$, $\tilde I_{\text{near}}:= [0, z_0], 
 \tilde J_{\text{near}} := [0, \theta_0]$.
\end{lemma}

\begin{remark}
Note that the commutators introduce weights at $z = \infty$ and at $\theta = \pi/2$, thereby we need to split the contribution into a near region and a far region. In the former (estimate~\eqref{eq:nearI}), the commutators give an arbitrarily large constant, whereas in the latter (estimate~\eqref{eq:farI}) we obtain a uniform constant (in $k$), and we will leverage the decay (in $z$ and $\theta$) of the profile to show Lemma~\ref{lem:calhdissip}.
\end{remark}

\begin{proof}
    The lemma follows from the positivity of the iterated commutators, for $\ell > 0$
    $$
    K_\ell := [\Upsilon^{\ell},I +  z\p_z  + \tscrloc_{\text{der}}]f \qquad \text{and} \qquad J_\ell := [\Lambda^\ell, I +  z\p_z  + \tscrloc_{\text{der}}]f, 
    $$
    and interpolation. Let us analyze the $f_1$ component. We have
    \begin{equation*}
        (\tscrloc_{\text{der}} f)_1 =m(z) D_\theta f_1,
    \end{equation*}
    where $m(z) := -2\check{\Psi}_0 [f_1^*]$ satisfies $\alpha D_z^2 m(z) + D_z m(z) = - \frac{4z}{(1+z)^2}$.
    
    We obtain, letting $(K^z_\ell)_1 := ([\Upsilon^{\ell},I +  z\p_z ]f)_1$, the following recursion relation for $\ell \geq 1$: $(K_{\ell}^z)_1 = \Upsilon (K_{\ell-1}^z)_1 + \p_z \Upsilon^{\ell-1} f$, with $K^z_0 = 0$. This implies
    \begin{equation}\label{eq:comm1}
    (K_\ell)_1 = (K^z_\ell)_1 + (\Upsilon^\ell m(z)) D_\theta f_1.
    \end{equation}
    We deduce that $(K_\ell)_1 = \ell \Upsilon^{\ell -1} f$ up to lower order terms in $f$.
    
    On the other hand, we have, letting $(J_\ell)_1 := ([\Lambda^{\ell},I +  z\p_z +  \tscrloc_{\text{der}}]f)_1$, we have the recursion relation 
    \begin{equation}\label{eq:comm2}
    (J_\ell)_1 = \Lambda (J_{\ell -1})_1 + 2 \cos^2 \theta \Lambda^{j-1} f_1.
    \end{equation}
    We deduce that $(J_\ell)_1 = 2 \ell (\cos \theta)^2 \Lambda^\ell f$ up to lower order terms in $f$. Combining Equations~\eqref{eq:comm1} and~\eqref{eq:comm2}, and using the smallness of $\varepsilon_0$ as a function of $z_0$, we then obtain inequality~\eqref{eq:nearI}, if $\text{supp} f \subset I_{\text{near}} \times J_{\text{near}}$, for sufficiently large $k$. The proof for the other components is analogous, after noting that $(\tscrloc_{\text{der}} f)_3$ is formally the same (up to lower order terms) as  $(\tscrloc_{\text{der}} f)_1$. Finally, the same proof carries over almost identically for inequality~\eqref{eq:farI}.
\end{proof}

We are now ready to state and prove the main Lemma of the present section.

\begin{lemma}[Coercivity under the $(\cdot, \cdot)_{{\tilde \calH}^k}$ inner product]\label{lem:calhdissip}
There exists a constant $c>0$ such that, for any smooth $f \in \mathscr{S}$, we have
\begin{equation}
(f + z\p_z f + \tscrloc f, f)_{{\tilde \calH}^k} \geq c | f|^2_{\calH^k}.
\end{equation}
\end{lemma}

\begin{proof} Recall, from Definition~\ref{def:allinner}, that
\begin{align}
    &(f, g)_{\tilde{\calH}^k} := (f, g)_{{\calH}^k_{,\text{high}}} + c_5 (f, g)_{\tilde{\calH}_{,\text{low}}}.  
\end{align}
In what follows, we are going to determine $c_5$ in order to deduce the Lemma. 

Let $\varepsilon > 0$. From Lemma~\ref{lem:scrlocp}, we have, choosing $\alpha$ sufficiently small\footnote{Recall the equivalence between $\calH^k$ and $\tilde{\calH}^k$.},
\begin{equation}\label{eq:lowproof}
(f + z\p_z f + \tscrloc f, f)_{{\tilde \calH}^k_{\text{low}}} \geq C |f|^2_{\calH^k_{\text{low}}}.
\end{equation}

Let $\chi(x): [0, \infty) \to \R$ be a non-negative smooth cut-off function which is identically 1 on $[0,1]$ and identically $0$ on $[2, \infty)$. We define a cut-off function $\chi_{z_0, \theta_0}$ by letting $
\chi_{z_0, \theta_0}(z,\theta):= \chi(z/z_0)\chi(\tan(\theta)/\tan(\theta_0)).$
This induces a decomposition $f := \chi_{z_0, \theta_0} f + (1-\chi_{z_0, \theta_0}) f =: f_{\text{near}} + f_{\text{far}}$.

Using the decomposition of Remark~\ref{rmk:opdec}, we have
\begin{align}
&(f + z\p_z f + \tscrloc f, f)_{{\tilde \calH}^k_{\text{high}}} = (f + z\p_z f + \tscrloc_{\text{der}} f, f)_{{\tilde \calH}^k_{\text{high}}}+ (\tscrl_{\text{rest}} f - \scrnonloc f , f)_{{\tilde \calH}^k_{\text{high}}}.
\end{align}
For ease of notation, we let $\langle f, g\rangle := (f + z\p_z f + \tscrloc_{\text{der}} f, g)_{{\tilde \calH}^k_{\text{high}}}$.

We then have
\begin{equation}
    \langle f, f \rangle = \langle f_{\text{near}}, f_{\text{near}} \rangle + \langle f_{\text{far}}, f_{\text{far}} \rangle + \langle f_{\text{near}}, f_{\text{far}} \rangle + \langle f_{\text{far}}, f_{\text{near}} \rangle.
\end{equation}
By Lemma~\ref{lem:derhigh}, we then have $\langle f_{\text{near}}, f_{\text{near}} \rangle \geq M |f_{\text{near}}|^2_{\tilde{\calH}^k_{\text{high}}} - C_k |f_{\text{near}}|^2_{\tilde{\calH}^k_{\text{low}}}$ and $\langle f_{\text{far}}, f_{\text{far}} \rangle \geq |f_{\text{far}}|^2_{\tilde{\calH}^k_{\text{high}}} - C_k |f_{\text{far}}|^2_{\tilde{\calH}^k_{\text{low}}}$. We also have, letting $\eta:= \sqrt{\chi_{z_0, \theta_0}(1-\chi_{z_0, \theta_0})}$, by interpolation,
$$
\langle f_{\text{near}}, f_{\text{far}} \rangle +\langle f_{\text{far}}, f_{\text{near}} \rangle \geq 2\langle \eta f, \eta f \rangle - \varepsilon |f|^2_{\tilde{\calH}^k_{\text{high}}} - C_k |f_{\text{near}}|^2_{\tilde{\calH}^k_{\text{low}}}.
$$
By Lemma~\ref{lem:nl} and interpolation, recalling that $\scrnonloc$ has finite rank, we have that, for a positive constant $C_0$ independent of $k$, and for positive constants $C_k$ which depend on $k$,
\begin{align}\label{eq:highproof}
&|(\tscrl_{\text{rest}} f - \scrnonloc f , f)_{{\tilde \calH}^k_{\text{high}}}-(\tscrl_{\text{rest}} f_{\text{far}} - \scrnonloc f_{\text{far}} , f_{\text{far}})_{{\tilde \calH}^k_{\text{high}}}| \leq C_0|f_{\text{near}}|_{{\tilde \calH}^k_{\text{high}}}^2 + C_k |f_{\text{near}}|_{{\tilde \calH}^k_{\text{low}}}^2,\\
&|(\tscrl_{\text{rest}} f_{\text{far}} - \scrnonloc f_{\text{far}} , f_{\text{far}})_{{\tilde \calH}^k_{\text{high}}}| \leq \varepsilon |f_{\text{far}}|_{{\tilde \calH}^k_{\text{high}}}^2 + C_k |f_{\text{far}}|_{{\tilde \calH}^k_{\text{low}}}^2.
\end{align}
The second estimate follows from choosing $z_0$ sufficiently large in Lemma~\ref{lem:derhigh}, and $\theta_0$ sufficiently close to $\pi/2$, and using the decay of the profile. We conclude by combining the above estimates, after choosing $M$ sufficiently large, and $\varepsilon = \frac 14$, as long as $c_5$ is chosen to be sufficiently large.
\end{proof}

\section{\texorpdfstring{On the kernel of the full linearized operator}{On the kernel of f + z dz f + L f}}\label{sec:invert}

In this Section, we analyze the full linearized operator:
\begin{equation}
    \mathfrak{L}f := f + z\p_z f + \scrl f.
\end{equation}
This operator is defined on its natural domain $D(\mathfrak{L}) \subset  \scrH^k$, such that $\mathfrak L: D(\mathfrak{L}) \to \scrH^k$. Note that $\scrloc$ preserves the space $\mathscr{S}$\footnote{This follows from the fact that, at $z=0$, after the transformation $\calB$, the operator reduces to $L_\theta$ which preserves the (purely angular) space $\calS$.}. We notice that the operator $\mathfrak{L}$ is a compact perturbation of an accretive operator on $\scrH^k$. Indeed, adding an appropriate rank $3$ operator to $\mathfrak{L}$, we can transform it into $\mathfrak{L}_{\text{acc}}$, where
$$
\mathfrak{L}_{\text{acc}}f = (f-\mathbb{Q}f) + z\p_z (f-\mathbb{Q}(f)) + \scrloc (f-\mathbb{Q}(f)) + M\mathbb{Q}(f),
$$
where $M>0$ is sufficiently large. We then have
\begin{equation}\label{eq:frakl}
    (f, \mathfrak{L}_{\text{acc}}f)_{\scrH^k} \geq c_0( f,f)_{\scrH^k},
\end{equation}
for a positive constant $c_0$. Our main goal in the present Section is to show the following Lemma concerning the kernel of $\mathfrak{L}$.
\begin{lemma}\label{lem:inv} 
Let $k \geq 4$. Then,
$\mathfrak{L}: D(\mathfrak{L}) \to \scrH^k$ is a compact perturbation of a maximal accretive operator and, in particular, it is a Fredholm operator of index $0$. Moreover, the kernel of $\mathfrak{L}$ is one-dimensional, and the kernel of $\mathfrak{L}$ is not contained in the image of $\mathfrak{L}$.
\end{lemma}

\begin{proof}[Proof of Lemma~\ref{lem:inv}]
The first claim follows by showing that  $\mathfrak{L}_{\text{acc}}-\varepsilon I$ is maximal accretive for a small $\varepsilon > 0$. This in particular implies that $\mathfrak{L}_{\text{acc}}$ is invertible, and the claim follows from the fact that $\mathfrak{L}$ is a compact perturbation of $\mathfrak{L}_\text{acc}$.

To show that $\mathfrak{L}_{\text{acc}}-\varepsilon I$ is maximal accretive, we only need to show that the operator is maximal, as~\eqref{eq:frakl} immediately implies accretivity. We will show that the operator  $\mathfrak{L}_{\text{acc}}$ is surjective, which suffices to show maximal accretivity of $\mathfrak{L}_{\text{acc}}-\varepsilon I$. Due to the estimate~\eqref{eq:frakl} and the fact that $\mathfrak{L}_{\text{acc}}$ is a closed operator, we see that the range of $\mathfrak{L}_{\text{acc}}$ is closed. Owing to the closed range theorem, we then only need to show that the kernel of the adjoint (under the $\mathscr{H}^k$-inner product) $\mathfrak{L}_{\text{acc}}^\dagger$ is trivial. Now, for all $f \in D(\mathfrak{L}_{\text{acc}}^\dagger)$, due to the fact that $\mathfrak{L}_{\text{acc}}$ is a closed operator, 
\begin{equation}
    (\mathfrak{L}_{\text{acc}}^\dagger f, f)_{\mathscr{H}^k} = (\mathfrak{L}_{\text{acc}} f, f)_{\mathscr{H}^k} \geq c_0 |f|^2_{\mathscr{H}^k}
\end{equation}
so that the kernel of the adjoint is trivial, and we conclude. In particular, $\mathfrak L$ is a Fredholm operator of index $0$. 
We now show that $\mathfrak L$ has a one-dimensional kernel. We suppose that a certain $f \in \scrH^k$ belongs to the kernel of $\mathfrak L$: $\mathfrak L f = 0$.

Upon applying the transformation $\calB$, and denoting $(G_1, P_0, G_2) := \calB f$ (and correspondingly $(G_1^*, P_0^*, G_2^*)$ the quantities corresponding to the approximate profile $f_*$), we obtain the following system:
\begin{equation}\label{eq:forker}
\begin{aligned}
&G_1 +z\partial_z G_1 + L(G^*_1)(\frac 12 D_\theta G_1 + G_1) +L(G_1)(\frac 12 D_\theta G^*_1 + G^*_1) = 2 P^*_0 G_2+2 P_0 G^*_2,\\
&G_2 + z\partial_z G_2+ L(G^*_1)(\frac 12 D_\theta G_2 + \frac 7 4 G_2)+ L(G_1)(\frac 12 D_\theta G^*_2 + \frac 7 4 G^*_2)\\
& \qquad \qquad =  \frac 12 \cos^2 \theta P^*_0 \Big(G_1- \dashint G_1\Big)+\frac 12 \cos^2 \theta P_0 \Big(G^*_1- \dashint G^*_1\Big),\\
&P_0 +z \partial_z P_0 = \frac 1 4 L(G^*_1) P_0+\frac 1 4 L(G_1) P^*_0.
\end{aligned}
\end{equation}
We know that $(z\partial_z G^*_1, z\partial_z G^*_2, z\partial_z P^*_0)$ is a solution to the above system, which directly implies that there is a nontrivial kernel element of $\mathfrak{L}$. Let us assume that there exists another nontrivial kernel element, which induces a solution $(G_1, G_2, P_0)$ of~\eqref{eq:forker}. We then have, considering the equation satisfied by $P_0$,
\[P_0+z\partial_z P_0=\frac{1}{4}L(G_1^*)P_0+\frac{1}{4}L(G_1)P_0^*.\]
Evaluating the above at $z=0$ and using that $L(G_{1,*})(0)=4,$ we see that
\[L(G_1)(0)=0.\] Now let us study the equation for $P_0'(0)$. By similar argumentation, we see that 
\[P_0'(0)+P_0(0)=-\frac{B_*}{2}\dashint G_1'(0).\] Now, we may assume that $P_0'(0)+P_0(0)=0$ by subtracting the proper multiple of the known kernel element.  It follows then that 
\[\dashint G_1'(0)=0.\]
Let us focus on\footnote{This is a slight abuse of notation since these two are actually functions of $\theta.$ Let us call them $\Gamma_1$ and $\Gamma_2$ respectively.} $G_1'(0)$ and $G_2'(0)$. Then, we see that 
\begin{align}
&6\Gamma_1+2D_\theta\Gamma_1-\frac{1}{B_*}P_0(0)\Gamma_{2,*}-2B_*\Gamma_2= 0,\label{eq:gam1}\\
&9\Gamma_2+2D_\theta\Gamma_2 - \frac{1}{2}\cos^2(\theta)P_0(0)(\Gamma_{1,*}-\dashint\Gamma_{1,*})-\frac{1}{2}\cos^2(\theta)B_*\Gamma_1=0.\label{eq:gam2}
\end{align}
Recall that, if $f: [0, \infty) \to \R$, $\bar f := f-f(0)$. We apply this operation to both~\eqref{eq:gam1} and~\eqref{eq:gam2}, we divide by $\gamma^2$, integrate between $\gamma = 0$ and $\gamma = \infty$, and consequently sum the 13 times the LHS of~\eqref{eq:gam1} to $2B_*$ times the LHS of~\eqref{eq:gam2}. We obtain:
\begin{equation}
    P_0(0) \Big(\frac{13}{B_*} \int \frac{\bar \Gamma_{2,*}}{\gamma^2} d \gamma + B_* \int\frac{\bar\Gamma_{1,*}+ \gamma^2 \dashint \bar \Gamma_{1,*}}{(1+\gamma^2)\gamma^2}d\gamma\Big)= 0.
\end{equation}
The number multiplying $P_0(0)$ in the above expression is strictly negative, hence we must have that $P_0(0)=0,$ which implies that $P_0'(0)=0$ and also then $\Gamma_1=\Gamma_2=0.$

We assume by induction that $G_1^{(j)}|_{z=0} = G_2^{(j)}|_{z=0} = P_0^{(j)}|_{z=0}=0$ for $j = 0, \ldots, n$. We then take $n$ derivatives of the linearized equations and calculate the result at $z=0$. We obtain (all quantities below are calculated at $z=0$):
\begin{align}
    &n P_0^{(n)} = - \frac{B_*}{2n} \dashint G_1^{(n)},\label{eq:n1}\\
    &(n-1) G_1^{(n)} + 2 D_\theta G_1^{(n)} + 6 G_1^{(n)} - 2 B_* G_2^{(n)} = 0,\label{eq:n2}\\
    &(n-1) G_2^{(n)} + 2 D_\theta G_2^{(n)} + 9 G_2^{(n)} - \frac 12 \cos^2 \theta B_* \Big(G_1^{(n)} - \dashint G_1^{(n)}\Big)=0.\label{eq:n3}
\end{align}
We then apply the operator $\mathcal{A}$ to~\eqref{eq:n2}, and integrate. We have:
\begin{equation}
    (n+9) \int \frac{\bar G^{(n)}_1}{\gamma^2} d\gamma - 2B_*\int \frac{\bar G^{(n)}_2}{\gamma^2} d\gamma = 0.
\end{equation}
Moreover, summing 13 times the above relation to~\eqref{eq:n3} (after having applied $\mathcal{A}$ and having integrated over $\gamma$) we have:
\begin{equation}
    13(n-1) \int \frac{\bar G^{(n)}_1}{\gamma^2} d\gamma + 2B_*\int \frac{\bar G^{(n)}_2}{\gamma^2} d\gamma = 0
\end{equation}
It then follows that $\int \frac{\bar G^{(n)}_1}{\gamma^2} d\gamma = \int \frac{\bar G^{(n)}_2}{\gamma^2} d\gamma= 0$. The energy argument then implies that necessarily $G^{(n)}_1 =G^{(n)}_2=0 $ everywhere. We make use of the good damping given by $n \geq 2$. Finally, $P_0^{(n)} = 0$.

Consider the $z$-interval $I = [0,M]$, with $M>0$ large. We show that $G_1(z)$ and $G_2(z)$ are identically $0$ on $I$. We divide the equations~\eqref{eq:forker} by $z^m$ ($m$ is chosen large, depending on the size of $I$), and perform weighted $L^2$ estimates between $0$ and $M$ with the weight $z^{-m}$. The damping term given by the positive commutator of the damping term in $z$ with the operator $z\p_z$ then gives that the $(G_1, P_0, G_2) = 0$ on $I$. Finally, an unweighted $L^2$ estimate will suffice to show that $(G_1, P_0, G_2)$ are identically zero for $z \geq M$. The claim then follows from inverting the transformation $\mathcal{B}$, by the same reasoning as in the proof of Lemma~\ref{lem:profileomega}. This shows that the kernel of $\mathfrak{L}$ is one-dimensional.

Now, we turn to showing that $\mathfrak{L}$ admits no nontrivial generalized kernel element. We suppose that $f^\# \in \scrH^k$ is such that $\mathfrak{L} f^\# =-B_*^{-1} z\p_z f_*$. This leads us to equation~\eqref{eq:forker}, where now $-B_*^{-1}z\p_z f_*$ is on the RHS. We consider the quantities\footnote{Note that these quantities have the property that they vanish when calculated on the profile. This implies that the RHS of equations~\eqref{eq:q1} and~\eqref{eq:q2} is trivial.} $Q_1 := G^\#_1 + \frac 1 {B_*} (z\p_z P^\#_0) \Gamma_1^*, Q_2:= G^\#_2 - \frac 1 {2B_*^2} z\p_z P^\#_0 \Gamma_2^*$, and $P_0 = P^\#_0$. Note that $(Q_1, P_0, Q_2) \in \scrH^k$. A calculation shows that these quantities satisfy the system (suppressing the symbol $\#$ in the notation)
\begin{align}
    &Q_1 + z\p_z Q_1 + \frac 1{z+1} (2 D_\theta Q_1 + 4 Q_1) + \frac{\Gamma_1^*}{2(1+z)} \dashint Q_1 \\
    & \qquad + \frac{z}{(1+z)^2}L(Q_1) \Big(\frac 12 D_\theta \Gamma_1^* + \frac 5 4 \Gamma_1^*\Big) + \frac{2B_*Q_2}{1+z}= 0, \label{eq:q1}\\
    &Q_2 + z\p_z Q_2 + \frac{1}{1+z}(2 D_\theta Q_2 + 7 Q_2) + \frac {B_*}{2(1+z)} \cos^2 \theta (Q_1 - \dashint Q_1) - \frac 1 {4 B_* (1+z)}\Gamma_2^* \dashint Q_1 \\
    & \qquad  - \frac z {2B_*(1+z)^2}L(Q_1) \big(\frac 12 D_\theta \Gamma_2^* + 2 \Gamma_2^* \big) = 0. \label{eq:q2}\\
    &P_0 + z\p_z P_0 = \frac 2{z+1} P_0 + \frac 1 4 \frac{B_*}{1+z} L(Q_1) +\frac z {(1+z)^2 }.\label{eq:q3}
\end{align}
We focus our attention on the equation satisfied by $Q_1$ and $Q_2$, which we rewrite as follows, letting $\tilde Q_1 :=  (1+z)^2 Q_1/z, \tilde Q_2 := - 2 B_* (1+z)^2 Q_2/z$:
\begin{align}
    &z\p_z \tilde Q_1 + \frac 1{z+1}  (\tilde L_\theta \tilde Q)_1 + \tilde L (\tilde Q_1) F_1^*(\theta) =0 , \label{eq:qt1}\\
    &z\p_z \tilde Q_2 + \frac 1{z+1}  (\tilde L_\theta \tilde Q)_2 + \tilde L (\tilde Q_1) F_2^*(\theta) =0. \label{eq:qt2}
\end{align}
Here, $\tilde L_\theta Q := \Big(2 D_\theta Q_1 + 6 Q_1 - Q_2 + \frac 12 \Gamma_1^* \dashint Q_1, 2 D_\theta Q_2 + 9 Q_2 - B_*^2 \Big(Q_1 - \dashint Q_1 \Big) +\frac 12 \Gamma_2^* \dashint Q_1\Big)$, and $\tilde L(Q_1) := \int_z^\infty \dashint Q_1(y) \frac{dy}{(1+y)^2}$, $F_1^* = \frac 12 D_\theta \Gamma_1^* + \frac 54 \Gamma_1^*,$ and finally $F_2^* = \frac 12 D_\theta \Gamma_2^* + 2 \Gamma_2^*$.

We let $\calH^k_{\tilde Q}$ the natural space for $\tilde Q$ induced by $\scrH^k$:
\begin{equation}
    (\tilde Q_1, \tilde Q_2) \in \calH^k_{\tilde Q} \iff \frac{z}{(1+z)^2} \tilde Q_1  \in \calH^k_{,1}, \ \ \text{and} \ \   \frac{z}{(1+z)^2}\tilde Q_2 \in \calH^k_{,2}. 
\end{equation}
Our goal is to show that, if $\tilde Q_1, \tilde Q_1 \in \calH^k_{\tilde Q}$, then $\tilde Q_1 = \tilde Q_2 = 0$. To simplify notation, we rewrite equations~\eqref{eq:qt1}--\eqref{eq:qt2}, as $\scrL_0 \tilde Q = 0$, where we view the operator $\mathscr{L}_0 : D(\mathscr{L}_0) \to \calH^k_{\tilde Q}$. This is carried out in the companion paper~\cite{symbolic}, which concludes the argument.
\end{proof}

\section{Estimates on the nonlinear terms}\label{sec:nonlinear}

In this section, we show the following lemma on trilinear expressions. This is the core ingredient in showing that the Boussinesq system is locally well-posed in the space $\scrH^k$.

\begin{lemma}[Nonlinear estimates]\label{lem:nl}

Let $f \in \calH^k$, $g, h \in \scrH^k$.  The following inequalities hold true with $C_k >0$ depending on $k$:
\begin{align}
    &|(f, \mathbb{Q} \tilde{\calN}_{\text{der}}(f,g) )_{\tilde{\calH}^k}|\leq C_k|f|^2_{\calH^k}|g|_{\scrH^{k}},\label{eq:nlsimple}\\
    &|(f, \mathbb{Q} \tilde{\calN}_{\text{rest}}(f,g) )_{\tilde{\calH}^k}|\leq C_k|f|^2_{\calH^k}|g|_{\scrH^{k}},
    \label{eq:nlsimplerest}\\
    &|(f, \alpha \mathbb{Q} \tilde{E}(f,g))_{{\tilde \calH}^k}|\leq \alpha C_k|f|^2_{\calH^k}|g|_{\scrH^{k}}.\label{eq:nlerrbetter}
\end{align}
In addition, we have the inequalities with a loss:
\begin{align}
        &|(f, \mathbb{Q} \tilde{\calN}(g,h))_{{\tilde \calH}^k}|\leq C_k|f|_{\calH^k}|g|_{\scrH^{k+1}}|h|_{\scrH^{k}},\label{eq:nltwo}\\
        &|(f, \alpha\mathbb{Q} \tilde{E}(g,h))_{{\tilde \calH}^k}|\leq \alpha C_k|f|_{\calH^k}|g|_{\scrH^{k+1}}|h|_{\scrH^{k}},\label{eq:nlerr}.
\end{align}
Here, we used the definition of $\mathbb{Q}$ from equation~\eqref{eq:qprojection}.
\end{lemma}

\begin{proof}[Proof of Lemma~\ref{lem:nl}]
Recall that $(f,g)_{\tilde \calH^k} = (f_2+\alpha f_3, g_2+\alpha g_3)_{\calH^k_{,0}} + \sum_{i = 1}^3 (f,g)_{\calH^k_{,i}}$. We begin by proving~\eqref{eq:nlsimple}. We have, since $f_2$ is a radial function,
 \begin{align*}
     &|(f,\mathbb{Q}\tilde \calN_{\text{der}}(f,g))_{\tilde \calH^k}| \leq |(f,\tilde \calN_{\text{der}}(f,g))_{\tilde \calH^k}|+ C_k |f|^2_{\calH^k}|g|_{\scrH^{k}}\\
     & \leq | (f_2+\alpha f_3,  - 2 \check\Psi_0[g_1]D_\theta (f_2 + \alpha f_3) )_{ \calH^k_{,0}}|\\
     &\qquad + |(f_1,  - 2 \check\Psi_0[g_1]D_\theta f_1)_{ \calH^k_{,1}}| +  |(f_3,   - 2 \check\Psi_0[g_1]D_\theta  f_3)_{ \calH^k_{,3}}|+ C_k |f|^2_{\calH^k}|g|_{\scrH^{k}} \leq  C_k |f|^2_{\calH^k}|g|_{\scrH^{k}}.
 \end{align*}
Here, we used we used Lemma~\ref{lem:alg0}, and Theorem~\ref{thm:elliptic} in conjunction with the transport estimates of Lemma~\ref{AngularTransportEstimate} and Lemma~\ref{RadialTransportEstimate}.

In what follows, we focus on the estimates of the top-order homogeneous norms. We turn to the proof of estimate~\eqref{eq:nlsimplerest}. We have
\begin{align*}
    &\Big|(f_2 + \alpha f_3, \tilde \calN(f,g)_2 + \alpha \tilde \calN (f,g)_3)_{\dot \calH^k_{,0}}\Big| \leq \Big|\Big(f_2 + \alpha f_3, \check\Psi_0[g_1](f_2 + \alpha f_3) +   2\alpha \cos(2 \theta) \check\Psi_0[g_1] D_z (f_2+\alpha f_3) \\
    &+\alpha \frac{f_2+\alpha f_3}{2} \tan \theta  D_z(\Psi_0[g_1]+\alpha \Psi_1[g_1]) + \frac 12 \alpha \big(2 \tan \theta \Psi_1[g] + \p_\theta \Psi_1[g] \big)(f_2+\alpha f_3)
    \Big)_{\calH^k_{,0}}\Big|+ C_k  |f|^2_{\calH^k}|g|_{\scrH^{k}}\\
    & \qquad \leq  C_k  |f|^2_{\calH^k}|g|_{\scrH^{k}}.
\end{align*}
 In the last line, we used the algebra property of Lemma~\ref{lem:alg0}, as well as Theorem~\ref{thm:elliptic} to estimate the terms in $\check\Psi_0$ and $\Psi_1$, and Lemma~\ref{RadialTransportEstimate} to estimate the transport term $D_z(f_2+\alpha_3)$\footnote{The worst term occurs when top-order derivatives fall on $D_z(\Psi_0[g_1] + \alpha \Psi_1[g_1])$). Here, we use crucially the elliptic estimates of Theorem~\ref{thm:elliptic}, noting that this term is multiplied by a factor of $\alpha$. Note that all other terms display better scaling.}.
 
 We also have
 \begin{align*}
    &\Big|(f_1, \tilde \calN(f,g)_1)_{\dot \calH^k_{,1}}\Big| \leq \Big|\Big(f_1 , g_2(\sin(2\theta)D_z (f_2+\alpha f_3) + 2 \cos^2 \theta \p_\theta f_3)\Big)_{ \calH^k_{,1}}\Big|+ C_k  |f|^2_{\calH^k}|g|_{\scrH^{k}}\\
    &\qquad \leq C_k  |f|^2_{\calH^k}|g|_{\scrH^{k}}.
\end{align*}

Here, we used the algebra property of Lemma~\ref{lem:alg0}, in conjunction with the fact that the $\calH^k_{,0}$ norm controls one additional $D_z$-derivative at the top order, and the fact that the $\calH^k_{,3}$ norm controls one additional $\p_\theta$ derivative at the top order.

We estimate
 \begin{align*}
    &\Big|(f_2, \tilde \calN(f,g)_2)_{\dot \calH^k_{,2}}\Big| \leq \Big|\Big(f_2, \check\Psi_0[g_1]\PP_0 (f_2 + \alpha f_3))\Big)_{ \calH^k_{,2}}\Big|+C_k  |f|^2_{\calH^k}|g|_{\scrH^{k}}\leq C_k  |f|^2_{\calH^k}|g|_{\scrH^{k}}.
\end{align*}
Here, we used again the algebra property of Lemma~\ref{lem:alg0}.

We also estimate
 \begin{align*}
    &\Big|(f_3, \tilde \calN(f,g)_3)_{\dot \calH^k_{,3}}\Big| \leq \Big|\Big(f_3, \PP_0^\perp\Big(-2 \check\Psi_0[g_1] D_\theta f_3  +2 \cos(2 \theta) \check\Psi_0[g_1] D_z (f_2+\alpha f_3)+ \check\Psi_0[g_1] f_3 \\
    & \qquad + \frac{f_2+\alpha f_3}{2} \tan \theta  D_z(\Psi[g_1]) + \frac 12 \Big(2 \tan \theta \Psi_1[g] + \p_\theta \Psi_1[g] \Big)(f_2+\alpha f_3)\Big)\Big)_{ \calH^k_{,3}}\Big|+C_k  |f|^2_{\calH^k}|g|_{\scrH^{k}}\\
    &\qquad \leq C_k  |f|^2_{\calH^k}|g|_{\scrH^{k}}.
\end{align*}
We argue similarly as the previous terms. The only terms which require some care are the first two advection terms. The $D_\theta f_3$ term is readily estimated using the angular transport bounds of Lemma~\ref{AngularTransportEstimate}. The term containing $D_z(f_2 + \alpha f_3)$, on the other hand, when differentiated by top order $\p_z$ derivatives, is estimated using the additional control warranted by the $\calH^k_{,0}$ (which controls $k+1$ radial derivatives). On the other hand, since $f_2$ is radial, when this term is differentiated at least once by $\Lambda$ it only contains $f_3$. Therefore, in this case, an application of the radial transport estimates of Lemma~\ref{RadialTransportEstimate} will suffice.

Finally, the proofs of~\eqref{eq:nlerrbetter},~\eqref{eq:nltwo} and~\eqref{eq:nlerr} are analogous and we omit them.
\end{proof}

\section{Construction of the profile and stability: a general framework} \label{sec:general}

This Section is devoted to building a general framework to construct a self-similar profile out of an approximate profile, with fairly general assumptions. Section~\ref{sec:construct}, we formulate a general framework for profile construction. Finally, in Section~\ref{sec:stability}, we show a suitable finite-codimension stability statement for the exact self-similar profile constructed, which will enable us to construct solutions with finite energy.

\subsection{A general framework for constructing self-similar profiles} \label{sec:construct}
Consider a collection of inner products indexed by integers $k$
$$
(a , b)_{\scrH^k}, \qquad (a , b)_{\calH^k}, \qquad (a , b)_{\tilde \calH^k}
$$
which resp. induce the norms $|f|_{\scrH^k}$, $|f|_{\calH^k}$, $|f|_{\tilde{\calH}^k}$. We assume that the norms induced by $\calH^k$ and $\tilde \calH^k$ are equivalent, whereas the norm induced by $\scrH^k$ is weaker than $|f|_{\calH^k}$ (and thus also weaker than $|f|_{\tilde{\calH}^k}$). We consider $\mathscr{S}$, a hyperplane contained in $\calH^k$. Let $\alpha >0$ be a parameter.

We consider linear operators $\mathcal{L}_1$, $\mathcal{L}_2$, $\calL_3$, $\calL_E$ such that the following conditions hold true:
\begin{enumerate}
    \item[$\ell1$.] $\mathcal{L}_1$ is a bounded operator from $\mathcal{H}^k$ to $\mathcal{H}^{k-k_1}$, for some $k_1 \in \N$. Moreover, $( f, \mathcal{L}_1 f)_{\tilde{\calH}^k} \geq c |f|^2_{\calH^k}$, for some constant $c > 0$ and for all $f \in  \mathscr{S}$.
    \item[$\ell2$.] $\mathcal{L}_2$ is a Fredholm operator of index 0 from $\scrH^k$ to $\text{range}(\calL_2) \subset \scrH^{k-k_1}$, for some $k_1 \in \N$. The inverse $\mathcal{L}_2^{-1}: \text{range}(\calL_2) \to \scrH^k$ is a bounded operator.
    \item[$\ell3$.] For all $k_2 >0$, $k_2 \in \N$, $\calL_3$ is a bounded operator mapping $\scrH^k$ to $\scrH^{k+k_2} \cap \text{range}(\calL_2)$.
    
    \item[$\ell4$.] $\mathcal{L}_E$ is a bounded linear map from $\scrH^k\times \scrH^k$ to $\mathcal{H}^{k-k_1}$ such that $(\mathcal{L}_E (f,g), f)_{\tilde{\calH}^k} \leq c (|f|^2_{\scrH^k}+|g|^2_{\scrH^k}) $ for all $f \in \calH^k$, $g \in \scrH^k$.
\end{enumerate}
We also consider a nonlinear term $\scrn(f,g)$ with the following properties:
\begin{enumerate}
    \item[$n1$.] $\scrn: \scrH^k \times \scrH^k \to \calH^{k - k_1}$ is bilinear and strongly continuous, for some $k_1 > 0$,
    \item[$n2$.] There exists $k_3>0$, $k_3 \in \N$, and there exists $k_4 > 0$, $k_4 \in \N$ such that for all $k \geq k_3$, we have
    \begin{equation}
    \begin{aligned}\label{eq:nonl}
    &|(\scrn(g,f), f)_{\tilde{\calH}^k}|  \leq c|f|^2_{\calH^k}|g|_{\scrH^k},\\
    & |\scrn(g,f)|_{\tilde{\calH}^k}\leq c | g|_{\scrH^k}|f|_{\calH^{k+k_3}}.
    \end{aligned}
    \end{equation}
    for all $f \in \calH^{k+k_4}$ and all $g \in \scrH^k$, and $c>0$ is a constant.
    \item[$n3$.] $\scrn$ satisfies the Leibnitz rule with respect to $\p_t$: $
        \p_t\scrn(f,g) = \scrn(\p_t f, g) + \scrn(f, \p_t g).$
\end{enumerate}
We also consider a strongly continuous map $\mathscr{M}(f,g): \calH^k \times \scrH^k \to \calH^k$ with the following property:
\begin{enumerate}
\item[$m1$.] For all $f\in \calH^k$, $g \in \scrH^k$ sufficiently small, we have
\begin{align}
&|(\mathscr{M}(f,g),f)_{\calH^k}| \leq C(\alpha + |f|_{\calH^k})^2(\alpha + |g|_{\scrH^{k+1}}) + C \alpha,\label{eq:finalmod1}\\
&(\p_t \mathscr{M}(f,g), \p_t f)_{\calH^{k-1}} \leq C (\alpha + |f|_{\calH^k}+|g|_{\scrH^{k+1}})(|\p_t f|_{\calH^{k-1}}+|\p_t g|_{\scrH^{k+1}})|\p_t f|_{\calH^{k-1}}.
\end{align}
\end{enumerate}

Under these conditions, we have the following Theorem.

\begin{theorem}\label{thm:profiles} Consider the following problem, posed on the space $f\in \calH^k$, $g \in \scrH^k$:
\begin{align}\label{eq:f}
\mathcal{L}_1 f &=  \scrn(f+g, f+g) + \alpha \calL_E(f,g) +\alpha E + \mathscr{M}(f,g),\\
\calL_2 g &= \calL_3 f.\label{eq:g}
\end{align}
Let us assume that $E$ (a time-independent forcing term) belongs to the space $\calH^{k}$ for all $k \in \N$. Let us also suppose that the RHS of~\eqref{eq:f} lies in $\mathscr{S}$ for all $f \in \calH^k$, $g \in \scrH^k$. Under these conditions, there exists $\alpha_0 > 0$ and $K \in \N, K > 0$ such that, for all $\alpha \in (0, \alpha_0)$ there exists a solution $(f,g) \in \calH^{K} \times \scrH^{K}$ of system~\eqref{eq:f}--\eqref{eq:g}.
\end{theorem}

\begin{remark}
A consequence of this theorem is that, upon choosing $\alpha$ small enough, we can constrain the resulting values of the modulation parameters $\lambda$ and $\mu$ in equation~\eqref{eq:profileshort} to be both as close to $1$ as we please.
\end{remark}

\begin{remark}\label{rmk:hypo1}
    Note that the assumptions are satisfied in our setup. Let the spaces $\calH^k$, $\tilde \calH^k$ and $\scrH^k$ be understood resp. as in Definition~\ref{def:allinner}, and as in equation~\eqref{eq:scriptH}. We are interested in solving (redefining $\mu$ and $\lambda$ appropriately)
    $$
    (1+\mu) q +(1+\lambda)z\p_zq + \tilde{\calN}(q,q) + \alpha \tilde{E}(q,q) = 0.
    $$
    Upon linearization around $f_*$, we obtain (with a slight abuse of notation we denote again by $q$ the perturbation)
    \begin{equation}\label{eq:lin1}
    q + z\p_z q  + \tilde{\mathscr{L}} q + \tilde{\calN}(q,q)+\alpha  \tilde E(q+f_*, q+f_*) + \lambda z\p_z q + \lambda z\p_z f_* + \mu q + \mu f_* =0.
    \end{equation}
    We then rewrite
    \begin{equation}\label{eq:forprofile}
    q + z\p_z q  + \tilde{\mathscr{L}} q + \alpha(\tilde E(q,f_*) + \tilde E(f_*, q)) + \alpha \tilde E(q,q) + \tilde{\calN}(q,q) + \lambda z\p_z q + \lambda z\p_z f_*  + \mu q + \mu f_*  + \alpha \tilde E(f_*, f_*)  =0.
    \end{equation}
    We let 
   \begin{align}
   &\calL_1 q := q + z\p_z q + \tilde{\scrloc} q,\\
   &\calL_2 q := q + z\p_z q + \mathscr{L} q,\\
   &\calL_3 q := \scrnonloc q + \lambda_2(q) z\p_z f_*,\\
   &\calL_E (q_1, q_2) :=\mathbb{Q}\big( -(\tilde E(q_1+q_2,f_*) + \tilde E(f_*, q_1+q_2)) - \alpha^{-1}(\tscrl q_2 - \scrl q_2)\big),\\
   &\mathscr{N}(q,q) := \mathbb{Q} \big(-\tilde \calN(q,q) - \alpha \tilde E(q,q)\big).
   \end{align}
   Here, we used the definition of the projection $\mathbb{Q}$ as in~\eqref{eq:qprojection}.
   Moreover, $\lambda_2(q)$ is determined such that $ \calL_3 q \in \text{range} (\calL_2)$.
   Note that this is a linear condition, which makes $\calL_3$ a linear operator on $q$. Writing now  $q = f+g$, we obtain that~\eqref{eq:lin1} is implied by the system of equations 
\begin{align}
    &\calL_1 f = \mathscr{N}(f+g, f+g)+ \alpha \calL_E(f,g) - (\lambda - \lambda_2(f)) z\p_z f_* -\mu f_* - \alpha \tilde E(f_*, f_*) \\
    &\qquad - \lambda z\p_z (f+g) - \mu (f+g),
    \label{eq:modulated1}\\
    &\calL_2 g =- \calL_3 f.\label{eq:modulated2}
\end{align}
We notice that there are two conditions that need to be satisfied for a vector function $f \in \calH^k$ to lie in $\mathscr{S}$. We need that the second component $f_2$ vanishes at $z = 0$. Thus, the choice of $\mu$ is dictated by imposing the condition
\begin{align}
    &\mu(f,g) (f_*)_2|_{z=0} = (\mathscr{N}(f+g, f+g)_2)|_{z=0}+ \alpha (\calL_E(f,g)_2)|_{z=0} - \alpha (\tilde E(f_*, f_*)_2)|_{z=0}.
\end{align}
In addition, to ensure that $f \in \mathscr{S}$, we need to impose the vanishing of a continuous linear operator 
$L_{\mathscr{S}}: \calH^{k-1} \to \R$ when applied to $f$\footnote{In particular, $L_{\mathscr{S}}f = \PP_0\Big(\frac{\calA(13(\calB f)_1 + (\calB f)_3)}{z(\cos(\theta))^2}\Big)\Big|_{z=0}$.}. In other words, $q \in \mathscr{S}$ if and only if $q \in \calH^k$ and
$L_{\mathscr{S}} q = 0.$
Therefore, applying $L_{\mathscr{S}}$, we define $\lambda_1(f,g)$ and consequently $\lambda = \lambda_1(f,g) + \lambda_2(g)$ such that
\begin{align}
    &\lambda_1(f,g) (L_{\mathscr{S}}(z\p_z f_*) +L_{\mathscr{S}}(z\p_z (f+g)) + \lambda_2(f)  L_{\mathscr{S}}(z\p_z (f+g))\\
    &= L_{\mathscr{S}}(\mathscr{N}(f+g, f+g)+ \alpha \calL_E(f,g) - \alpha \tilde E(f_*, f_*)) - \mu(f,g) L_{\mathscr{S}}(f_*+f+ g).
\end{align}
We then let $\mathscr{M}(f,g) := - \lambda_1(f,g) z\p_z f_* - (\lambda_1(f,g) + \lambda_2(f)) z\p_z(f+g)- \mu(f,g) (f_* + f+ g)$. Under this choice of $\mathscr{M}$ the projection $\mathbb{Q}$ simplifies, and the system~\eqref{eq:modulated1}--\eqref{eq:modulated2} is in the framework of Theorem~\ref{thm:profiles}.
\end{remark}

Next, we prove Theorem~\ref{thm:profiles}.

\begin{proof}[Proof of Theorem~\ref{thm:profiles}]
We argue by a long-time limit procedure. To this end, we define an auxiliary system
\begin{align}\label{eq:fe}
\p_t f +  \mathcal{L}_1 f &=  \scrn(f+g, f+g) + \alpha \calL_E(f+g) +\alpha E + \mathscr{M}(f,g),\\
\calL_2 g &= \calL_3 f.\label{eq:ge}
\end{align}
We let the initial data $f|_{t=0} := f_0 = 0$. We note that the operator $\calL_4 := \calL_2^{-1} \calL_3$ is bounded from $\scrH^k$ to $\scrH^{k+k_2}$, for all $k_2 > 0$. We can therefore rewrite the above system in the form (upon renormalizing by a factor of $\alpha$, $f \mapsto \alpha f$):
\begin{equation}\label{eq:master}
    \p_t f + \mathcal{L}_1 f = \alpha \scrn(f+\calL_4 f, f+\calL_4 f) + \alpha \calL_E(f+\calL_4 f) + E +\alpha^{-1} \mathscr{M}(\alpha f,\alpha \calL_4 f).
\end{equation}
We set the bootstrap assumption $|f|_{\calH^k} \leq 2A$
for some $A>0$. We take the $\tilde{\calH}^k$ inner product of equation~\eqref{eq:master} with $f$, and obtain the following:
\begin{align*}
    &\frac 12 \p_t |f|^2_{\tilde{\calH}^k} + c_1 |f|^2_{\tilde{\calH}^k} \leq \alpha c_2 |(\scrn(f+ \calL_4 f, f+ \calL_4 f), f)_{\tilde\calH^k}|\\
    & \qquad \qquad + |f|_{\calH^k}(|E|_{\calH^k}+\alpha^{-1}|\mathscr{M}(\alpha f,\alpha  \calL_4 f)|_{\calH^k}) + \alpha (\calL_E(f+ \calL_4 f), f)_{\tilde{\calH}^k}.
\end{align*}
We rewrite, using bilinearity,
\begin{equation}
\scrn(f + L_4 f, f+L_4 f) = \scrn(f+L_4 f, L_4f) + \scrn(f+L_4 f, f).
\end{equation}
We have, using the inequalities in display~\eqref{eq:nonl} on $\scrn$,
\begin{align*}
    &|(\scrn(f+L_4 f, f), f)_{\tilde \calH^k}| \leq C |f+L_4 f|_{\scrH^k} |f|_{\calH^k}^2 \leq C |f|_{\calH^k}^3,\\
    &|(\scrn(f+L_4 f, L_4f),f)_{\tilde \calH^k}|  \leq C |f+L_4 f|_{\scrH^k} |f|_{\calH^k} |L_4 f|_{\scrH^{k+k_4}} \leq C|f|_{\calH^k}^3,
\end{align*}
by choosing $k_2 \geq k_4$. Using assumptions $(\ell4)$, $(m2)$ we then obtain the inequality:
\begin{equation}
    \frac 12 \p_t |f|_{\tilde{\calH}^k} + c_1 |f|_{\tilde{\calH}^k} \leq \alpha c_3 |f|^2_{\calH^k} +\alpha c_4 |f|_{\calH^k} + |E|_{\calH^k} + c_5.
\end{equation}
for some other constants $c_3, c_4, c_5 > 0$. We use the bootstrap assumption to obtain
\begin{equation}\label{eq:togron}
    \frac 12 \p_t |f|_{\tilde{\calH}^k} + (c_1 - 2 c_3 \alpha A)|f|_{\tilde{\calH}^k} \leq |E|_{\calH^k}+c_6,
\end{equation}
for another constant $c_6 > 0$ independent of $\alpha$. Using  Gr\"onwall's inequality, the equivalence of $\tilde{\calH}^k$ and $\calH^k$ norms, and choosing $A$ to be large enough and $\alpha$ small enough depending on $c_1, c_2, \ldots, c_6$, we show the improved bound $|f(t)|_{\calH^k}\leq A$, which closes the bootstrap. This in particular shows that the system~\eqref{eq:fe}--\eqref{eq:ge} has a global-in-time solution for all $\alpha \in (0, \alpha_0)$, where $\alpha_0 > 0$ has been determined above.

Lastly, we take a time derivative of the equation~\eqref{eq:master}, and we obtain, since $E$ is autonomous,
\begin{equation}\label{eq:mastert}
    \p^2_t f + \mathcal{L}_1 \p_t f = \alpha \p_t(\scrn(f+\calL_4 f, f+\calL_4 f)) + \alpha \calL_E(\p_t(f+\calL_4 f)) + \alpha^{-1}\p_t(\mathscr{M}(\alpha f, \alpha \calL_4 f)) .
\end{equation}
We then test equation~\eqref{eq:mastert} against $\p_t f$ (note that $\p_t f \in \mathscr{S}$) with respect to the $\tilde \calH^{k-1}$ inner product. This yields:
\begin{align}
    &\frac 1 2 \p_t |\p_t f(t)|_{\calH^{k-1}}^2 + c_1 |\p_t f|^2_{\calH^{k-1}} \\
    &\leq \alpha |(\p_t (\scrn(f + \calL_4 f, f + \calL_4 f)), \p_t f)_{\tilde \calH^{k-1}}| + \alpha |\calL_E( \p_t(f + \calL_4 f), \p_t f)_{\tilde \calH^{k-1}}|\\
    &\quad  + \alpha^{-1} |(\mathscr{M}(\alpha f, \alpha \calL_4 f), \p_t f)_{\tilde \calH^{k-1}}|.
\end{align}
We use the Leibnitz rule (assumption $(n3)$), bilinearity of $\scrn$, the inequalities in display~\eqref{eq:nonl}, together with assumptions $(\ell2)$ and $(\ell3)$ to obtain
\begin{align*}
     |(\p_t (\scrn(f + \calL_4 f, f + \calL_4 f)), \p_t f)_{\tilde \calH^{k-1}}| \leq C |\p_t f|^2_{\calH^{k-1}} |f|_{\calH^k}.
\end{align*}
We estimate the term containing $\calL_E$ using assumption $(\ell4)$, and we use assumption $(m1)$ and the uniform bound on $|f|_{\calH^k}$ to estimate the term containing $\mathscr{M}$:
$$
\alpha^{-2} |(\mathscr{M}(\alpha f, \alpha \calL_4 f), \alpha \p_t  f)_{\tilde{\calH}^{k-1}}| \leq \alpha C |\p_t f|^2_{\calH^{k-1}}.
$$
By choosing $\alpha$ small, we then obtain the inequality
\begin{equation}
    \frac 1 2 \p_t |\p_t f(t)|_{\tilde{\calH}^{k-1}}^2 + \frac{c_1}{2} |\p_t f|^2_{\tilde{\calH}^{k-1}} \leq 0.
\end{equation}
This shows that $|\p_t f(t)|_{\calH^{k-1}} \to 0$ exponentially as $t \to \infty$. Then, $f_\infty:= \int_0^\infty \p_t f(t) dt$ is well defined, $f_\infty \in \calH^{k-1}$, and the convergence is strong in $\calH^{k-1}$. Upon taking the limit $t \to \infty$ in display~\eqref{eq:master}, using the assumptions $(\ell1)-(\ell4)$, $(n1)-(n3)$, $(m1)$, we deduce the existence of $k_5 > 0$ such that $f_\infty$ satisfies the steady equation~\eqref{eq:master} strongly in $ \calH^{k-k_5}$. The claim then follows by choosing $k$ large enough and by setting $g_\infty = \calL_4 f_\infty$.

\end{proof}

\subsection{Finite-codimension stability}
\label{sec:stability}
In addition to the assumptions in Section~\ref{sec:construct}, we are going to need the following assumptions about $\mathcal{L}_2$ and $\calL_3$. Note that we will linearize about the solution in Section~\ref{sec:construct}.

\begin{itemize}
    \item[$s1.$] $\mathscr{H}^k = \mathscr{U} \oplus \mathscr{S}_0$, where $\mathscr{U}$ and $\mathscr{S}_0$, are invariant subspaces under $\calL_2$ such that  $\mathscr{U}$ is a finite direct sum of eigenspaces of $\calL_2$ with eigenvalues of non-positive real part, and for $f_0 \in \mathscr{S}_0$ the following semigroup estimate holds:
    \begin{equation}
    | \exp(-t \calL_2 ) f_0|_{\mathscr{H}^k} \leq C \exp(- \eta t).
    \end{equation}
    Here, $\eta > 0$. We denote by $\Pi$ the projection onto $\mathscr{S}_0$ given by the above direct sum decomposition.
    
    \item[$s2.$] For any function $f \in \mathscr{H}^k$, we have
    \begin{equation}
        \calL_3 f = F_* L_1(f),
    \end{equation}
    where $F_*  \in \calH^\infty$, and $L_1$ is a bounded linear functional from $\mathscr{H}^k$ to $\R$.
    
\end{itemize}

We will also need to replace the modulation condition $(m1)$ with the following condition, due to the fact that we are now linearizing around an exact profile
\begin{itemize}
    \item[$m1'.$] We assume that the term $\mathscr{M}$ satisfies the following inequalities, for all $f, \bar f \in \calH^k$ and all $g, \bar g \in \mathscr{H}^{k+1}$ both sufficiently small, and for $\bar k \leq k -4$:
\begin{align}\label{eq:impmod1}
&|(\mathscr{M}(f,g), f)_{\calH^k}| \leq C |f|^2_{\calH^k} |g|_{\mathscr{H}^{k+1}},\\
&|(\mathscr{M}(f,g)- \mathscr{M}(\bar f, \bar g), f - \bar f)_{\calH^{\bar k}}|\\
& \qquad \leq C( |f- \bar f|^2_{\calH^{\bar k}} + |g- \bar g|^2_{\mathscr{H}^{\bar k+1}}) (|f|_{\calH^k}+ |\bar f|_{\calH^k} + |g|_{\mathscr{H}^k} + |\bar g|_{\mathscr{H}^k}) .\label{eq:impmod2}
\end{align}
\end{itemize}

We are ready to state our main stability Theorem.
\begin{theorem}\label{thm:stability}
There exist $\alpha_0 > 0$, and a map $\calM: \calH^k \times \mathscr{H}^k\to  \calH^k \times \mathscr{H}^k$, such that $\calM(0,0) = 0$ and such that $\calM_1(f,g) +  \calM_2 (f,g) = f$ in a neighborhood of $z^2 + \gamma^2 = \infty$ for all $f \in \calH^k$ and all $g \in \mathscr{H}^k$, and the following holds true. Consider the initial value problem $I(f_{\text{data}}, g_{\text{data}})$
\begin{align}
\label{eq:evo1}
\p_t f + \mathcal{L}_1 f &=  \scrn(f+g, f+g) + \alpha \calL_E(f,g)  + \mathscr{M}(f,g),\\
\p_t g + \calL_2 g &= \calL_3 f,\label{eq:evo2}\\
(f,g)|_{t=0}&=(f_{\text{data}}, g_\text{data}). \label{eq:evo3}
\end{align}
Let us also suppose that the RHS of~\eqref{eq:evo1} always lies in $\mathscr{S}\subset \mathscr{H}^k$, and that $|f_\text{data}|_{\mathscr{H}^k}+|g_\text{data}|_{\mathscr{H}^k} \leq \alpha_0$. Then, setting 
$$
\tilde f_\text{data} := \calM_1(f_\text{data}), \qquad \tilde g_\text{data} := \calM_2(g_\text{data}),
$$
the initial value problem $I(\tilde f_0, \tilde g_0)$ admits a global-in-time solution $(f(s),g(s))$. Moreover, we have
\begin{equation}
    |f(s)|_{\mathscr{H}^k} + |g(s)|_{\mathscr{H}^k} \leq C \exp(- \eta s)
\end{equation}
for all $s \geq 0$, where $C, \eta > 0$.
\end{theorem} 

\begin{remark}\label{rmk:hypo2}
    We check that the assumptions of the Theorem are satisfied in our setup. Our equation is
    \begin{equation}\label{eq:beforemodsta}
    \p_s f + \frac{\tilde \mu_s}{\tilde \mu} f  - \frac{{\tilde \lambda}_s}{\tilde \lambda} z\p_z f  + \tilde \calN(f,f) + \alpha \tilde E(f,f) = 0.
    \end{equation}
    Upon linearization around the exact profile $f_*$ found in Theorem~\ref{thm:profiles} (with scaling parameters $\mu_*$, $\lambda_*$, and setting $f = q + f_*$), we obtain 
    \begin{equation}
    \begin{aligned}
    &\p_s q + \frac{\tilde \mu_s}{\tilde \mu} q - \frac{\tilde \lambda_s}{\tilde \lambda}z\p_z q + \tilde{\mathscr{L}} q + \tilde{\calN}(q,q)\\
    & \qquad +\alpha  (\tilde E(q, q) + \tilde E(f_*, q) + \tilde E(q, f_*) )+ \Big(\frac{\tilde \mu_s}{\tilde \mu} - \mu_* \Big)f_* +  \Big(-\frac{\tilde \lambda_s}{\tilde \lambda} - \lambda_* \mu_* \Big)z \p_z f_*=0.
    \end{aligned}
    \end{equation}
    We then split this equation into the system on two unknowns $f,g$ as follows:
    \begin{equation*}
    \begin{aligned}
        &\p_s f +  f + z\p_z f+ \tscrloc f  + \Big(\frac{\tilde \mu_s}{\tilde \mu} -\mu_* \Big) (f+g) + \Big(- \lambda_* \mu_* - \frac{\tilde \lambda_s}{\tilde \lambda}\Big)z\p_z (f+g) + (\tscrl - \scrl) g + \tilde{\calN}(f+g,f+g)\\
        &\qquad +\alpha  (\tilde E(f+g, f+g) + \tilde E(f_*, f+g) + \tilde E(f+g, f_*) ) + \Big(\frac{\tilde \mu_s}{\tilde \mu} - \mu_* \Big)f_* +  \Big(-\frac{\tilde \lambda_s}{\tilde \lambda} - \lambda_* \mu_* \Big)z \p_z f_*\\
        & \qquad + (\mu_*-1) (f+g) +(\mu_* \lambda_*-1)  z\p_z (f+g) =0\\
        &\p_s g +  g +  z\p_z g + \scrl g = - \scrnonloc f.
    \end{aligned}
    \end{equation*}
    We let
    \begin{align}
        & \calL_1 q := q+z\p_z q+ \tscrloc q,\\
        & \calL_2 q := q+ z\p_z q+ \scrl q,\\
        &\calL_3 q :=  - \scrnonloc q,\\
        &\scrn(q,q) :=- \tilde \calN(q,q) - \alpha \tilde E(q,q),\\
        &\calL_E(f,g) := -(\tilde E(f+g,f_*) + \tilde E(f_*, f+g)) - \alpha^{-1}(\tscrl g - \scrl g)\\
        & \qquad \qquad -\alpha^{-1}(\mu_*-1) (f+g) -\alpha^{-1}(\mu_* \lambda_*-1)  z\p_z (f+g) ,\\
        &\mathscr{M}(f,g) := -\Big(\frac{\tilde \mu_s}{\tilde \mu} -\mu_* \Big) (f+g) - \Big(- \lambda_* \mu_* - \frac{\tilde \lambda_s}{\tilde \lambda}\Big)z\p_z (f+g) -\Big(\frac{\tilde \mu_s}{\tilde \mu} - \mu_* \Big)f_* -  \Big(-\frac{\tilde \lambda_s}{\tilde \lambda} - \lambda_* \mu_* \Big)z \p_z f_*
    \end{align}
    With these definitions, equation~\eqref{eq:beforemodsta} becomes
    \begin{align}
        \p_t f + \mathcal{L}_1 f &=  \scrn(f+g, f+g) + \alpha \calL_E(f,g)  + \mathscr{M}(f,g),\label{eq:evo1post}\\
\p_t g + \calL_2 g &= \calL_3 f,\label{eq:evo2post}
    \end{align}
    Here, $\tilde \mu$ and $\tilde \lambda$ are determined by the condition that the RHS of~\eqref{eq:evo1post} belongs to the space $\mathscr{S}$. More precisely, we choose $\tilde \mu$ and $\tilde \lambda$ to satisfy
    \begin{equation}\label{eq:modulaz}
    \begin{aligned}
        &\Big(\frac{\tilde{\mu}_s}{\tilde \mu} -\mu_* \Big) (f^*_2+g_2)|_{z=0} = (\calN(f+g,f+g))_2|_{z=0} + \alpha (\calL_E(f,g))_2|_{z=0}, \\
        &\Big(-\frac{\tilde \lambda_s}{\tilde \lambda} - \lambda_* \mu_* \Big) L_{\mathscr{S}}(z \p_z(f^*+f+g))\\
        & \qquad = -\Big(\frac{\tilde{\mu}_s}{\tilde \mu} -\mu_* \Big) L_{\mathscr{S}}(f^* + f+ g) + L_{\mathscr{S}}( \scrn(f+g, f+g) + \alpha \calL_E(f,g) ). 
    \end{aligned}
    \end{equation}
    We then substitute these expressions in the expression for $\mathscr{M}$, which implies that the expression of $\mathscr{M}$ is local-in-time, whereas the parameters $\tilde \mu$ and $\tilde \lambda$ satisfy a system of (decoupled) ODEs with forcing. We then see that the term $\mathscr{M}$ satisfies the estimate~\eqref{eq:impmod1}.
    
    We finally project each term on the RHS of~\eqref{eq:evo1post} using $\mathbb{Q}$, which takes the system~\eqref{eq:evo1post}--\eqref{eq:evo2post} to a form compatible with the hypotheses of Theorem~\ref{thm:stability}.
\end{remark}

We are now going to prove Theorem~\ref{thm:stability}.

\begin{proof}[Proof of Theorem~\ref{thm:stability}]

We decompose $F_* = \Pi(F_*) + (F_* - \Pi(F_*))=:F_*^s + F_*^u$.
We write
\begin{equation}\label{eq:gpre}
    \p_s \exp(s \calL_2)g = \exp(s \calL_2 )F_*^s L_1(f) + \exp(s \calL_2 )F_*^u L_1(f),
\end{equation}
 and correspondingly $g_{s = 0} = \Pi(g|_{s=0}) + (g|_{s=0} - \Pi (g|_{s=0}))$.
 
Using the Duhamel formula, we obtain the following solution\footnote{
Note that, if we assume a-priori that $f(s)$ decays exponentially, due to the semi-group estimate on the stable part, and the exponential decay on the unstable part, $g(s)$ will also decay exponentially.
} to~\eqref{eq:gpre}:
\begin{equation}
\begin{aligned}
     &g(s) = \int_0^s \exp(-(s-s') \calL_2 )F_*^s L_1(f)  ds' - \int_s^\infty \exp(-(s-s') \calL_2 )F_*^u L_1(f) ds'.
\end{aligned}
\end{equation}
We are led to consider the following iterative procedure. We define $(f_n, g_n)$ inductively as follows. Let $f_0 \equiv 0$, and for $n \geq 0$,
\begin{align}
\label{eq:evok1}
&\p_s f_{n+1} + \mathcal{L}_1 f_{n+1} =  \scrn(f_{n+1}+g_{n+1}, f_{n+1}+g_{n+1}) + \alpha \calL_E(f_{n+1},g_{n+1}) + \mathscr{M}(f_{n+1},g_{n+1}),\\
&g_{n+1}(s) = \int_0^s \exp(-(s-s') \calL_2 )F_*^s L_1(f_n) ds' - \int_s^\infty \exp(-(s-s') \calL_2 )F_*^u L_1(f_n) ds'.\label{eq:duha}
\end{align}
We impose the following initial conditions for $(f_n, g_n)$:
\begin{align*}
&g_n|_{s = 0} :=  - \int_0^\infty \exp(-s' (-\calL_2) )(F_*^u L_1(f_{n-1}))ds',\\
&(f_n)_1|_{s = 0} := (f_\text{data})_1 - (1-\chi_a(z)\chi_a(\gamma)) (g_n|_{s=0})_1|_{s=0} + z a^{-1} \kappa_a(g_n) \chi_1(z) \chi_1(\gamma) ,\\
&(f_n)_2|_{s = 0} := (f_\text{data})_2 - (1-\chi_a(z)) (g_n)_2|_{s=0},\\
&(f_n)_3|_{s = 0} := (f_\text{data})_3 - (1-\chi_a(z)\chi_a(\gamma)) (g_n)_3|_{s=0}.
\end{align*}
Here, $\chi_a(w) = \chi(w/a)$, where $\chi$ is a compactly supported non-negative bump function which is identically $1$ on the interval $[0,1]$ and vanishes on the interval $[2, \infty)$. Moreover, $\kappa_a(g_n|_{s=0})$ is constructed\footnote{The reason for these choices is that we want to ensure, at the level of initial data, that the first and third component of $f_n + g_n$ (which are non-smooth in the angle $\theta$ at $\pi/2$) coincide with $f_{\text{data}}$ when either $z$ is large or $\theta$ is sufficiently close to $\pi/2$. This will allow us to choose $f_{\text{data}}$ to cancel exactly the non-smooth part of the profile at $\theta = \pi/2$, and the unbounded part of the support (in $z$).} as the continuous linear map from $\scrH^k$ to $\R$ which ensures that $f_n|_{s=0} \in \mathscr{S}$. It is immediate to show that
$$
|\kappa_a(g_n|_{s=0})| \leq C |g_n|_{s=0}|_{\scrH^k}
$$
for a constant $C > 0$ independent of $a$.

We claim that for $\alpha_0$ sufficiently small and for $a > 0$ sufficiently large,
\begin{equation}\label{eq:expbd}
    |f_n(s)|_{\calH^k} \leq  \alpha^{\frac{1}{2}}_0 \exp(-s\eta/2)
\end{equation}
holds true for all $n$ and $s \geq 0$. We proceed by induction. From~\eqref{eq:duha} and the induction hypothesis, for arbitrary $\tilde k > 0$, we have $|g_{n+1}(s)|_{\scrH^{k+ \tilde k}} \leq C  \alpha^{\frac12}_0 \exp(- s\eta/3)$.
We then take the $\tilde{\calH}^k$ inner product of equation~\eqref{eq:evok1} with $f_{n+1}$, and we obtain, following the same steps as in the proof of Theorem~\ref{thm:profiles}, by possibly redefining $\eta > 0$ and choosing $a> 0$ and $\alpha_0$ small, for all $s_2\geq s_1 \geq 0$,
\begin{equation}
     |f_{n+1}(s_2)|_{\tilde{\calH}^k} + \eta\int_{s_1}^{s_2} |f_{n+1}(s')|_{\tilde{\calH}^k} ds' \leq |f_{n+1}(s_1)|_{\calH^k} + C \alpha_0 \exp(-2/3 \eta s_1),
\end{equation}
This implies, via the Gr\"onwall inequality, $|f_{n+1}(s)|_{\calH^k}^2 \leq C \alpha_0 \exp(- 2/3 \eta s)$. Upon taking $\alpha_0$ sufficiently small, this closes the induction on $n$.

It remains to show that the sequence $\{f_n\}$ is Cauchy in an appropriate space\footnote{In particular, this implies that the sequence $\{g_n|_{s=0}\}$ is Cauchy in $\scrH^k$, which determines the map $\calM$.}. We claim that there exists $\eta_1 >0$ such that, for $\alpha_0$ sufficiently small\footnote{Here, $L^\infty_s := L^\infty([0,\infty); \calH^k)$.},
\begin{equation}\label{eq:cauchy}
    |\exp(\eta_1 s) (f_{n+1}(s) - f_n(s))|_{L^\infty_s(\calH^{\bar k})} \leq  \alpha^{\frac{1}{4}}_0 |\exp(\eta_1 s) (f_{n}(s) - f_{n-1}(s))|_{L^\infty_s(\calH^{\bar k})},
\end{equation}
holds true for some $\bar k \leq k$, and all $n \geq 0$, $n \in \N$.

Denote $a_n := f_n - f_{n-1}$, $b_n := g_{n}-g_{n-1}$. We compute the equation for $a_n$, $b_n$:
\begin{align}
\label{eq:evodiff}
&\p_s a_{n+1} + \mathcal{L}_1 a_{n+1} =  \scrn(a_{n+1}+b_{n+1}, f_{n+1}+g_{n+1})+\scrn(f_{n}+g_{n}, a_{n+1}+b_{n+1}) \\
&\qquad + \alpha \calL_E(a_{n+1},b_{n+1}) + \mathscr{M}(f_{n+1},g_{n+1})- \mathscr{M}(f_{n},g_{n}),\\
&b_{n+1}(s) = \int_0^s \exp(-(s-s') \calL_2 )F_*^s L_1(a_n) ds' - \int_s^\infty \exp(-(s-s') \calL_2 )F_*^u L_1(a_n) ds'.\label{eq:duhadiff}
\end{align}
We then proceed as in the proof of inequality~\eqref{eq:expbd}, the only modification being that we are allowed to lose derivatives in the equation for $a_n$. Indeed, we take the $\tilde \calH^{\bar k}$ inner product of equation~\eqref{eq:evodiff} with $a_{n+1}$, and proceed as before, using in addition the bound~\eqref{eq:impmod2}. We obtain:
\begin{equation}
     |a_{n+1}(s_2)|_{\tilde{\calH}^{\bar k}} + \eta\int_{s_1}^{s_2} |a_{n+1}(s')|_{\tilde{\calH}^{\bar k}} ds' \leq  C \alpha^{\frac 12}_0 \exp(-4/3 \eta s_1) |\exp(\eta_1 s) a_n|_{L^\infty_s(\calH^{\bar k})}.
\end{equation}
The conclusion follows from Gr\"onwall's inequality, taking $\alpha_0$ sufficiently small.

An analysis of the ODEs satisfied by $\tilde \mu$ and $\tilde \lambda$ (Equation~\eqref{eq:modulaz}) finally reveals \footnote{Note that we initialize $\tilde \mu(0) = \mu_*$ and $\tilde \lambda(0) = \lambda_*$.} that $\lim_{s\to \infty} \tilde \mu = +\infty$, and  $\lim_{s\to \infty} \tilde \lambda = 0$, which concludes the proof.
\end{proof}

\subsection{Proof of Theorem~\ref{thm:boussinesq}}
We are ready to prove Theorem~\ref{thm:boussinesq}.

\begin{proof}[Proof of Theorem~\ref{thm:boussinesq}]
    In view of Remark~\ref{rmk:hypo1}, we can apply Theorem~\ref{thm:profiles} to our setup, and we obtain a self-similar profile with $\alpha > 0$, which we call $f^\#$. Notice that the corresponding solution \emph{does not satisfy} the desired smoothness and support properties in the statement of Theorem~\ref{thm:boussinesq}. To produce a solution whose initial data moreover satisfies the required smoothness in $\theta$ and compact support in $r$, we use the stability Theorem~\ref{thm:stability}. We checked that all hypotheses are satisfied in our setup (this was done in Remark~\ref{rmk:hypo2}). The only thing left is to ensure the smoothness and support properties of the initial condition. Due to the definition of the space $\calH^k$ and the associated spaces, this is a consequence of the following observation. Let $\mathring{\chi}: [0, \infty) \to \R$ be a smooth and non-negative cut-off function which is identically $1$ on the interval $[0,1]$ and identically $0$ on the interval $[2, \infty)$. Recall that $\gamma = \tan \theta$. Let $F_{j}(z, \gamma) \in \mathcal{H}^k_{,j}$ whenever $j \in \{0,1,3\}$, and let $F_{(2)}(z) \in \mathcal{H}^k_{,2}$. Then, given a parameter $a > 0$, we let
    \begin{align}
    &F_{j, \text{cut}} := (1-\mathring{\chi}(z/a) \mathring{\chi}(\gamma/a)) F_{j} \qquad \text{if } j \in \{0,1,3\},\\
    &F_{2, \text{cut}} := (1-\mathring{\chi}(z/a)) F_{2}(z).
    \end{align}
    In this setting, the crucial observation is the following. We have that
    \begin{align}
    &\lim_{a \to \infty} |F_{j, \text{cut}}|_{\calH^k_{,j}} = 0 \qquad \text{if } j \in \{0,1,3\},\\
    &\lim_{a \to \infty} |F_{2, \text{cut}}|_{\calH^k_{,2}} = 0.
    \end{align}
    This in particular implies that the function $F_{j, \text{cut}}$, which is identical to the original function $F_{j}$ in a neighborhood of $z = \infty$ and in a neighborhood of $\theta = \pi/2$, can be made arbitrarily small when measured in the space $\calH^k_{,j}$. We apply Theorem~\ref{thm:stability} with $(f_{\text{data}})_{j} := -f^\#_{j, \text{cut}}$ for $j = 1,2,3$ to conclude. Note that the initial data for the full system is $f^\# + f_{\text{data}} = f^\#-f^\#_{\text{cut}}$, which is smooth in the angle $\theta$ at $\pi/2$ and compactly supported in $z$.
\end{proof}

\section{Proof of Theorem 1.3}\label{sec:eulerproof}
In this Section, we indicate how to adapt the argument to prove the statement about the Euler system in Theorem~\ref{thm:euler}.

\subsection{Setup and modulation}
Recall the 3d Euler system in axisymmetry:
\begin{align}
    &\p_t\Big(\frac{\omega_\theta}{r} \Big) + u \cdot \nabla \Big(\frac{\omega_\theta}{r} \Big) = \p_{x_3}[(r^{-1} u_\theta)^2],\\
    &\p_t(r u_\theta) + u \cdot \nabla (r u_\theta) = 0,\\
    &u = (u_r, u_3) = \frac 1 r (\p_{x_3} \psi_{EU}, - \p_r \psi_{EU}),\\
    &\frac{1}{r}\p_r \Big(\frac 1 r \p_r \psi_{EU} \Big) + \frac 1 {r^2} \p^2_{x_3}\psi_{EU} = \frac{\omega_\theta}{r}.
\end{align}
We rename variables in the following way
\begin{align*}
x = r, \qquad y = x_3, \qquad
\omega = \omega_\theta, \qquad \rho = r^{-1}u_\theta^2, \qquad \psi = r^{-1} \psi_{EU}.
\end{align*}
This gives, for $x > 0$:
\begin{align}
    &\p_t \omega + u \cdot \nabla \omega - \frac 1 x u_1 \omega =  \p_y \rho,\\
    &\p_t \rho + u \cdot \nabla\rho+ 3 x^{-1} u_1 \rho = 0,\\
    &u = (u_1, u_2) =  (\p_y \psi, - \p_x \psi - x^{-1}\psi),\\
    &\Delta \psi + x^{-1} \p_x\psi - x^{-2} \psi = \omega,
\end{align}
Here, the elliptic operator in the last line is understood with zero Dirichlet boundary conditions at $x = 0$.

We re-center our coordinate system at $x = \xi(t)$ (we write the equations in terms of the coordinates $\tilde t = t$, $\tilde x := x- \xi(t)$, $\tilde y = y$). We have $\p_t = \p_{\tilde t} - \p_t \xi(t) \p_{\tilde{x}}, \ \p_x = \p_{\tilde x}$, and $\p_y = \p_{\tilde y}$. We also let $\xi_t := \p_t \xi$. 

Additionally, we add a positive cut-off function to $\rho$ to ensure non-negativity of $\rho$ everywhere. Let $\mathring{\chi}(w)$ be a smooth, positive cut-off function which is identically 1 for $w \in [0,1]$. Given positive parameters $a_1, a_2$, and given functions $\tilde \xi$ and $\tilde u$, we let
\begin{align}
\chi_p(t=0) = a_1 \mathring{\chi}(a_2 r).
\end{align}
We let $\rho := \tilde \rho + \chi_p$. In these new coordinates\footnote{Suppressing the tilde over the coordinates $(t,x,y)$ for convenience.}, the system reduces to:
\begin{align}
    &\p_t {\omega} - {\xi}_t  \p_x {\omega}+ {u} \cdot \nabla {\omega} - \frac 1{x+{\xi}} {u}_1 {\omega} =  \p_y \tilde{\rho} + \p_y \chi_p,\\
    &\p_t \tilde{\rho} - {\xi}_t  \p_x \tilde{\rho}+{u} \cdot \nabla\tilde{\rho}+ 3\frac{1}{x+ {\xi}} {u}_1 \tilde{\rho} = 0,\\
    &\p_t \chi_p - {\xi}_t  \p_x \chi_p +{u} \cdot \nabla \chi_p +  3\frac{1}{x+ {\xi}} {u}_1 \chi_p= 0,\label{eq:tildesys}\\
    &{u} = ({u}_1, {u}_2) =  (\p_y \psi , - \p_x \psi - (x+{\xi})^{-1}\psi),\\
    &\p^2_x \psi + \p_y^2 \psi + (x+ {\xi})^{-1} \p_x\psi - (x+ {\xi})^{-2} \psi = {\omega}.
\end{align}
We let $E_{p, {\omega}}:= \p_y \chi_p$. Note importantly that, in what follows, by a slight abuse of notation, we will also suppress the tilde on $\tilde \rho$ in system~\eqref{eq:tildesys}.

Since the system exhibits a translational behavior, we will modulate the quantity $\xi$ according to linear motion along the $x$-axis. More precisely, we set the following evolution equation for $\xi$:
\begin{equation}\label{eq:modulation}
    \p_t \xi = \p_y \psi(0,0).
\end{equation}
This immediately induces a decomposition of the stream function $\psi$ whose  part supported on the first Fourier shell\footnote{Since $\psi$ is $C^{2,\alpha}$ in a neighborhood of $0$, by the symmetry conditions, we will see that $\mathring{\psi}$ vanishes to order at least $2$ at $0$.} encompasses the linear motion at the origin, as follows. Let $\chi(s)$ be a smooth, positive cut-off function which is identically 1 for $s \in [0,1]$, and let $\chi := \mathring{\chi}( r)$. We let
$$
\mathring{\psi} := \psi -\underbrace{ y \p_y \psi(0,0)\chi}_{:= \chi_
\psi} 
$$

\subsection{Decomposition of the Euler system into Boussinesq plus remainder}

The aim of this section is to further decompose the modulated system~\eqref{eq:tildesys} in order to write it as a perturbation of the Boussinesq system, to which our general framework will apply. The main obstacle is the lack of symmetry in the $x$-variable. Our setup will compensate for this lack of symmetry. 

We decompose the full stream function $\psi$ into its main (``Boussinesq'') part $\psi_B$ and the error $\psi_E = \psi_{E_1} + \psi_{E_2}.$  Whenever $\eta$ is a function of $x$ and $y$, we let $\mathbb{A}_\xi(\eta) := \eta(x,y) + \eta(-2\xi-x,y)$. Let us start by a simple remark.

\begin{remark}
Note that, if $\omega$ is supported in $x > - \xi$, then 
\begin{align*}
\Delta_D \psi = \omega  \iff \Delta_{\R^2} \psi = \mathbb{A}_{- \xi}  \omega.
\end{align*}
Here, $\Delta_D$ is the Laplacian with zero Dirichlet boundary conditions at $x = - \xi$.
\end{remark}
We define $\psi_B$ and $\psi_E$ as solutions to the following (upper triangular) system:
\begin{equation}\label{eq:biotpre}
\begin{aligned}
&\Delta_{\R^2} \psi_B  = \mathbb{A}_0 (\omega), \\
&\Delta_{\R^2} \psi_{E_1} =( \mathbb{A}_{-\xi}\omega - \mathbb{A}_0 (\omega) )\\
&L \, \psi_{E_2}  = - ( (x+\xi)^{-1} \p_x(\psi_B + \psi_{E_1}) - (x+\xi)^{-2} (\psi_B + \psi_{E_1})),
\end{aligned}
\end{equation}
where the operator $L := \Delta_D  + (x+\xi)^{-1} \p_x - (x+\xi)^{-2}$ is understood with zero Dirichlet boundary conditions at $x = - \xi$. In this setting, $\psi = \psi_B + \psi_{E_1} + \psi_{E_2}$ solves $\p^2_x \psi + \p_y^2 \psi + (x+ {\xi})^{-1} \p_x\psi - (x+ {\xi})^{-2} \psi = {\omega}$.

\begin{remark}
    The term $\psi_{E_1}$ is what sets this argument apart from previous perturbation arguments from Boussinesq to 3d Euler in axisymmetry.
\end{remark}
We let 
\begin{align}\label{eq:ueub}
u_B = (\p_y \psi_B, - \p_x \psi_B), \qquad u_E = \Big(\p_y \psi_E, - \p_x \psi_E - \frac{1}{x + \xi} \psi_E - \frac{1}{x + \xi} \psi_B\Big), \qquad     \mathring{u}_E :=  (\p_y \mathring{\psi}_E, (u_E)_2).
\end{align}
These considerations yield the main lemma of this section:

\begin{lemma}[Splitting lemma]\label{lem:split1}
The system~\eqref{eq:tildesys} reduces to the following form:
\begin{align}
    &\p_t \omega + u_B \cdot \nabla \omega =  \p_y \rho + E_\omega + E_{p,\omega},\\
    &\p_t \rho+u_B \cdot \nabla\rho = E_\rho,\\
    &\p_t \chi_p +u_B \cdot \nabla\chi_p = E_{\chi_p}, \label{eq:aftersplit}\\
    &u_B = ((u_B)_1, (u_B)_2) =  (\p_y \psi_B , - \p_x \psi_B),\\
    &\Delta_{\R^2} \psi_B = \mathbb{A}_0(\omega),
\end{align}
where $(u_B, u_E, \mathring{u}_E)$ are defined according to display~\eqref{eq:ueub} and $(\psi_B, \psi_E, \mathring{\psi}_E)$ are defined according to display~\eqref{eq:biotpre}, and in addition we defined
\begin{equation}\label{eq:errori}
\begin{aligned}
    &E_\omega =  -\mathring{u}_E \cdot \nabla \omega  - W_\psi \p_x \omega   + \frac 1 {x+ \xi} (u_E + u_B)_1 \omega,
    &E_{p, \omega} =  \p_y \chi_p,\\
    &E_\rho = -\mathring{u}_E \cdot \nabla \rho - W_\psi \p_x \rho  - \frac 3 {x+ \xi} (u_E + u_B)_1 \rho,
    &E_{\chi_p} = -\mathring{u}_E \cdot \nabla {\chi_p} - W_\psi \p_x {\chi_p}  - \frac 3 {x+ \xi} (u_E + u_B)_1 {\chi_p},
\end{aligned}
\end{equation}
together with $W_{\psi} := (\mathring{\chi}(r)-1) \p_y \psi(0,0) + \sin^2 r \p_r  \mathring{\chi}(r) \p_y \psi(0,0)$ and $\psi_E =  \psi_{E_1} + \psi_{E_2}$.
\end{lemma}

\subsection{The system in self-similar formulation}
At this point, in order to perform the $r^\alpha$ transformation, we will rewrite system~\eqref{eq:aftersplit} in the coordinates $(t, R, \theta)$, where $(r,\theta)$ are the standard polar coordinates, and $R = r^\alpha$.

First, we define $q$ such that\footnote{In the course of the proof, we will show that $\rho$ vanishes to infinite order at $\theta = \pm\pi/2$, and $\rho > 0$ whenever $x \geq 0$, and $\rho < 0$ when $x \leq 0$.} $-xq^2 = \rho$, and introduce the $r^\alpha$ transformation by setting $R = r^\alpha$, and
\begin{align}
    &\tilde{\Omega}_0(t,R,\theta) := \omega(t,r,\theta), 
    &\tilde{P}(t,R,\theta) := q(t,r,\theta),\\
    &\tilde{X}(t,R,\theta) := \chi_p(t,r,\theta),
    &\tilde{\Psi}(t,R,\theta) := r^{-2} \psi(t,r,\theta).
\end{align}
We introduce the self-similar coordinates and quantities by the following relations: 
\begin{equation}
\begin{aligned}
    &\frac{ds}{dt} = \tilde{\mu}(s), \qquad z = \frac{R}{\tilde{\lambda}(s)},\\
    &\tilde{\Omega}(t,R, \theta) = \alpha \tilde{\mu}(s) \Omega_0(s, z, \theta), \\
    &\PP_0(\tilde{P}(t,R, \theta)) = \tilde{\mu}(s) P_0(s, z), & \PP^\perp_0(\tilde{P}(t,R, \theta)) = \alpha \tilde{\mu}(s)  P_1(s, z, \theta),\\
    &\PP_1(\tilde{\Psi}_B(t,R, \theta)) =\tilde{\mu}(s) \Psi_0(s, z, \theta), & \PP_1^\perp(\tilde{\Psi}_B(t,R, \theta)) = \alpha\tilde{\mu}(s)  \Psi_1(s, z, \theta),\\
    &\tilde{\Psi}_E(t, R, \theta) = \tilde{\mu}(s) \Psi_E(s, z, \theta),\\
    &\tilde{X}(t,R,\theta) = \tilde{\mu}(s) X(s,z,\theta).
\end{aligned}
\end{equation}
Using the above definitions,  we rewrite the system in terms of $f = (f_1, f_2, f_3, f_4) = (\Omega_0, P_0, P_1, X)$. The system~\ref{eq:aftersplit} reduces to
\begin{equation}\label{eq:mastereu}
    \p_s f + \frac{\tilde \mu_s}{\tilde \mu} f  - \frac{{\tilde \lambda}_s}{\tilde \lambda} z\p_z f  + \calN(f,f) + \alpha E(f,f) + \calE(f,f) = 0,
\end{equation}
where $(\calN(f,g))_{i = 1, 2,3}$ and $(E(f,g))_{i = 1, 2, 3}$ are as in Section 2.1, and 
\begin{align}
    \calN(f,g)_4 := \frac 12 \ellud(\boldom_0[g]) D_\theta X[f] ,\\
    E(f,g)_4 := u_B[g] \cdot \nabla \chi_p[f] - \calN(f,g)_4.
\end{align}
Moreover, $\calE_{1,2,3,4}(f,g)$ are the remaining (Euler) error terms from~\eqref{eq:aftersplit}.

\subsection{Angular and radial cut-off of the profile}

Note that the original profile $(\Omega_0^*, P_0^*, P_1^*)$ only possesses fractional regularity at $\theta = \pi/2$. In the Boussinesq case, since the line $x =0$ is preserved by the associated Lagrangian flow, our framework was able to deal with this lack of regularity (essentially by employing a space with zero homogeneity for higher derivatives only at the set $\theta = \pi/2$). In the Euler case, the line $x =0$ is no longer an invariant set, we introduce a cut-off to handle the high-derivative contributions at (or close to) $\theta = \pi/2$.

More precisely, let $\theta_0 > 0$ be a parameter. Let $\chi (\theta): (-\infty, \infty) \to [0,1]$ be a smooth, even, positive cut-off function which is identically $1$ for $\theta > 1$ and vanishes on the set $(0,1/2)$. We then define $\chi_{\theta_0}: (- \pi, \pi) \to [0,1]$ as follows:
$$
\chi_{\theta_0}(\theta) := \chi(-(x-\pi/2)/\theta_0) + \chi((\pi/2 + x)/a) - 1.
$$
This function vanishes in a neighborhood of $\theta = \pi/2, -\pi/2$. Let also $\mathring{\chi}(w)$ be a smooth cut-off function identically $1$ in the interval $[0,1]$ and identically $0$ on the interval $[2,\infty)$. 

Recall the profile $\Omega^*$, $P_0^*$, $P_1^*$. We define
\begin{align}
    &\Omega^{**} := \chi_{\theta_0}(\theta) \mathring{\chi}(4 r/R_0)\Omega^*, \qquad \qquad  P_0^{**} := \mathring{\chi}(4 r/R_0)P_0^*,  \qquad \qquad P_1^{**} := \chi_{\theta_0}(\theta)\mathring{\chi}(4 r/R_0) P^*_1,\\
    &\Omega^{d} := (1-\chi_{\theta_0}(\theta)\mathring{\chi}(4 r/R_0))\Omega^*,\qquad P_0^{d} := (1-\mathring{\chi}(4 r/R_0))P_0^*, \qquad  P_1^{d} := (1-\chi_{\theta_0}(\theta)\mathring{\chi}(4 r/R_0)) P^*_1.
\end{align}
We define the corresponding four-vectors as $f_{**} := (\Omega_0^{**}, P_0^{**}, P_1^{**}, 0)$, and $f_d := (\Omega_0^{d}, P_0^{d}, P_1^{d}, 0)$.

\begin{remark}
    There are three expansion parameters: $\alpha$, $R_0$, $\theta_0$. When $\alpha = 0$, $R_0 =\infty$, $\theta_0 = 0$, the system reduces to Boussinesq with $\alpha = 0$, for which we have an explicit profile.
\end{remark}

In the sequel, we are going to linearize around the above cut-off profile $f_{**}$. We are going to do so by being careful that the analysis of the linear operator in the Boussinesq case carries over. More precisely, we write the nonlinear terms in~\eqref{eq:mastereu} as follows:
\begin{align}
    &\tilde \calN(f + f_{**}, f+ f_{**}) + \alpha E(f+ f_{**}, f+ f_{**}) + \calE(f + f_{**}, f+ f_{**})\\
    &= \tilde\calN(f, f_*) + \tilde \calN(f_*, f) - \tilde\calN(f, f_d) -\tilde \calN(f_d, f) + \tilde \calN(f_{**},f_{**}) + \tilde\calN(f,f)\\
    & \quad + \alpha \tilde E(f+ f_{**}, f+ f_{**}) + \calE(f + f_{**}, f+ f_{**}).
\end{align}
Then, the estimates in Lemma 6.1 are enough to handle the nonlinear terms, provided that we can extend them to $\calE$. The analysis of the linear operator is the same as in the Boussinesq case, since we are writing the linear operator in terms of the non-smooth profile $f_*$ (the analysis in the Boussinesq case is conducted on the linear operator involving $\tilde\calN(f, f_*) + \tilde \calN(f_*, f) $).

\subsection{Construction of a solution}
Having described our setup, we proceed to outline the last ingredients of our procedure, which will essentially rely on the construction of an element of the unstable manifold for the equations in self-similar form.
\begin{enumerate}
\item Show that the nonlinear estimates required to run the argument in the Boussinesq part carry over to the new error terms generated in the Euler case. This is indicated in Section~\ref{sec:eulernonlinearterms}.
\item Show that the contributions from the modulation (the terms involving $\xi$) are uniformly controlled. These will depend on uniform support properties of the solution, which we will have to control in a bootstrap argument. This is indicated in Section~\ref{sec:suppmodul}
\item Show the support properties and a uniform (in time) estimate for $u$ in $L^\infty$. In particular, show that $\rho$ is supported away from the $x$-axis and from the support of $X$ (which is then used to show that $\rho \geq 0$ for $x \geq 0$ and $\rho \leq 0$ otherwise). This is indicated also in Section~\ref{sec:suppmodul}
\end{enumerate}

We proceed to describe each item.

\subsection{Nonlinear and elliptic estimates in the Euler case}\label{sec:eulernonlinearterms}

First, we describe the functional framework for the nonlinear estimates (encompassing the component $f_4$ as well).

Recall the definitions in Section~\ref{sec:spaces}. We have the following definitions.

\begin{definition}
We let the following inner products, with $c_0, c_1, c_2, c_3, c_4$ positive constants (in these definitions, $F$ and $G$ are scalar functions, whereas $f$ and $g$ are 4-vectors):
\begin{align}
    & (F, G)_{{\dot \calH}^k_{,4}}:=(F, G)_{{\dot \calH}^k_{,3}},\\
    & (F,G)_{\calH_{,\text{low};4}}:=(F,G)_{\calH_{,\text{low};3}},\\
    & (F,G)_{\calH^k_{,4}}:= (F,G)_{\dot \calH^k_{,4}} + (F,G)_{\calH_{, \text{low};4}},\\
    &(f, g)_{{\calH}^k_{,\text{high}}} := c_0 (f_2 + \alpha f_3,g_2 + \alpha g_3)_{\dot \calH^k_{,0}}+ c_1 (f_1,g_1)_{\dot \calH^k_{,1}} + c_2  (f_2 ,g_2)_{\dot \calH^k_{,2}}+c_3 (f_3,g_3)_{\dot \calH^k_{,3}} + c_4  (f_4,g_4)_{\dot \calH^k_{,4}}.
\end{align}
\end{definition}
We redefine the main inner product spaces as follows.
\begin{definition}
 Whenever $c_5 > 0$, we define
\begin{align}
    &(f, g)_{\tilde{\calH}^k} := (f, g)_{{\calH}^k_{,\text{high}}} + c_5 (f, g)_{\tilde{\calH}_{,\text{low}}}.  
\end{align}
In addition, we define, for $d_1, d_2, d_3, d_4, d_5> 0$ (replacing each of the instances of $c_i$ with $d_i$ in the definition of $(f, g)_{\tilde{\calH}^k_{,\text{high}}} $):
\begin{align}
    (f, g)_{{\calH}^k} &:= (f, g)_{\calH^k_{,\text{high}}} + d_5 (f, g)_{\calH_{,\text{low}}}.  
\end{align}
Here, we used
\begin{align}
&(f, g)_{\tilde{\calH}_{,\text{low}}} := (\hat f, \hat g)_{\mathbf{B}} + (f_4,g_4)_{\calH_{,\text{low};4}}\\
&(f, g)_{{\calH}_{,\text{low}}} := (f_2 + \alpha f_3,g_2 + \alpha g_3)_{\calH_{, \text{low};0}}+(f_1,g_1)_{\calH_{, \text{low};1}} +(f_2,g_2)_{\calH_{, \text{low};2}} + (f_3,g_3)_{\calH_{, \text{low};3}}+ (f_4,g_4)_{\calH_{,\text{low};4}},
\end{align}
where $\hat f := (f_1, f_2, f_3)$. We denote by $|\cdot|_{\tilde{\calH}^k}$ (resp. $|\cdot|_{\calH^k}$) the norms induced by the above inner products).
\end{definition}

We redefine the space $\scrH^k$, which does not require vanishing of the second component $f_2$ at $z=0$. To this end, let the projection $\mathbb{Q}$ be defined as
\begin{equation}
\mathbb{Q}(f) := c(f) f_* + d(f) z\p_z f_*,
\end{equation}
where $c(f) = \frac{1}{B_*} f_2|_{z=0}$, and $d(f)$ is such that $f - \mathbb{Q}f \in \mathscr{S}$ (the dissipative subspace).

\begin{definition}
We let
\begin{equation}
    (f,g)_{\scrH^k} := (f - \mathbb{Q}(f), g - \mathbb{Q}(g))_{\tilde{\calH}^k} + c(f) c(g) + d(f) d(g).
\end{equation}
We let the associated norm be $|\cdot|_{\mathscr{H}^k}$. 
\end{definition}
Finally, whenever $H$ is a function space, we define $\underline{H}$ to be the space localized to the set $\{ r \leq 4 R_0 \} \cap A_{\theta_0}$, where  $A_{\theta_0}$ is the following set:
\begin{equation}
    A_{\theta_0} := \{(z,\theta): |\theta - \pi/2| > r^{2} \theta_0, \text{ and } |\theta + \pi/2|>r^2 \theta_0 \}.
\end{equation}
We have the following lemmas.

\begin{lemma}[Nonlinear estimates, Euler case]\label{lem:nleuler}

There exist $R_0$, $\theta_0$ and $\xi_0 > 0$ such that the following holds true.  Let $f \in \calH^k$, $g, h \in \scrH^k$ supported in $\{x \geq - \xi_0/2\} \cap \{r \leq R_0\}$, such that $f_1, f_2 + \alpha f_3$ and $g_1, g_2 + \alpha g_3$ are supported in the set $A_{\theta_0}$. The following inequalities hold true with $C_k >0$ depending on $k$:
\begin{align}
    &|(f,  \mathbb{Q} \calE(f,g))_{{\tilde \calH}^k}|\leq C \alpha |f|_{\calH^k}^2 |g|_{\scrH^{k}}.\label{eq:nlerrbettereuler}
\end{align}
In addition, we have the inequality with a loss:
\begin{align}
  &|(f,\mathbb{Q} \calE(g,h))_{{\tilde \calH}^k}|\leq  C\alpha|f|_{\calH^k}|g|_{\scrH^{k+1}}|h|_{\scrH^{k}}.\label{eq:nlerreuler}
\end{align}
Here, we used the definition of $\mathbb{Q}$ from equation~\eqref{eq:qprojection}.
\end{lemma}
\begin{remark}
   Note: the loss of a factor of $\cos \theta$ in the Euler error terms is compensated here by the support properties of $f_2+\alpha f_3$ and $f_1$.
\end{remark}

\begin{proof}
    The proof of Lemma~\ref{lem:nleuler} is analogous to the proof for the Boussinesq system, the differences being in the loss of the factor of $\cos \theta$ (which is handled by the angular support properties of $f,g,h$), and the different elliptic estimates from Lemma~\ref{lem:ellipticeuler}.
\end{proof}

We also describe the main elliptic lemma required. We have

\begin{lemma}[Elliptic estimates, Euler case]\label{lem:ellipticeuler} There exists $\alpha_0$ such that, for all $\alpha \in (0, \alpha_0)$, the following holds. There exist $R_0$, $\theta_0$ and $\xi_0 > 0$ such that the following holds true. Let $i \in \{0,1,3\}$. Assume $F\in\mathcal{H}^k_{,i}$, and suppose that the support of $F$ lies in set $\{x> -\xi_0/2\} \cap \{r \leq R_0\}$. Let $a= a(z,\theta, \tilde \lambda, \xi) := \frac{(\tilde \lambda z)^{\frac 1 \alpha}}{(\tilde \lambda z)^{\frac 1 \alpha} \cos \theta + \xi}$. Consider the following system:
\begin{align}
&(\alpha^2 D_z^2 + 4 \alpha D_z + (4 + \p_\theta^2))\Psi_B  = \mathbb{A}_0 (F), \\
&(\alpha^2 D_z^2 + 4 \alpha D_z + (4 + \p_\theta^2))\Psi_{E_1} =(\mathbb{A}_{- \xi} F - \mathbb{A}_0(F)),\\
& (\alpha^2 D_z^2 + 4 \alpha D_z + (4 + \p_\theta^2)) \Psi_{E_2} + a (\alpha \cos \theta D_z \Psi_{E_2} + 2\cos \theta \Psi_{E_2}  - \sin \theta \p_\theta \Psi_{E_2}) -  a^2 \Psi_{E_2}  \\
&\qquad =-  a(\alpha \cos \theta D_z (\Psi_B+\Psi_{E_1}) + 2\cos \theta (\Psi_B + \Psi_{E_1})  - \sin \theta \p_\theta ( \Psi_{B} + \Psi_{E_1})) + a^2( \Psi_{B} + \Psi_{E_1}).
\end{align}
If $\Psi_{E_1}, \Psi_{E_2}, \Psi_B$ solve the above system, letting $\Psi_{E} := \Psi_{E_1}+\Psi_{E_2}$, and $\mathring{\Psi}_E := \frac{1}{r^2} (\psi_E - y \psi(0,0) \p_y \chi)$, we have
\begin{align}
&|\mathbb{P}_2^\perp \partial_{\theta\theta}\Psi_B|_{{\mathcal{H}}^k_{,i}}+|\PP_2^\perp \partial_{\theta}\Psi_B|_{{\mathcal{H}}^k_{,i}}+\alpha|D_z\partial_\theta\Psi_B|_{{\mathcal{H}}^k_{,i}}+\alpha^2|D_z^2\Psi_B|_{{\mathcal{H}}^k_{,i}} \leq C|F|_{\mathcal{H}^k_{,i}},\\
&|\mathbb{P}_2^\perp \partial_{\theta\theta}\mathring{\Psi}_E|_{\underline{\mathcal{H}}^k_{,i}}+|\PP_2^\perp \partial_{\theta}\mathring{\Psi}_E|_{\underline{\mathcal{H}}^k_{,i}}+\alpha|D_z\partial_\theta\mathring{\Psi}_E|_{\underline{\mathcal{H}}^k_{,i}}+\alpha^2|D_z^2\mathring{\Psi}_E|_{\underline{\mathcal{H}}^k_{,i}}  \leq C (|\mathbb{E}(F)|_{\mathcal{H}^k_{,i}} + \alpha_0 |F|_{\mathcal{H}^k_{,i}} ),\label{eq:improvederror}
\end{align}
for a constant $C$ depending on $k, \alpha_0, \xi_0, R_0$ and $\theta_0$. Here, $\PP_2$ is the projection onto the second Fourier shell in $\theta$ (including both the $\cos$ and the $\sin$ terms\footnote{Note that if $\eta$ is a multiple of $\sin(2\theta)$, $\mathbb{P}_2 \eta$ reduces to $\mathbb{P}_1 \eta$.}), and $\mathbb{E}$ denotes even (in $x$) symmetrization about $x = 0$.
\end{lemma}

\begin{remark}
    Note that the quantities $\Psi_B$, $\Psi_{E_1}$, $\Psi_{E_2}$ are obtained from $\psi_B$ and $\psi_E$ by setting: $\psi_B = \tilde \mu r^2 \psi_B$, $\Psi_{E_j} = \tilde \mu \psi_{E_j}$, when $j \in \{1,2\}$.
\end{remark}

\begin{proof}[Proof of Lemma~\ref{lem:ellipticeuler}]
The proof follows in a similar way as the proof of Theorem~\ref{thm:elliptic}. The notable difference is that we have to estimate the contribution on the first Fourier shell. This is done by projecting the equation onto $\sin(\theta)$ and solving the resulting second order ODE by the variation of constants formula.
\end{proof}

\subsection{A-priori estimates on the velocity and the supports} \label{sec:suppmodul}
The following lemma will be used to control the radial extent of the support of $\omega$ and $\rho$ in the evolution, and the angular extent of said support outside of a neighborhood of the point $r = 0$ in our modulated coordinates. There, the angular extent of the support of $\omega$ and $\rho$ will be controlled by a direct Lagrangian analysis. We have the following Lemma, whose proof follows from elliptic regularity. 
\begin{lemma}
Suppose that $\psi$ satisfies 
$$
\Delta \psi + x^{-1} \p_x\psi - x^{-2} \psi = \omega.
$$
with zero Dirichlet boundary conditions at $x = 0$, and $\text{supp}( \omega) \subset B_{2R_0}(\bar\xi, 0)$ (note that in this case, the coordinate $x$ the original, un-modulated, $x$-coordinate), with $\bar \xi > 8 R_0$. Suppose also that $\psi$ vanishes at the point $(\bar \xi, 0)$. Then, we have the following estimate
\begin{align}
 |u|_{L^\infty(\R^2)} \leq C |\omega|_{L^{4}_{x,y}}.
\end{align}
In particular, $|\p_y \psi(\bar \xi,0)| \leq C |\omega|_{L^{4}_{x,y}}$.
\end{lemma}

\begin{proof}
The claims follow by elliptic regularity.
\end{proof}

\subsection{Concluding the argument}
This part follows closely the non-linear arguments in the Boussinesq part. We describe how to choose the control parameters. We have to check that the support properties used to apply the elliptic estimates and the nonlinear estimates are satisfied a priori. The key a-priori estimate in this case is an estimate of $\omega$ in $L^4_{x,y}$, which follows directly by an estimate of the type
$$
|\omega(t,x,y)| \leq C_{\omega_0} r^{-\alpha},
$$
where $C_{\omega_0}$ is a positive constant which does not depend on $s$, together with a bound on the extent of the support of $\omega$. This estimate can be proved by Lagrangian analysis in our framework. Once we have established this a-priori estimate, we are first going to cut-off the profile (radially) at $r = R_0/2$, and in the angle at $\theta_0$ (this ensures that our initial condition for the perturbation as well as the time-independent forcing terms are perturbative).

Then, the cut-off function $\chi_p$ is going to be set to be supported on $r \leq R_0$, and we choose $\xi_0 > \exp(R_0)$. This ensures that the Euler error terms in the Biot-Savart law are under control. Finally, we choose $\tilde{\lambda}(s=0) < C^{-1}_{\omega_0} R_0^{-4}$, so that the support of $\omega$ and $\rho$ cannot be transported more than $C \tilde{\lambda}(s=0) \cdot R^{1-\alpha}_0 \leq R_0$.

Finally, we have to show that the Lagrangian trajectories preserve (up to a change in $\theta_0$) the regions $A_{\theta_0}$. This is achieved by a direct Lagrangian analysis of the self-similar equations.

\section{Some closing remarks}

In the course of proving Theorem \ref{thm:boussinesq}, we have provided a new scenario in which a singularity can be constructed in the Boussinesq and Euler systems that is fundamentally smoother than the one constructed in \cite{E_Classical}. The mechanism builds off of well-known hydrodynamic instabilities to create a non-linear feedback loop for the growth of vorticity. The limitation on the regularity in the present paper might be an artifact of our proof; the low regularity in the radial variable does not play any role in breaking the geometry of the problem (as compared to \cite{E_Classical}). Further studies need to be done to determine whether a singularity can truly occur for localized smooth solutions to the Boussinesq equation on $\mathbb{R}^2$ or for solutions to the Euler equation in free-space; there does not appear to be any conceptual obstruction to this. Allowing for solutions that are smooth in the angle at the blow-up point also opens the door to similar results for other 2d active scalar equations, like the SQG system.

\section*{Acknowledgements}
The authors acknowledge funding from the NSF DMS-2043024 and the Alfred P. Sloan foundation. T.M.E. also acknowledges support from a Simons Fellowship.  

\appendix

\addtocontents{toc}{\protect\setcounter{tocdepth}{0}}

\section{\texorpdfstring{Equivalence between $|\cdot|_{\calH^k}$ and $|\cdot|_{\tilde{\calH}^k}$}{Equivalence between H and tilde H norms}}\label{app:spaces}

In this Section, we are going to show that the two inner products in Definition~\ref{def:allinner} induce equivalent norms.

\begin{lemma}\label{lem:equivalence}
There exist positive constants $C_1$ and $C_2$ such that, for any $f \in \calH^k \cap \tilde \calH^k$,
\begin{equation}
C_1 |f|_{\calH^k} \leq |f|_{\tilde{\calH}^k} \leq C_2|f|_{\calH^k}.
\end{equation}
\end{lemma}

\begin{proof}[Proof of Lemma~\ref{lem:equivalence}]
The inequality $|f|_{\tilde{\calH}^k} \geq C_1 |f|_{\calH^k}$ follows directly from our definition.

For the opposite inequality, we need to control the expression $(f,f)_{\mathbf{B}}$ (which appears in equation~\eqref{eq:binnerdef}) in terms of $|f|^2_{\calH^k}$. All the terms except the $B_1$ term are estimated directly. We have:
\begin{align}\label{eq:masterequiv}
    (\calB f, \calB f)_{\dbold} \leq C\int_{0}^\infty (f_2^2 +(\calB f, \calB f)_{\calH^{3}_\theta}) \frac{(1+z)^2}{z^2} dz
\end{align}
Recall the definition of $\calB f$ and the definition of the norm $|\cdot|_{\calH^{k}_\theta}$:
\begin{equation}
        \calB(f) = \Big(\cos^2 \theta \p_\theta f_1, \  f_2,  \ \cos^2 \theta \p_\theta \Big(\frac 12 \sin(2\theta) D_z f_2 + \cos^2 \theta \p_\theta f_3\Big)\Big), \qquad 
     |f|_{\mathcal{H}^k_\theta}^2=\sum_{j=0}^k |(\cos \theta)^{j-\frac{7}{4}}  \partial_\theta^j f|_{L_\theta^2}^2
\end{equation}
The claim now follows by interpolation.
\end{proof}

\section{Elliptic Estimates}

The purpose of this section is to provide a framework for establishing $L^2$-based estimates for solutions to the equation
\begin{equation}\label{BSLaw}
\alpha^2 D_z^2 \Psi +4\alpha D_z\Psi +\partial_{\theta\theta}\Psi +4\Psi = F,
\end{equation}
in weighted spaces of Sobolev type. Throughout this section, we will assume that both $F$ and $\Psi$ are periodic of period $2 \pi$ and odd symmetric with respect to $0$ and $\pi/2$. The equation is solved with zero Dirichlet boundary conditions at $z=0$ and at $z = \infty$.

\subsection{Estimates on the second order \texorpdfstring{$\theta$}{theta}-frequency shell}
Upon projecting to the second Fourier shell, we define the quantities $\Psi_2(z) := \int_0^{2 \pi} \Psi(z, \theta) \sin(2 \theta) d\theta$, and $F_2(z) := \int_0^{2 \pi} F(z, \theta) \sin(2 \theta) d\theta$. We obtain:
\begin{equation}\label{2ndOrderBSLaw}
\alpha^2D_z^2\Psi_2+4\alpha D_z\Psi_2=F_2,
\end{equation}
Let us first observe the basic $L^2$ estimate. 
\begin{lemma}\label{BasicEstimate1}
For smooth solutions to \eqref{2ndOrderBSLaw}, we have the a-priori estimate:
\[\alpha|D_z\Psi_2|_{L^2}+\alpha^2|D_z^2\Psi_2|_{L^2}\leq 4|F_2|_{L^2},\] so long as $0\leq \alpha\leq 2.$
\end{lemma}
\begin{remark}
By an induction argument, the same estimates hold replacing $\Psi_2$ with $\p_z^m \Psi_2$, and replacing $F_2$ with $\p^m_z F_2$, where $m \in \N$. 
\end{remark}
\begin{proof}
Testing first with $D_z\Psi_2,$ we get:
\[\frac{\alpha^2}{2}(\partial_z(D_z\Psi_2)^2,z)+4\alpha |D_z\Psi_2|_{L^2}^2=(F_2, D_z\Psi_2).\] It follows that if $2\alpha\geq\frac{\alpha^2}{2},$ 
\[\alpha|D_z\Psi|_{L^2}\leq \frac{1}{2}|F_2|_{L^2}.\] It then follows directly from the equation that
\[\alpha^2 |D_z^2\Psi_2|_{L^2}\leq 3|F_2|_{L^2}.\]
\end{proof}
\begin{remark}
Note that we always use this Lemma in estimating solutions to~\eqref{eq:biotsavartPsi}, in which the RHS always appears with an $\alpha$ factor in front.
\end{remark}
We observe the following corollary of the proof.
\begin{lemma}\label{BasicEstimate2}
Let $w:(0,\infty)\rightarrow (0,\infty)$ be smooth and such that $\lim_{z\rightarrow 0} z^3 w(z)=0.$ Assume also that there is a constant $C>0$ so that \[|\partial_z (z w(z))|\leq C_1 w(z).\] Then, there exists $\alpha_*$ depending only on $C_1$ and a constant $C_2$ so that smooth solutions to \eqref{2ndOrderBSLaw} satisfy\footnote{Recall that $|f|^2_{L^2_w}:= \int_0^\infty (f(z))^2 w(z) dz$.}:
\[\alpha|D_z\Psi_2|_{L^2_w}+\alpha^2|D_z^2\Psi_2|_{L^2_w}\leq C_2|F_2|_{L^2_w},\] whenever $0\leq \alpha\leq \alpha_*.$
\end{lemma}
\begin{remark}
For us, $w$ will be asymptotically equivalent to $\frac{1}{z^2}$ when $z\rightarrow 0$ and bounded when $z\rightarrow \infty.$
\end{remark}

\subsection{Estimates off the second order \texorpdfstring{$\theta$}{theta}-frequency shell}
For this part, we will consider solutions to \eqref{BSLaw} that are orthogonal in $L^2$ to $\exp(2i\theta)$ and odd with respect to $0$ and $\frac{\pi}{2}$ on $\mathbb{S}^1$. This means both $\Psi$ and $F$ can be expanded in $\sin$ series with even frequencies starting at $n=4.$ 
We again start with the basic $L^2$ estimate and recall the equation for convenience: 
\[\alpha^2 D_z^2 \Psi +4\alpha D_z\Psi +\partial_{\theta\theta}\Psi +4\Psi = F\]
\begin{lemma}\label{lem:ellipticfirst}
Smooth solutions to \eqref{BSLaw} satisfying the above symmetries and orthogonality properties satisfy:
\[|\partial_{\theta\theta}\Psi|_{L^2}+\alpha |D_z\Psi|_{L^2}+\alpha^2 |D_z^2\Psi|_{L^2}\leq C|F|_{L^2},\] where $C$ is independent of $\alpha\in [0,2].$
\end{lemma}
\begin{proof}
Testing with $\Psi$ and integrating by parts we get:
\[-|\partial_\theta\Psi|_{L^2}^2+(4-2\alpha+\frac{\alpha^2}{2})|\Psi|_{L^2}^2-\alpha^2 |D_z\Psi|_{L^2}^2=(F,\Psi).\] By the symmetries of $\Psi$ and the Fourier expansion, we have that
\[|\partial_\theta\Psi|_{L^2}^2\geq 16|\Psi|_{L^2}^2.\]
It follows that if $0\leq \alpha\leq 2,$ we have that 
\[|\partial_\theta\Psi|_{L^2}\leq 4|F|_{L^2}.\] We can similarly test with $\partial_{\theta\theta}\Psi$ and deduce:
\[|\partial_{\theta\theta}\Psi|_{L^2}\leq C|F|_{L^2}.\] The remaining estimates now follow directly from the equation (by throwing $\partial_{\theta\theta}\Psi$ and $4\Psi$ on the right hand side) as well as Lemma \ref{BasicEstimate1}. 
\end{proof}
We now move to discuss weighted estimates for \eqref{BSLaw}. Let us observe that, as a corollary to the proof of Lemma 1.4, we get that if $w$ is a weight (as in Lemma \ref{BasicEstimate2}) satisfying 
\[
\Big|\frac{d}{dz}(z w_z(z))\Big|\leq C_1 w_z(z),
\]
for $\alpha$ sufficiently small we have 
\begin{cor}
Let $w_{z}$ be as in Lemma \ref{BasicEstimate2} then, if $\alpha$ is sufficiently small depending on $C_1$, we have that solutions to \eqref{BSLaw} satisfy 
\[|\partial_{\theta\theta}\Psi|_{L^2_{w_z}}+\alpha |D_z\Psi|_{L^2_{w_z}}+\alpha^2 |D_z^2\Psi|_{L^2_{w_z}}\leq C|F|_{L^2_{w_z}},\] with $C$ independent of $\alpha$, once $\alpha$ is small enough.  
\end{cor}
We now want to add a weight in $\theta.$ Let us recall a sharp Hardy inequality established and used in \cite{E_Classical}.
Consider the weight 
\[w_{\theta}=\frac{1}{\cos(\theta)^{1/2}}.\]
Then, we have that 
\begin{lemma}\label{Hardy1}
There exists a universal constant $C>0$ so that for all $f\in C^\infty([0,\pi/2])$ with $f(\pi/2)=0,$ we have
\[\int_{0}^{\pi/2} \frac{f(\theta)^2}{\cos(\theta)^{5/2}}d\theta\leq \frac{2}{3} \int_{0}^{\pi/2}\frac{f'(\theta)^2}{\cos(\theta)^{1/2}}d\theta + C|f|_{H^1}^2.\]
\end{lemma}
The above lemma implies the following Corollary.
\begin{cor}
Let $w_{z}$ be as in Lemma \ref{BasicEstimate2} and $w_{\theta}=\frac{1}{\cos^{1/2}(\theta)}$ then, if $\alpha$ is sufficiently small  depending on $C_1$, we have that solutions to \eqref{BSLaw} satisfy 
\[|\partial_{\theta\theta}\Psi|_{L^2_{w}}+\alpha |D_z\Psi|_{L^2_{w}}+\alpha^2 |D_z^2\Psi|_{L^2_{w}}\leq C|F|_{L^2_{w}},\] 
where \[w(z,\theta)=w_z(z) w_\theta(\theta),\]
and $C$ is independent of $\alpha$, once $\alpha$ is small enough.  
\end{cor}

\subsection{Higher Order Estimates}

To discuss higher order estimates, we need only investigate how the derivatives we use to construct $\mathcal{H}^k$ commute with \eqref{BSLaw}. The first case is the $\partial_{z}$ derivative. If $j\in\mathbb{N}$ is fixed, we see that 
\begin{equation}\label{Commutator1}\partial_{z}^{j}D_z=j\partial_{z}^j+D_z\partial_z^{j}.
\end{equation}
Consequently,
\begin{equation}\label{Commutator2}
\partial_{z}^j D_z^2=  j\partial_{z}^j D_z+D_z\partial_{z}^j D_z=j^2\partial_{z}^j +jD_z\partial_z^j+D_z^2\partial_z^j.
\end{equation}
We next move to discuss the commutation of angular derivatives $\cos(\theta)\partial_\theta$ with the equation \eqref{BSLaw}. It suffices to consider commutation with $\partial_{\theta\theta}$. Since this is slightly more messy than the preceding case, but it can be readily checked by induction that we may write for any $l\in\mathbb{N}$:
\begin{equation}\label{Commutator3}(\cos(\theta)\partial_\theta)^l\partial_{\theta\theta}-\partial_{\theta\theta}(\cos(\theta)\partial_\theta)^l)=\sum_{m=0}^{l-1} F_{m,l}(\theta)(\cos(\theta)\partial_\theta)^m\partial_{\theta\theta},
\end{equation}
where $F_{m,l}$ are trigonometric polynomials.

Let us mention in passing a useful fact. 

\begin{lemma}\label{Hardy2}
There is a universal constant $C>0$ with the property that for every $f\in C^\infty([0,\pi/2)) \cap C^\beta([0, \pi/2]),$ with $\beta > 0$, with $f(\pi/2) = 0$, we have that 
\[|f|_{L^2([0, \pi/2])}\leq  4|\cos(\theta)\partial_\theta f|_{L^2([0, \pi/2])}.\]
\end{lemma}
\begin{proof}
Observe that 
\[\int_0^{\frac \pi 2} \cos(\theta)\sin(\theta)\partial_\theta f f d\theta=\frac{1}{2}\int_0^{\frac \pi 2} f^2\sin^2(\theta)d\theta-\frac{1}{2}\int_0^{\frac \pi 2} f^2\cos^2(\theta)d\theta\]
and
\[-\int_0^{\frac \pi 2} \sin(\theta)\partial_\theta f f d\theta =\frac{1}{2}\int_0^{\frac \pi 2} \cos(\theta) f^2 d\theta \geq \frac{1}{2}\int_0^{\frac \pi 2} \cos^2(\theta)f^2d\theta.\]
It follows that 
\[4|f|_{L^2([0, \pi/2])}|\cos(\theta)\partial_\theta f|_{L^2([0, \pi/2])}\geq |f|_{L^2([0, \pi/2])}^2.\]
\end{proof}
The preceding Lemma, as well as the commutation properties \eqref{Commutator1},\eqref{Commutator2},\eqref{Commutator3}, smallness in $\alpha$, and the fact that $((\cos(\theta)\partial_\theta)^l \Psi)|_{\theta = \frac \pi 2}=0,$ yield the following corollary.
\begin{cor}
Let $w_{z}$ be as in Lemma \ref{BasicEstimate2} and $w_{\theta}=\frac{1}{\cos^{1/2}(\theta)}$ then, if $\alpha$ is sufficiently small depending on $C_1$, we have that solutions to \eqref{BSLaw} satisfy 
\begin{align}
&|\partial_{\theta\theta}D\Psi|_{L^2_{w}}+\alpha |D_zD\Psi|_{L^2_{w}}+\alpha^2 |D_z^2D\Psi|_{L^2_{w}}\leq C|DF|_{L^2_{w}}, \qquad \text{and}\\
&|\partial_{\theta\theta}D\Psi|_{L^2_{w}}+\alpha |D_zD\Psi|_{L^2_{w}}+\alpha^2 |D_z^2D\Psi|_{L^2_{w}}\leq C\left|D\left(\frac{F}{z}\right)\right|_{L^2_{\tilde{w}}}, 
\end{align}
where \[w(z,\theta)=w_z(z) w_\theta(\theta), \qquad \text{and} \qquad \tilde w = z^2 w,\] with $D=\partial_z^j(\cos(\theta)\partial_\theta)^l$. Here, 
$C$ is independent of $\alpha$, once $\alpha$ is small enough depending on $C_1,$ $j\geq 0,$ and $l\geq 0$.  
\end{cor}
Commuting a single $\partial_\theta$ with the equation allows us to improve the weight in $\theta$ from $(\cos \theta)^{-1/2}$ to $(\cos \theta)^{-5/2}$ once $l\geq 1,$ which is the content of our main theorem.

\begin{theorem}\label{thm:elliptic}
Let $i \in \{0,1,3\}$. Assume $F\in\mathcal{H}^k_{,i}.$ Then, there exists $\alpha_0 >0$ such that, for $\alpha < \alpha_0$, if $\Psi$ solves \eqref{BSLaw}, we have that 
\[
|\mathbb{P}_2^\perp \partial_{\theta\theta}\Psi|_{\mathcal{H}^k_{,i}}+|\PP_2^\perp \partial_{\theta}\Psi|_{\mathcal{H}^k_{,i}}+\alpha|D_z\partial_\theta\Psi|_{\mathcal{H}^k_{,i}}+\alpha^2|D_z^2\Psi|_{\mathcal{H}^k_{,i}}\leq C_{k}|F|_{\mathcal{H}^k_{,i}},
\] for some constant $C_k$. Here, $\PP_2$ is the projection onto the second Fourier shell in $\theta$ (including both the $\cos$ and the $\sin$ terms).
\end{theorem}

\begin{remark}
Note in particular that, if $\alpha^2 D_z^2 \Psi_1 +4\alpha D_z\Psi_1 +\partial_{\theta\theta}\Psi_1 +4\Psi_1 = \PP^\perp f_1$, then
\begin{equation}
    |2 \tan \theta \Psi_1 + \p_\theta \Psi_1|_{\calH^k_{,i}} \leq C |\PP_2^\perp f_1|_{\calH^k_{,i}}.
\end{equation}
\end{remark}

\section{Lemmas on nonlinear estimates}\label{sec:spaces2}
In this Section, we collect a few facts useful to show the estimates on the nonlinear terms in Section~\ref{sec:construct}. 

\begin{remark} For all $i \in \{0,1,2,3\}$,
if $F\in\mathcal{H}_{,i}^k$, then $\cos(\theta)\partial_\theta F,\ (1+z)\partial_z F\in \mathcal{H}_{,i}^{k-1}.$
\end{remark}

Let us start with the $L^\infty$ embedding. 

\begin{lemma}[$L^\infty$ Embedding]
Assume that $F\in\mathcal{H}_{,i}^2$, with $i\in \{0,1,2,3\}$. Then, $F\in L^\infty.$ In particular,  
\[|F|_{L^\infty}+\Big|\frac{F}{z}\Big|_{L^\infty}\leq C |F|_{\mathcal{H}_{,i}^2},\] for some universal constant $C.$
\end{lemma}

\begin{proof}
Note that $C_c^\infty$ is dense in $\mathcal{H}^k_{,i}$, so it suffices to prove the estimate for such $F.$ We have that
\[
|F(z,\theta)|=\Big|\int_0^z\int_0^\theta \partial_{z\theta} F(z',\theta') d\theta' dz' + \int_0^\theta \p_\theta F(0,\theta') d\theta'+ F(0,0)\Big|
\leq C|F|_{\mathcal{H}_{,i}^2}. \]
A similar reasoning holds for $F/z$.
\end{proof}

This directly translates to the following embedding of $\calH^k_{,i}$ spaces, for $i \in \{0,1,2,3\}$:

\begin{lemma}[$L^\infty$ embedding for $\calH^k_{,i}$]\label{lem:sobolev}
Let $k \geq 3$, and $i \in \{0,1,2,3\}$. There is a constant $C>0$ such that the following inequality holds for all smooth $F$ which vanish linearly at $z = 0$:
\begin{equation}
    |\Lambda^{j_1} \p_z^{j_2}(F/z)|_{L^\infty} \leq C|F|_{\calH_{,i}^k},
\end{equation}
whenever $j_1 + j_2 \leq k-4$.
\end{lemma}

As a consequence, we have that $\mathcal{H}^k_{,i}$ is an algebra, as we will now proceed to show. 
Recalling Lemma \ref{Hardy2}, 
we note that in computing $\mathcal{H}^k_{,i},$ we are free to use the derivatives 
$(\sqrt{1+\gamma^2}\partial_\gamma)^{k_2}$ with the given weights or to use the derivatives 
$\partial_{\theta}^{k_2}$ with an extra factor of $\cos^{2k_2}(\theta)$ in the weight.

\begin{lemma}[Algebra Property]\label{lem:alg0}
Assume $k\geq 4$. Let $F,G,H$ be smooth functions which vanish linearly at $z = 0$, with $H$ purely radial. Then,
\begin{align}
    &|F G|_{\calH^k_{,0}} \leq C_k |F|_{\calH^k_{,0}}|G|_{\calH^{k+1}_{,1}}, \\
    &|H \sin(2\theta) D_z G|_{\mathcal{H}^k_{,1}}\leq C_k|H|_{\mathcal{H}_{,2}^k}|G|_{\mathcal{H}_{,0}^k}, 
    &|H \cos^2 \theta \p_\theta G|_{\mathcal{H}^k_{,1}}\leq C_k|H|_{\mathcal{H}_{,2}^k}|G|_{\mathcal{H}_{,3}^k}, \label{eq:gradrho}\\
    &|\PP_0 (F G)|_{\mathcal{H}^k_{,2}}\leq C_k|F|_{\mathcal{H}_{,1}^k}|G|_{\mathcal{H}_{,0}^k}, & |\PP_0(F D_\theta G)|_{\calH^k_{,2}} \leq C |F|_{\calH^k_{,1}} |G|_{\calH^k_{,3}},\\
   & |\PP_0(\cos(2\theta) F D_z G)|_{\calH^k_{,2}} \leq C |F|_{\calH^k_{,1}} |G|_{\calH^k_{,0}},\\
    &|\PP^\perp_0 (F G)|_{\mathcal{H}^k_{,3}}\leq C_k|F|_{\mathcal{H}_{,1}^k}|G|_{\mathcal{H}_{,0}^k}.
\end{align}
We also have, with $i \in \{0,1,2,3\}$,
\[
|FG|_{\mathcal{H}^k_{,i}}\leq C_k|F|_{\mathcal{H}_{,i}^k} |G|_{\mathcal{H}_{,i}^k}.
\]
\end{lemma}
\begin{remark}
Note importantly that the space $\calH^k_{,0}$ controls $k+1$ derivatives.
\end{remark}

\begin{remark}
Note that the crucial estimates for local existence are~\eqref{eq:gradrho}, which correspond to the $\p_y \rho$ term on the RHS of the vorticity equation.
\end{remark}

\begin{proof}
The Lemma follows from the Leibniz rule and Lemma~\ref{lem:sobolev}.
\end{proof}

We also have the following interpolation Lemma in terms of the $\calH^k_{,i}$ spaces.
\begin{lemma}[Interpolation Lemma]\label{lem:interpolation}
For all $F$ smooth scalar functions, and for all $i \in \{0,1,2,3\}$ and all $1 \leq j_1 \leq k$, there exists $C>0$ such that
\begin{equation}
    |F|_{\calH^{j_1}_{,i}} \leq C |F|_{\calH^{k}_{,i}}.
\end{equation}
\end{lemma}
We will now proceed to establish some scalar transport estimates. Note that these estimates apply to all of the scalar 
$\mathcal{H}^k_{,i}$ norms defined above.

\begin{lemma}[Angular Transport Estimates] \label{AngularTransportEstimate} 
Assume $k\geq 4$, and $i \in \{0,1,3\}$.
Let $F\in\mathcal{H}^k_{,i}$ and assume that $u,\frac{u}{\cos(\theta)}\in\mathcal{H}_{,i}^k$. Then,\footnote{Note that, as always, we can formally define this quantity even though $u\partial_\theta f$ is not actually an $\mathcal{H}_{,i}^k$ function.} 
\[(F,u\partial_\theta F)_{\mathcal{H}_{,i}^k}\leq C_k\Big|\frac{u}{\cos(\theta)}\Big|_{\mathcal{H}_{,i}^k}|F|_{\mathcal{H}_{,i}^k}^2.\]
\end{lemma}

\begin{proof}
Note that the assumptions imply that $\partial_\theta u\in\mathcal{H}_{,i}^{k-1},$ which implies that $\partial_\theta u\in L^\infty$ and $\frac{u}{\cos(\theta)}\in L^\infty,$
since $k\geq 4.$ Now, call $U=\frac{u}{\cos(\theta)}.$
Let us consider $D^\beta:= \Upsilon^j(\cos(\theta)\partial_\theta)^l$ be any choice of $|\beta|=j+l$ derivatives in $z$ and $\theta$, with $|\beta| = k+1$ if $i = 0$, $|\beta| =k$ if $i = 1$, and $|\beta|= k+1$ with $j \leq k$ if $i = 3$. Let $\mathcal{D}:=(0,\infty) \times (0,\pi/2)$ be the domain of integration. We have
\begin{align*}
&\int_{\mathcal{D}} D^\beta F D^\beta (U\cos(\theta)\partial_\theta F) W_{i}(z,\theta) dz d\theta\\
&\qquad =\int_{\mathcal{D}} D^\beta F\sum_{\substack{\beta_1 + \beta_2= \beta\\ \beta_2 \neq \beta}} (D^{\beta_1}U) D^{\beta_2}(\cos(\theta)\partial_\theta F)W_{i}(z,\theta) dz d\theta + \int_{\mathcal{D}} u D^\beta F \partial_\theta D^\beta F W_{i}(z,\theta) dz d\theta.
\end{align*}
Here, $W_i(\theta, z)$ is the weight which appears at the top order in the definition of $\calH^k_{,i}$. The first term in the display above is estimated using the Cauchy-Schwarz inequality and the $L^\infty$ embedding. For the second one, we integrate by parts in $\p_\theta$. We conclude using the following property of the weight $W_i$: $|\cos(\theta)\partial_\theta W_i|\leq CW_{i}.$
\end{proof}
Next, we have the radial transport estimates, which are simpler because the weights in $z$ are smooth functions of $z$.
\begin{lemma}[Radial Transport Estimates]\label{RadialTransportEstimate} 
Assume $k\geq 4$ and let $i \in \{0,1,2,3\}$. Let $F\in\mathcal{H}_{,i}^k$ and assume that $v\in\mathcal{H}_{,i}^k.$ Then, we have that
\[
(F,v\partial_z F)_{\mathcal{H}_{,i}^k}\leq C_k \big(|v|_{\mathcal{H}_{,i}^k}+\Big|\frac{v}{z}\Big|_{\mathcal{H}_{,i}^k}\big)|F|_{\mathcal{H}_{,i}^k}^2.
\]
\end{lemma}
\begin{proof}
The proof is similar to the proof of Lemma \eqref{AngularTransportEstimate}. The only difference is that this time we have that $|z\partial_z W_{\beta}|\leq C_{k} W_{\beta}$ pointwise, while $|\frac{v}{z}|_{L^\infty}\leq C_k |v|_{\mathcal{H}_{,i}^k}.$ 
\end{proof}

\section{Notation and useful calculations}\label{sec:horrible}

For the reader's convenience, we collect the notation used in the paper, as well as the full expression for some of the objects considered in our analysis.

Throughout this section, $\eta(\theta)$ is a regular function of $\theta$ only ($2\pi$ periodic and even at $k\pi/2$ for all $k \in \mathbb{Z}$), $F(z,\theta)$ is a regular function of $z$, $\theta$ and $f = (f_1, f_2, f_3)$ where $f_1$ and $f_3$ are regular functions of $z,\theta$ and $f_2$ is a regular function of $z$ only. Moreover, let $\gamma = \tan \theta$.

\subsection*{Differentiation and integration}

\begin{itemize}
    \item $D_\theta \eta := \sin(2\theta) \p_\theta \eta = 2 \gamma \p_\gamma \eta$.
    \item $D_z F := z \p_z F$.
    \item $\dashint \eta :=\frac 2 \pi \int_0^{\frac \pi 2} \eta(\theta) d\theta = \frac 2\pi \int_0^\infty \eta(\gamma) \frac{d\gamma}{1+\gamma^2}$.
    \item $L_{12}(F) := \frac 1 \pi \int_z^\infty \int_0^{2 \pi} \frac{F(s,\theta)}{s} \sin(2\theta) d\theta ds$.
    \item $\overline{L}_{12}(F) := \frac 1 \pi \int_0^\infty \int_0^{2 \pi} \frac{F(s,\theta)}{s} \sin(2\theta) d\theta ds.$
    \item $L(F) := 2 \int_z^\infty \frac 1 {z'} \dashint F(z', \theta) dz'$.
\end{itemize}

\subsection*{Spaces}

\begin{itemize}
    
    \item Angular spaces.
    \begin{itemize}
        \item The angular inner product $(\cdot, \cdot)_{\mathbf{C}}$ makes $L_\theta$ coercive (when restricted to the subspace $\calS$).
        \item $|\eta|_{\calH^k_{\theta}}$ is the space in which we measure the angular regularity (note that the differential operators is $0$-homogeneous at $\pi/2$) (see equation~\eqref{eq:hknorm}).
        \item The inner product $(\cdot, \cdot)_{\mathbf{B}}$ makes $\scrloc$ coercive (when restricted to the subspace $\mathscr{S}$).
    \end{itemize}

    \item Vectorial spaces for functions of $(z,\theta)$.
    \begin{itemize}
        \item The inner products $(\cdot, \cdot)_{\tilde{\calH}^k}$ and $(\cdot, \cdot)_{{\calH}^k}$ are the main inner products on which we conduct our analysis. They are introduced in Section~\ref{sec:spaces}. Note that they induce equivalent norms $|\cdot|_{\tilde{\calH}^k}$ and $|\cdot|_{{\calH}^k}$, which require vanishing of the $f_2$ component at $z = 0$. 
        \item The inner product $(\cdot, \cdot)_{\mathscr{H}^k}$ is introduced in Definition~\ref{def:scripth}. It induces the norm $|\cdot|_{\scrH^k}$, which does not require vanishing of the second component $f_2$ at $z = 0$.
    \end{itemize}
\end{itemize}

\subsection*{Projections and special subspaces}

\begin{itemize}
    \item Angular projections and special subspaces.
    \begin{itemize}
        \item $\PP_0 F:= \frac 1{2\pi} \int_0^{2\pi} F(\theta) d \theta$.
        \item $\PP_1 F:=\sin (2\theta) \frac 1 {\pi} \int_0^{2\pi} F(\theta) \sin(2\theta) d\theta$.
        \item $\PP_2 F : = \sin (2\theta) \frac 1 {\pi} \int_0^{2\pi} F(\theta) \sin(2\theta) d\theta +\cos (2\theta) \frac 1 {\pi} \int_0^{2\pi} F(\theta) \cos(2\theta) d\theta$.
    \end{itemize}

    \begin{itemize}
        \item $\bar \eta := \eta(\theta) - \eta(0)$.
        \item $\calA(\eta) := \frac{1}{\tan^2\theta} \bar\eta$.
        \item $\calS := \Big\{ (\Gamma_1, \Gamma_2) \in D(L_\theta): \int_0^\infty (13 \bar \Gamma_1 + \bar \Gamma_2) \gamma^{-2} d \gamma = 0\Big\}$.
        \item $\PP_{\calS}$ is a projection onto $\calS$.
    \end{itemize}
    
    \item Projections on functions of $(z,\theta)$.
    \begin{itemize}
        \item $\mathscr{S}$ is the space which makes $\scrloc$ coercive (see Lemma~\ref{lem:scrlocp}.
        \item $\mathbb{Q} f$ is a projection onto the subspace $\mathscr{S}$.
    \end{itemize}
\end{itemize}

\subsection*{Linear operators}

\begin{itemize}
    \item $\calB$ is a purely angular operator which performs the transformation to the ``good'' unknowns in Lemma~\ref{lem:good}.
    \item $L_\theta$ is the main angular operator defined in Definition~\ref{eq:lthetadef}.
    \item $\scrl f := \calN(f, f_*) + \calN(f_*, f)$.
    \item $\overline{\scrl}$ is the nonlocal part of $\scrl$ defined in~\eqref{eq:defnonloc}.
    \item $\scrloc := \scrl - \scrnonloc$.
    \item $\underline{\mathcal{L}}f := \underline{\calM}(\calB f, \calB f_*)  + \underline{\calM}(\calB f_*, \calB f) + \calB \calW (f)$ is the local linear operator acting on the ``good'' unknowns (after the transformation $\calB$).
    \item $\mathfrak{L}f := f + z\p_z f + \scrl f$.
\end{itemize}

The full expression for $\scrloc $ in~\eqref{eq:scrlocdef} reads:
\begin{align}
    &(\scrloc (f))_1 = \frac 12 \lna(\boldom_0[f_*]) D_\theta \boldom_0[f] - \boldP_0[f_*](\sin(2\theta)D_z \boldP_0[f] + 2 \cos^2 \theta \p_\theta \boldP_1[f])\\
    & \qquad \qquad +  \frac 12 \lna(\boldom_0[f]) D_\theta \boldom_0[f_*] - \boldP_0[f](\sin(2\theta)D_z \boldP_0[f_*] + 2 \cos^2 \theta \p_\theta \boldP_1[f_*]) + \frac 12 \lfar(\boldom_0[f_*]) D_\theta \boldom_0[f]  , \\
   &(\scrloc (f))_2 = - \frac 1 4 \lna(\boldom_0[f_*]) \boldP_0[f]- \frac 1 4 \lna(\boldom_0[f]) \boldP_0[f_*]- \frac 1 4 \lfar(\boldom_0[f_*]) \boldP_0[f], \\
    &(\scrloc (f))_3 = \PP_0^\perp\Big\{\frac 12 \lna(\boldom_0[f_*]) D_\theta \boldP_1[f]   - \frac 12 \cos(2 \theta) \lna(\boldom_0[f_*]) D_z \boldP_0[f] -\frac 14 \lna(\boldom_0[f_*])\boldP_1[f] \\
    & \qquad \qquad \qquad - \frac{\boldP_0[f_*]}{4} \sin^2 \theta D_z(\lna(\boldom_0[f])) + \frac 12 \Big(2 \tan \theta \tilde{\boldpsi}_1[f_*] + \p_\theta \tilde{\boldpsi}_1[f_*] \Big)\boldP_0[f]\Big\}\\
    &\qquad +\PP_0^\perp\Big\{\frac 12 \lna(\boldom_0[f]) D_\theta \boldP_1[f_*]  - \frac 12 \cos(2 \theta) \lna(\boldom_0[f]) D_z \boldP_0[f_*] -\frac 14 \lna(\boldom_0[f])\boldP_1[f_*] \\
    & \qquad - \frac{\boldP_0[f]}{4} \sin^2 \theta D_z(\lna(\boldom_0[f_*])) + \frac 12 \Big(2 \tan \theta \tilde{\boldpsi}_1[f] + \p_\theta \tilde{\boldpsi}_1[f] \Big)\boldP_0[f_*]\\
    &\qquad +\frac 12 \lfar(\boldom_0[f_*]) D_\theta \boldP_1[f]  - \frac 12 \cos(2 \theta) \lfar(\boldom_0[f_*]) D_z \boldP_0[f] -\frac 14 \lfar(\boldom_0[f_*])\boldP_1[f]\Big\},
\end{align}

\subsection*{Nonlinear terms and error terms}
\begin{itemize}
    \item $\mathcal{K}(f,g)$ is a \emph{two} vector collecting all the nonlinear terms in the full self-similar system before the $\PP_0$ projection (see~\eqref{eq:K1}--\eqref{eq:K2}).
    \item $\calN(f,g)$ is a \emph{three} vector whose components are the main (order 1 in $\alpha$) nonlinear terms in the self-similar system, \emph{after} projection with $\PP_0$ on the second and third component ($P_0$ and $P_1$) (see~\eqref{eq:masterbou}).
    \item $E(f,g)$ are the error terms arising from isolating the terms $\calN(f,g)$ in the system~\eqref{eq:masterbou}.
    \item $\calM(f,g)$ is the bilinear form obtained from $\calN$ by applying the transformation $\calB$. $\underline{\calM}$ is the local version of $\calM$ (see lemmas~\ref{lem:good} and~\ref{lem:goodloc}).
    \item $\tilde{\calN}(f,g)$ is a variation of $\calN(f,g)$ which includes some of the terms in $E(f,g)$, so that the resulting decomposition does not suffer from derivative loss (see~\eqref{eq:ntildes}).
    \item $\tilde{E}(f,g)$ are the resulting error terms with no derivative loss (see~\eqref{eq:tilderr}).
\end{itemize}
The expression for these error terms in~\eqref{eq:tilderr} reads:
\begin{align}
& \tilde{E}(f,g)_1 := \p_\theta \Psi[g] D_z f_1 - D_z \Psi[g] \p_\theta f_1 - 2 \Psi_1[g] \p_\theta f_1 + g_3(\sin(2\theta)D_z(f_2 + \alpha f_3) + 2 \cos^2 \theta \p_\theta f_3),\\
&\tilde{E}(f,g)_2= \PP_0\Big( \alpha \p_\theta  \Psi_1[g] D_z (f_2 + \alpha f_3) -  \alpha D_z  \Psi[g] \p_\theta f_3   - 2 \alpha \Psi_1[g] \p_\theta f_3   \Big),\\
&\tilde{E}(f,g)_3= \PP^\perp_0 \Big( \p_\theta \Psi_1[g] D_z (f_2 + \alpha f_3)  -  D_z \Psi[g]  \p_\theta f_3 - 2 \Psi_1[g] \p_\theta f_3 \Big).
\end{align}

\bibliographystyle{plain}
\bibliography{bibliography.bib}

\end{document}